\documentclass[12pt,letterpaper,reqno]{amsart}
\usepackage[letterpaper,margin=1.2in,headheight=15pt]{geometry} 
\usepackage{graphicx,bbm}
\usepackage{amsmath, amssymb, amsfonts}
\usepackage[utf8]{inputenc}
\usepackage{epstopdf}
\usepackage{tikz-cd} 
\usepackage{setspace}
\usepackage{xcolor}
\usepackage{tcolorbox}
\usepackage[titletoc,title]{appendix}
\usepackage{enumitem}
\usepackage{soul}

\usepackage{hyperref}
\definecolor{darkred}{rgb}{0.5,0.15,0.15}
\hypersetup{colorlinks=true,urlcolor=darkred,linkcolor=darkred,citecolor=darkred}

\newcommand{\CY}{\mathcal{X}}

\newcommand{\C}{\mathbb{C}}
\newcommand{\Q}{\mathbb{Q}}
\newcommand{\Z}{\mathbb{Z}}
\newcommand{\R}{\mathbb{R}}

\newcommand{\VH}{\mathcal{H}}
\newcommand{\Oo}{\mathcal{O}}

\newcommand{\Fc}{F}

\newcommand{\cB}{\ensuremath{\mathcal B}}

\newcommand{\cu}{C} 
\newcommand{\cuo}{C^\circ}
\renewcommand{\ell}{X} 

\newcommand{\de}{\mathrm{d}}
\newcommand{\e}{{\mathrm e}}
\newcommand{\del}{{\partial}}
\newcommand{\delbar}{\ensuremath{\overline{\partial}}}

\newcommand{\ubar}{\ensuremath{{\overline{u}}}}

\renewcommand{\Im}{\mathrm{Im}\;}
\renewcommand{\hat}{\widehat}

\theoremstyle{plain}
\newtheorem{thm}{Theorem}[section]
\newtheorem{prop}[thm]{Proposition}
\newtheorem{lem}[thm]{Lemma}
\newtheorem{cor}[thm]{Corollary}

\theoremstyle{definition}
\newtheorem{dfn}[thm]{Definition}
\newtheorem{ex}[thm]{Example}

\newtheorem*{question*}{Question}

\theoremstyle{remark}
\newtheorem{rem}[thm]{Remark}

\numberwithin{equation}{section}

\title{Parabolic Higgs bundles, $tt^*$ connections and Opers}

\author{Murad Alim}
\author{Florian Beck} 
\address{Fachbereich Mathematik, Universit\"at Hamburg, Bundesstr. 55, 20146, Hamburg}
\author{Laura Fredrickson}
\address{Department of Mathematics, Stanford University, Stanford, CA 94305, USA}

\begin{document}

\maketitle

\begin{abstract}

The non-abelian Hodge correspondence identifies complex variations of Hodge structures with certain Higgs bundles. 
In this work we analyze this relationship, and some of its ramifications, when the variations of Hodge structures are determined by a (complete) one-dimensional family of compact Calabi--Yau manifolds.
This setup enables us to apply techniques from mirror symmetry.
For example, the corresponding Higgs bundles extend to parabolic Higgs bundles to the compactification of the base of the families.
We determine the parabolic degrees of the underlying parabolic bundles in terms of the exponents of the Picard--Fuchs equations obtained from the variations of Hodge structure. 

Moreover, we prove in this setup that the flat non-abelian Hodge or $tt^*$-connection is gauge equivalent to an oper which is determined by the corresponding Picard--Fuchs equations. 
This gauge equivalence puts forward a new derivation of non-linear differential relations between special functions on the moduli space which generalize Ramanujan's relations for the differential ring of quasi-modular forms.

\end{abstract}

\section{Introduction}\label{intro}

The geometry of gauge theories and string theories has inspired many fruitful interactions between mathematics and physics and has put forward new structures and relations within mathematics. In this work we focus on different families of flat $G$-connections, coming from two entirely different physical setups. 

On the one hand we will consider $G$-Higgs bundles which originated in \cite{Hit87a, Hit87b} and were further developed by many people in various directions, see e.g. \cite{Simpson:1992} for the extension to the parabolic setting and \cite{Donagi-Spectral} for the extension to $G$-Higgs bundles\footnote{See \cite{DalakovSurvey} for a survey and \cite{Anderson:2017} for some recent interactions and open problems.} .
The physical origin of $G$-Higgs bundles are the Hermitian Yang--Mills equations for a gauge theory with gauge group $G\subset GL(N,\C)$ in four real dimensions. 
The dimensional reduction of the latter to two real dimensions, so-called Hitchin's (self-duality) equations, admit a formulation on any Riemann surface $\cu$. 
These are equations for a pair $((E,\varphi),h)$ consisting of a $G$-Higgs bundle $(E,\varphi)$ and a Hermitian metric $h$ on $E$.
Recall that a $G$-Higgs bundle is a holomorphic vector bundle $E$ with a $G$-structure and a one-form $\varphi\in \Gamma(\cu,K_\cu\otimes \mathrm{ad}(E))$ with values in the adjoint bundle $\mathrm{ad}(E)$ associated to $E$. 
The Higgs field $\varphi$ encodes the gauge field data in the two reduced dimensions.
The pair $((E,\varphi),h)$ satisfies Hitchin's equations if 
\begin{equation}\label{sd-eqs}
F_{D(\bar{\partial}_E,h)} + \left[ \varphi, \varphi^{\dagger_h}\right]=0\,, \quad \bar{\partial}_{E} \varphi =0\,.
\end{equation}
Here $D(\delbar_E, h)$ is the Chern connection associated to the holomorphic structure $\delbar_E$ on $E$ and the Hermitian metric $h$. 
The equations \eqref{sd-eqs} are equivalent to the flatness of the $\C^\times$-family 
\begin{equation}
    \nabla_{\mathrm{NAH}}^\zeta=D(\bar{\partial}_{E},h)+\zeta^{-1} \varphi+\zeta \varphi^{\dagger_h},\quad \zeta\in \C^\times
\end{equation}
of $G$-connections on $E$.
Passing from $(E,\varphi)$ to $\nabla_{\mathrm{NAH}}^\zeta$ and vice versa is the content of the non-abelian Hodge correspondence\footnote{As is well-known, see \cite{Hit87a, Simpson:1992}, this correspondence is an equivalence between stable $G$-Higgs bundles and irreducible flat $G$-connections. Stability/Irreducibility does not play a crucial role in our work so that we suppress it here.}.
We therefore refer to the family $\nabla_{\mathrm{NAH}}^\zeta$ as the family of non-abelian Hodge connections associated to the $G$-Higgs bundle $(E,\varphi)$.
For our purposes, it is sufficient to consider $G=GL(N,\mathbb{C})$. 

On the other hand we consider the geometry emerging from the variation of Hodge structures (VHS) on the middle dimensional cohomology of certain families of Calabi--Yau (CY) manifolds which features prominently in the mirror symmetry of CY threefolds, see e.g. \cite{CoxKatz}. The flat holomorphic Gau\ss--Manin connection of the VHS
plays an important role in identifying the mirror isomorphism of CY threefolds with the help of the associated Picard--Fuchs equations. The flatness conditions constrains the base manifold to be a special K\"ahler manifold, see \cite{Strominger:1990pd,Freed1999} for definitions and \cite{LledoReview} for a review. 

A vast generalization of special K\"ahler geometry and the flat holomorphic Gau\ss--Manin connection is given by the notion of topological-anti-topological fusion or $tt^*$-geometry which was put forward by Cecotti and Vafa in \cite{Cecotti:1991me}.
The origin of $tt^*$-geometry is the study of families of two dimensional physical field theories with $\mathcal{N}=2$ supersymmetry. 
Some of these can be realized as non-linear sigma models into a target manifold $X$ which could be a CY manifold, or a Fano manifold. Others are only realizable as Landau-Ginzburg models, which are non-compact manifolds equipped with a complex-valued function called the superpotential.

In every case, the $tt^*$-geometry refers to the data of a bundle $E$, together with a Hermitian metric $h$ on $E$ and a symmetric complex pairing $\eta$ on $E$, over a moduli space $\mathcal{B}$ of theories.
Their variation is described by a flat connection $\nabla_{tt^*}$ which is constructed as a combination of a connection $D$ and a one-form $\mathsf{C}$ valued in the endomorphisms of the bundle $E$.
Here the connection $D$ is compatible with a Hermitian metric $h$ as well as with a complex pairing $\eta$. The endomorphism-valued one-form $\mathsf{C}$ is given by an underlying Frobenius manifold structure on $\mathcal{B}$ defined by the operators of the two-dimensional theories parameterized by $\mathcal{B}$.
The $tt^*$-geometry on $E$ over $\mathcal{B}$ is then governed by the $tt^*$-equations: 
\begin{gather}\label{tt-eqs}
[D',D'']+[\mathsf{C}',\mathsf{C}'']=0, \quad D'\mathsf{C}''=0=D''\mathsf{C}', 
\\ 
\label{eq:tt-eqs2} [D',D']=0=[D'', D''], \quad [\mathsf{C}',\mathsf{C}']=0=[\mathsf{C}'',\mathsf{C}''], \quad [D',\mathsf{C}']=0=[D'',\mathsf{C}''].
\end{gather}

Here $'$ and $''$ denote the $(1,0)$- and $(0,1)$-part respectively and the commutators are defined such that $[C',C']$ is a $(2,0)$-form etc.
Equations
\eqref{tt-eqs} and \eqref{eq:tt-eqs2} are equivalent to the flatness of the $\C^{\times}$-family of $tt^*$-connections
\begin{equation}
    \nabla_{tt^*}^\zeta=D+\zeta^{-1}\mathsf{C}'+\zeta \mathsf{C}'', \quad \zeta\in \C^\times,
\end{equation}
on $E$. Here \eqref{tt-eqs} is equivalent to  $\mathcal{F}^{(1,1)}_{\nabla_{tt^*}^\zeta}=0$, while \eqref{eq:tt-eqs2} is equivalent to $\mathcal{F}^{(0,2)}_{\nabla_{tt^*}^\zeta}=0=\mathcal{F}^{(2,0)}_{\nabla_{tt^*}^\zeta}$.

A differential-geometric approach to $tt^*$-equations was given in \cite{Dubrovin:1993}. A complex geometric framework was given for the $tt^*$-equations in \cite{Hertlingtt} in terms of $DC\tilde{C}$-structures (\cite[Definition 2.9]{Hertlingtt}) and TERP-structures (cf. \cite[Definition 2.12]{Hertlingtt}), see \cite{Hertling1,Hertling2} for an overview. 

Although Hitchin's equations and the $tt^*$-equations have different physical origins, they have a common mathematical structure. 
More precisely, solutions to these equations determine twisted (pluri-)harmonic maps from the base manifold to $GL(N,\C)/U(N)$ (cf. \cite[\S 9]{Hit87a}, \cite{Donaldson-twisted}, \cite[Theorem 3]{Dubrovin:1993}, \cite{CortesSchaefer})). In some cases, the non-abelian Hodge machinery can be used to produce solutions of the 
$tt^*$-equations.  In \cite{Guest1, Guest2, Guest3}, Guest, Its, and Lin
proves the existence existence of radially-symmetric solutions of the $tt^*$-equations when the base
manifold is $\mathbb{C}^\times$.  In \cite{Mochizuki2},Mochizuki efficiently proves the existence of the same radially-symmetric solutions using the Hitchin-Kobayashi correspondence, i.e. he describes the relevant wild Higgs bundles then produces a harmonic map using the Hitchin-Kobayashi correspondence.

We take a different perspective to compare the equations, namely we specialize the $tt^*$-equation to the following setup.
If $\mathcal{B}$ is a Riemann surface parametrizing compact Calabi--Yau manifolds, then solutions of the corresponding $tt^*$-equations are solutions of Hitchin's equations on $\mathcal{B}$ which are invariant (up to gauge transformations) under the $\C^\times$-action rescaling the Higgs field. On higher-dimensional K\"ahler manifolds $X$, the $tt^*$-equations are $\mathcal{F}_\nabla=0$ while Hitchin's equations are instead 
\begin{equation*}
 \mathcal{F}_\nabla^{2,0} = 0 \qquad \mathcal{F}_\nabla ^{0,2}=0 \qquad \Lambda \mathcal{F}_\nabla^{1,1}=0,
\end{equation*}
where $\Lambda$ denotes contraction with the symplectic form of $X$.
Thus, on a K\"ahler manifold, the $tt^*$-equations are a special case of Hitchin's equations. 

However, $tt^*$-equations are more general in other directions. 
Firstly, they can be formulated on any complex manifold.  Secondly, the physical theories which have a geometric realization in terms of a non-linear sigma model admit more physical deformations than just the geometric ones.  
All the additional deformations are governed by the $tt^*$-equations as well. A rigorous mathematical framework describing this general version of the $tt^*$-equations remains challenging. 

In this paper, we relate the two setups in the case where the $tt^*$-geometry is governed by a VHS $\mathcal{H}$ determined by the middle-dimensional cohomology of a family $\pi\colon\mathcal{X}\to \mathcal{B}$ of CY $d$-folds with $\dim_{\C}(\mathcal{B})=1$.  
This is, for example, the context of mirror symmetry for CY manifolds where the family $\pi\colon\mathcal{X}\to \mathcal{B}$ is assumed to be complete, i.e. locally isomorphic around each $u\in \mathcal{B}$ to the moduli space of complex structures on $X_u:=\pi^{-1}(u)$\footnote{In the following we adopt the terminology in physics and refer to a `moduli space' as a parameter space with some non-degeneracy assumption like completeness of the parameterized families of CY $d$-folds.}. 
In this context, we explain how various notions that originated on the $tt^*$-equation side naturally appear on the Higgs bundle side, and vice versa.  This goes far beyond comparing the equations \eqref{sd-eqs} and \eqref{tt-eqs} (cf. Remark \ref{rem:tthit}).  

For example, we show that opers, a distinguished family of flat connections which play an important role in the geometric Langlands correspondence (see \cite{beilinson2005opers} and the discussion of the conformal limit below), are equivalent to variations of Hodge structures with special sections. 
We call such a section a generic cyclic vector. 
Its derivatives with respect to the Gau\ss--Manin connection give a basis of the complex vector bundle $\mathcal{H} \rightarrow \cB$ on an open and dense subset of $\mathcal{B}$.
The existence of a generic cyclic vector is natural from the perspective of mirror symmetry: if the VHS $\mathcal{H}$ is induced by a complete family $\mathcal{X}\to \mathcal{B}$ of CY manifolds, then $\mathcal{H}$ admits a generic cyclic vector.
Hence mirror symmetry provides many concrete examples of opers. 

Parabolic Higgs bundles \cite{SimpsonNC}, \cite{Yokogawa}, a generalization of Higgs bundles, naturally arise in the $tt^*$-geometry under consideration. 
Typically, the base $\mathcal{B}$ is a non-compact Riemann surface and the family $\pi\colon \mathcal{X}\to \mathcal{B}$ extends to a family $\hat{\pi}\colon \hat{\mathcal{X}}\to \hat{\mathcal{B}}$ over the compactification $\hat{\mathcal{B}}$ of $\mathcal{B}$.
The fibers over $D:=\hat{\mathcal{B}}-\mathcal{B}$ are usually singular and each $d\in D$ is a regular singular point of the holomorphic Gau\ss--Manin connection $\nabla$ of $\mathcal{H}$.
Equivalently, the Picard--Fuchs equations, which are locally determined by a (generic) cyclic vector, are regular singular at $d\in D$.

The qualitative behaviour of the (multi-)valued solutions to the Picard--Fuchs equations at $d\in D$ is determined by their so-called exponents associated to $d$.
After reviewing how Deligne's canonical extension of $(\mathcal{H},\nabla)$ (\cite{Deligne}, \cite{Schmid}) determines a parabolic Higgs bundle $(\hat{E},\hat{\varphi})$ on $\hat{\mathcal{B}}$ with divisor $D$ (see \cite{SimpsonUbiquity} for $D=\emptyset$ and \cite[Chapter 13]{CarlsonEtAl2} for an arbitrary reduced divisor $D$), we compute the (parabolic) degrees of $(\hat{E},\hat{\varphi})$ in terms of the exponents at each $d\in D$.
In the typical context of mirror symmetry, where $D=\{0,1,\infty\}\subset \hat{\mathcal{B}}=\mathbb{CP}^1$, this is sufficient to explicitly determine the corresponding parabolic Higgs bundles.
 
We point out that the first computation of the parabolic degrees in terms of the exponents appeared in \cite[\S 6]{EskinEtAl} for $D=\{0,1,\infty\}\subset \hat{\mathcal{B}}=\mathbbm{CP}^1$ and certain VHS $\mathcal{H}$ of weight $3$ with $\mathrm{rk}(\mathcal{H})=4$ in the context of Lyapunov exponents. 
This was taken up by \cite{DoranEtAlVHS} for the same base curve but more general VHS in order to compute Hodge numbers of certain compact K\"ahler manifolds which are fibered over $\mathbb{CP}^1$.
Our approach neither requires a restriction on $\hat{\mathcal{B}}$ nor on $\mathcal{H}$.
The previous papers focus on the underlying parabolic bundles whereas we focus on the parabolic Higgs bundles. 

Another motivation for our work comes from Gaiotto's conjecture (\cite[\S 4.2]{Gaiotto:2014bza}).
It relates the family $\nabla_{\mathrm{NAH}}^\zeta$ of non-abelian Hodge connections associated to a Higgs bundle $(E,\varphi)$ on $\mathcal{B}$ to a family of opers, associated to the same Higgs bundle, through the so-called conformal limit.
This limit is obtained by introducing an additional real parameter $R$ in the family $\nabla_{\mathrm{NAH}}^\zeta$ and taking an appropriate double scaling limit of $\nabla_{\mathrm{NAH}}^{\zeta,R}$ as $R,\zeta$ approach $0$, while holding their ratio $\lambda = \frac{\zeta}{R}$ fixed. 
This conjecture was proven if $\mathcal{B}$ is a compact Riemann surface in \cite{Dumitrescu2016a, CollierWentworth}. Understanding the precise relation between the two families of flat connections is, in particular, important for understanding the relation between the exact WKB-analysis\footnote{See \cite{Iwaki} and references therein.} and the works of Gaiotto, Moore and Neitzke \cite{Gaiotto:2009hg}.

All of the parabolic Higgs bundles  $\mathcal{B}=\mathbb{CP}^1-\{0,1,\infty\}$ coming from the $tt^*$-equations are fixed by the natural $\C^\times$-action $\varphi \to \zeta \varphi$ (up to isomorphism).  It follows that the two-parameter family $\nabla_{\mathrm{NAH}}^{R \lambda, R}$ is independent of $R$, hence the conformal limit $\lim_{R \to 0} \nabla_{\mathrm{NAH}}^{R \lambda, R}$ 
is simply equal to $\nabla_{\mathrm{NAH}}^\lambda$.  While the conformal limit is a trivial process, the change of frame identifying $\nabla_{\mathrm{NAH}}^{\lambda}$ with the Gau\ss-Manin connection $\nabla_{\mathrm{GM}}^{\hbar}$ contains highly non-trivial information: 
it gives a new way to derive certain differential rings on the moduli space $\mathcal{B}$, see \S\ref{sec:gauge}.  In our cases, 
the Hermitian metric $h$ solving Hitchin's equations and the $tt^*$-equations is expressed in terms of special functions $g_i$ on $\mathcal{B}$.
These differential rings are rings with a differential $D$ which are generated by finitely many special functions $g_i$, $i\in I$, on $\mathcal{B}$. 
In fact, they are algebraic: the derivatives  $Dg_i$, $i\in I$, are polynomials in the $g_j$, $j\in I$, with holomorphic functions on $\mathcal{B}$ as coefficients.
These relations are called differential ring relations and give nonlinear differential equations for the $g_i$. 

If $\mathcal{B}$ is a moduli space for elliptic curves, the differential ring relations are directly related to the differential ring of quasi-modular forms developed by Kaneko and Zagier \cite{Kaneko} as shown in \cite{Hosono,Mov12c} (see Appendix \ref{app:differential_ring} for a short summary).
If $\mathcal{B}$ is a moduli space of CY threefolds, then the differential ring relations have strong implications for mirror symmetry: 
in \cite{YY}, Yamaguchi and Yau discovered differential ring relations for the moduli space of the mirror quintic CY threefold. 
In these cases, the special functions on the moduli space $\mathcal{B}$ are given by the generating functions for higher genus Gromov--Witten invariants on the quintic CY threefold (under mirror symmetry).
Hence the differential ring relations show that all higher genus Gromov--Witten invariants are determined by finitely many data.
These results were later generalized to arbitrary mirror pairs of CY threefolds in \cite{Alim1} and developed for lattice polarized K3 surfaces in \cite{Alim:2014vea}.

Besides mirror symmetry, such differential rings are interesting in their own right because they provide an analogue for CY threefolds of quasi-modular forms for elliptic curves\footnote{See e.g. \cite{MovasatiReview} and references therein}.
By providing the explicit link between these developments and Higgs bundles, we hope to further open the door towards new insights and generalizations of these structures.

The plan of our paper is the following. In \S\ref{explicit} we start with an explicit example of a VHS which appears in the context of mirror symmetry for lattice polarized K3 surfaces with complex one-dimensional moduli spaces. This example will allow us to elucidate the ingredients and the constructions relevant for our paper. 
We proceed in \S\ref{s:VHStt} to motivate VHS and to show their precise connection to opers.
In \S\ref{s:VHSHiggs} we first review the relation between Deligne's canonical extensions of VHS and parabolic Higgs bundles. 
Afterwards we compute the (parabolic) degrees of the resulting parabolic Higgs bundles in terms of the associated exponents. 
Finally, in \S\ref{examples} we give further examples motivated from mirror symmetry and review the necessary constructions, we furthermore discuss the general structure of the gauge transformation from the family of non--abelian Hodge flat connections to the family of opers.
In the appendices we collect basic definitions and notions concerning quasi-modular forms, differentials rings as well as parabolic Higgs bundles. 

In our beginning example in \S2 and throughout the paper, we provide a number of explicit and detailed examples in order to facilitate building a bridge between the geometry of Higgs bundles on the one side and $tt^*$-geometry, VHS and mirror symmetry on the other side.

\subsection*{Acknowledgements}
This project was initiated at the American Institute of Mathematics (San Jose, CA) workshop titled ``Singular geometry and Higgs bundles in String Theory'' from October 30 to November 3, 2017.  We would like to thank all the workshop participants for stimulating discussions and especially Szil\'{a}rd Szab\'{o} and Rodrigo Barbosa for discussions related to this project. The work of the MA and FB is supported through the DFG Emmy Noether grant AL 1407/2-1.  The work of LF is supported by NSF DMS-2005258. LF was supported by


\section{Introductory example} \label{explicit}

In this section, we walk through the example of Calabi--Yau manifolds of dimension $2$, i.e. K3 surfaces.  
This example is simple enough not to be overwhelming but complex enough to showcase some of our findings.   
For example, we determine the gauge transformations between the families of opers and of flat non-abelian Hodge connections attached to the corresponding variations of Hodge structures of weight $2$ and rank $3$.

In the following, we consider families 
\begin{equation}\label{eq:k3dim1}
    \pi\colon \mathcal{X} \to \mathcal{B},\quad \dim(\mathcal{B})=1
\end{equation}
of K3 surfaces and mostly $\mathcal{B}\cong \mathbb{CP}^1-\{0,1,\infty\}$.
Many such examples are obtained from mirror symmetry of lattice polarized K3 surfaces (\cite{Dolgachev}). 
A lattice polarized (or $\check{M}$-polarized) K3 surface is a pair $(X,j)$ of a K3 surface and a primitive embedding $j\colon \check{M}\hookrightarrow \mathrm{Pic}(X)$ of a lattice $\check{M}$ into the Picard group of $X$.
\begin{ex}\label{ex:quarticpolarized}
Let $X$ be any smooth quartic surface in $\mathbb{CP}^3$.
Then $X$ is an algebraic K3 surface with $\mathrm{rk}(\mathrm{Pic}(X))=1$.
More precisely, $\mathrm{Pic}(X)$ is generated by (the isomorphism class of) a line bundle $L\in \mathrm{Pic}(X)$ with $\int_X c_1(L)\wedge c_1(L)=4$.
In particular, $\check{M}:=\langle L\rangle$ is isomorphic to the even lattice $\langle 4 \rangle$ and $j\colon \check{M}\hookrightarrow \mathrm{Pic}(X)$ is primitive.
Hence$(X,j)$ is a $\langle 4\rangle$-polarized K3 surface. 
\end{ex}
Mirror symmetry of lattice polarized K3 surfaces is a statement about pairs of complete families\footnote{Every tangent space $T_u \mathcal{B}$ of a complete family $\mathcal{X}\to \mathcal{B}$ is naturally isomorphic to the space of infinitesimal deformations of the lattice polarized K3 surface $(X_u,j_u)$. In particular, a complete family varies non-trivially if $\dim(\mathcal{B})>0$.} 
\begin{equation}\label{eq:XBXcheckBcheck}
    \check{\pi}\colon\check{\mathcal{X}}\to \mathcal{B}_{\check{M}} \mbox{ and } \pi\colon \mathcal{X} \to \mathcal{B}_M
\end{equation}
of $\check{M}$- and $M$-polarized K3 surfaces respectively which are dual, or ``mirror'', to each other in a precise sense (see the introduction of \cite{Dolgachev}). 
Here $\mathcal{B}_{\check{M}}$ and $\mathcal{B}_{M}$ is the moduli space of $\check{M}$- and $M$-polarized K3 surfaces respectively. 

The lattice $M$ is obtained from the lattice $\check{M}$ as follows. 
For any algebraic K3 surface $X$, $H^2(X,\Z)$ with the natural intersection pairing is isomorphic to the orthogonal sum
\begin{equation}
    L_{K3}:=U^{\perp 3}\perp E_8^{\perp 2}
\end{equation}
which has signature $(t_+,t_-)=(3,19)$.
Here $U$ is the lattice of signature $(1,1)$ and $E_8$ is the lattice whose intersection product is determined by the $E_8$-root system and hence of signature $(0,8)$.
The orthogonal complement $\check{M}^\perp$ of $\check{M}$ in $L_{K3}$ splits (non-canonically) into $\check{M}^\perp=U\perp M$.
The lattice $M$ is uniquely determined up to isomorphism and is called the mirror of $\check{M}$.
It satisfies $\check{\check{M}}=M$.
Moreover, if $\check{M}$ is of signature $(1,t)$, then $M$ is of signature $(1,18-t)$.
 
The completeness condition on the families \eqref{eq:XBXcheckBcheck} implies  
\begin{equation*}
    T_{\check{u}} \check{\mathcal{B}}\cong H^{1,1}(\check{X}_{\check{u}})/j_{\check{u}}(\check{M}), \quad T_u\mathcal{B}\cong H^{1,1}(X_u)/j_u(M).
\end{equation*}
This makes sense because $\mathrm{Pic}(X)\subset \left(H^{1,1}(X)\cap H^2(X,\Z)\right)\subset H^{1,1}(X)$ for any algebraic K3 surface $X$. 
If $\check{M}$ is of rank one, which implies that $\check{M}$ is isomorphic to one of the even one-dimensional lattices $\langle 2 n \rangle$, $n\in \Z$, then 
\begin{equation*}
    \dim(\mathcal{B})=\dim H^{1,1}(X_u)-\mathrm{rk}(M)=20-19=1.
\end{equation*}
Constructions from mirror symmetry (cf. \S\ref{ss:ms}) therefore produce many families $\pi: \mathcal{X} \rightarrow \mathcal{B}$ with $\dim_{\mathbb{C}} \mathcal{B}=1$ as mirror families of $\langle 2n \rangle$-polarized K3 surfaces for $n\in \Z$.

\begin{ex}\label{ex:FirstAppearanceOfMirrorQuartic}
Let $\check{\pi}\colon \check{\mathcal{X}} \to \check{\mathcal{B}}_{\check{M}}$ be a complete family of quartic K3 surfaces so that $\check{M}=\langle 4 \rangle$, cf. \ref{ex:quarticpolarized}, and 
\begin{equation*}
M:=\langle 4 \rangle^\vee\cong U\perp E_8\perp E_8 \perp \langle -2n\rangle,
\end{equation*}
see \cite[\S 7]{Dolgachev}. 
Then a mirror family $\pi:\mathcal{X}\to \mathcal{B}_M$, simply called \emph{mirror quartic}, is a family of $M$-polarized K3 surfaces. 
In fact, $\mathcal{B}\cong \mathbb{CP}^1-\{ 0,1,\infty\}$, see \cite[\S 7]{Dolgachev}.
\end{ex}

\subsection{Hodge decomposition and filtration}\label{ss:hodgek3}
Let $X$ be an algebraic K3 surface so that we have the Hodge structure 
\begin{equation}
H^{2}(X,\mathbbm{C}) = H^{2,0} (X) \oplus H^{1,1}(X) \oplus H^{0,2}(X), \quad \overline{H^{2,0}}=H^{0,2}
\end{equation}
of weight $2$, cf. \S \ref{s:VHStt}.

Let $\check{M}=\langle 2n\rangle$ and $M:=\check{M}$ be its mirror. 
If $X$ is $M$-polarized, then the full Hodge structure on $H^2(X,\Z)$ is too large for our purposes.
Instead, we consider 
\begin{equation}
    T(X):= M^\perp= U\perp \langle 2n \rangle \subset H^2(X,\Z).
\end{equation}
This defines a Hodge substructure of $H^2(X,\Z)$ of weight $2$ and rank $3$ wit Hodge decomposition space $T^{p,q}(X)$.
Since $T^{2,0}(X)=H^{2,0}(X)$ and $T^{1,1}(X)\subset H^{1,1}(X)$, we write 
\begin{equation}
 T^{p,q}(X)=H^{p,q}(X)   
\end{equation}
by abuse of notation throughout this section. 
The equivalent data of a Hodge filtration is denoted by 
\begin{equation*}
   F^3T(X)=\{0\} \subset  F^2(X)=H^{2,0}(X)\subset F^1(X)\subset F^0(X).
\end{equation*}
It satisfies $T^{p,q}(X) = F^pT(X)/F^{p+1}T(X)$ for $p = 0, 1, 2$. 
The Hodge structure $T(X)$ carries a polarization given by the non-degenerate bilinear form
\begin{align} 
Q: T(X) \times T(X) \rightarrow \mathbb{Z}\,, \quad Q(\alpha,\beta) = - \int_{X} \alpha \wedge \beta\, .
\end{align}
which extends to $T(X,\C):=T(X)\otimes_{\Z}\C$.
The Weil operator $W \in \mathrm{End}(T(X,\C))$ acts on $T^{p,q}$ by multiplication with a constant:
\begin{equation}
\left. W \right|_{H^{p,q}} = i^{p-q}\,.
\end{equation}
Together with the polarization, it induces the non-degenerate pairing
\begin{align}
\eta : T(X,\mathbb{C}) \times T(X,\mathbb{C}) \rightarrow \mathbb{C}\,,\quad \eta(\omega_i,\omega_j)=Q(W(\omega_i),\omega_j)\,, 
\end{align}
as well as the Hermitian (Hodge) metric, 
\begin{align}
h : T(X,\mathbb{C}) \times T(X,\mathbb{C}) \rightarrow \mathbb{C}\,,\quad h(\omega_i,\omega_j)=Q(W(\omega_i),\overline{\omega_j})\,,
\end{align}
for $ \omega_i,\omega_j \in T(X,\C)\,$.


\subsection{Variation of Hodge structure}\label{ss:vhsintro}
We next consider a complete family 
\begin{equation*}
    \pi\colon \mathcal{X}\to \mathcal{B}:=\mathbb{CP}^1-\{0,1,\infty\}
\end{equation*}
of $M$-polarized K3 surfaces with $\mathrm{rk}(M)=18$. 
This setup generalizes Example \ref{ex:FirstAppearanceOfMirrorQuartic} is the most common in mirror symmetry for such lattice polarized K3 surfaces. 
The variation of the Hodge filtrations over the base manifold $\mathcal{B}$ is governed by the variation of Hodge structures
$$
(\mathcal{H}_{\mathbb{Z}},\nabla, Q, F^{\bullet}\mathcal{H}_{\mathcal{O}}).
$$
Here $\mathcal{H}_{\mathcal{O}}= \mathcal{H}_{\mathbb{Z}} \otimes \mathcal{O}_{\mathcal{B}}$ for the locally constant sheaf $\mathcal{H}_{\Z}$ with stalks $H^2(X_u,\Z)$ and $\nabla$ is the induced Gau\ss--Manin connection.
The decreasing filtration $F^{\bullet} \mathcal{H}_{\mathcal{O}}$ of holomorphic sub-bundles $F^i\mathcal{H}_{\mathcal{O}}\,, i=2,1,0$, is defined by taking fiberwise the previously defined Hodge filtration.
It satisfies Griffiths' transversality:
$\nabla F^p \mathcal{H}_{\mathcal{O}} \subset F^{p-1} \mathcal{H}_{\mathcal{O}} \otimes \Omega^1_{\mathcal{B}}\,,$ 
see \S\ref{s:VHStt} for more details.
Moreover, we write $\mathcal{H}^{p,q} = F^{p}\mathcal{H}_{\Oo}/F^{p+1}\mathcal{H}_{\Oo}$ as in the fiberwise case.

Let $\mathcal{L}=F^2\mathcal{H}_{\mathcal{O}}$ be the line bundle over $\mathcal{B}$ whose fibers are the spaces $H^{2,0}(X_u)$ and $\mathcal{L}^{-1}$ be its dual. 
Further let $\omega_0$ be a local frame.  
By the properties of a VHS, this form satisfies the following Picard--Fuchs equation in cohomology:
\begin{equation}\label{cohPF}
\nabla^3_{\frac{\partial}{\partial u}} \omega_0 =  -b_2 \nabla^2_{\frac{\partial}{\partial u}} \omega_0-b_1(u) \nabla_{ \frac{\partial}{\partial u}} \omega_0  -b_0 \omega_0   \,,
\end{equation}
with rational functions $b_i(u)\,,i=0,1,2$.

This relation in cohomology becomes a third order ODE, the Picard--Fuchs equation, with regular singular points for the (multi-valued) periods $\pi^i\,, i=0,1,2$  obtained by integrating $\omega_0$ over a basis of cycles $\gamma^0,\gamma^1,\gamma^2\in H_2(X_u,\mathbb{Z})$. 
\begin{equation}\label{eq:pfK3}
\left( \frac{d^3}{du^3} + b_2 \frac{d^2}{du^2} +  b_1 \frac{d}{du} + b_0\right) \pi^i (u) =0\,, \quad i=0,1,2\,.
\end{equation}
In this example, the rational functions $b_i$ are related to the following Hodge-theoretic pairing: 
\begin{dfn}\label{def:gy}
The Griffiths--Yukawa coupling $c \in \Gamma(\mathcal{L}^{-2} \otimes ((T^*\mathcal{B})^2)$ is defined by 
\begin{equation}
    c(\omega,\omega'):=\eta(\omega, \nabla^2\omega')\in \Gamma(T^*\mathcal{B})^2),\quad \omega,\omega'\in \mathcal{L}.
\end{equation}
If we have fixed a local frame $\omega_0$ and $\tfrac{\partial}{\partial u}$ of $\mathcal{L}$ and $T\mathcal{B}$ respectively, then we denote by 
\begin{equation}\label{eq:defc}
    c_{uu}:=\eta(\omega_0,\nabla^2_u\omega_0)\,,\quad \nabla_u:=\nabla_{\partial/ \partial u}\,,
\end{equation}
the coordinate expression of $c$.
\end{dfn}
The next proposition gives the relation to the functions $b_i$ in the Picard--Fuchs equation.
\begin{prop}
The coordinate expression $c_{uu}$ of the Griffiths--Yukawa coupling satisfies the differential equation:
\begin{equation}\label{eq:gyeq}
\partial_u c_{uu} = -\frac{2}{3} b_2\,c_{uu}\,,
\end{equation}
where $b_2$ is the rational function appearing in the Picard--Fuchs equation.
\end{prop}

\begin{proof}
We have
\begin{equation} \label{cder}
\partial_u c_{uu} = \int_{X_u} \nabla_u \omega_0 \wedge \nabla^2_u \omega_0 + \int_{X_u} \omega_0 \wedge \nabla^3_u \omega_0\,,
\end{equation}
and furthermore
\begin{equation*}
\int_{X_u} \omega_0 \wedge \nabla_u \omega_0 =0\,,
\end{equation*}
by type considerations. By differentiating this last equation it follows that \begin{equation}\label{eq:defc2}
    c_{uu}= - \int_{X_u} \nabla_u \,\omega_0 \wedge \nabla_u\, \omega_0
\end{equation} and hence $$ \partial_u c_{uu} =- 2 \int_{X_u} \nabla_u \omega_0 \wedge \nabla^2_u \omega_0\,,$$
the claim follows by substituting this last expression for the first term on the RHS of \eqref{cder} and the Picard--Fuchs equation for the second term.
\end{proof}


\subsection{Frame for the bundle and Hermitian metric}

We next construct a local frame $\omega= \left( \omega_0\,, \omega_1\,, \omega_2 \right)$ for the bundle $\mathcal{H}$ with fibers $H^2(X_u,\mathbb{C})$ such that $\omega_i\in \mathcal{H}^{2-i,i}$. 
We choose a local frame $\omega_0$ of $\mathcal{H}^{2,0}=\mathcal{L}$ and define
\begin{equation} \label{eq:h00} e^{-K}:= h_{0\bar{0}} = h(\omega_0,\omega_0)= \int_{X_u} \omega_0 \wedge \overline{\omega}_0\,,  \end{equation}
where $K:\mathcal{B} \to \mathbb{R}$ defines a K\"ahler potential for the projective special K\"ahler metric $G$ on $\mathcal{B}$. We then proceed to construct $\omega_1$ and $\omega_2$ from $\omega_0$ using the Gau\ss--Manin connection. By Griffiths transversality we obtain 
$$(\nabla_u\, \omega_0 ) \de u  \in F^1H^2(X_u,\C) \otimes T^* \mathcal{B}\, , \quad \nabla_u:=\nabla_{\partial/\partial u} $$ 
and hence we make the ansatz
\begin{equation*}
\nabla_u \, \omega_0 \de u =  A_u  \de u \omega_0 +\omega_1\,,
\end{equation*}
with $A_u \de u \in \Gamma(T^* \mathcal{B})$ and $\omega_1 \in H^{1,1}(X_u)\otimes T^*\mathcal{B}$. 
Thus we obtain $A_u= - \partial_u K =:- K_u$ from 
$$ \partial_u e^{-K}= - K_u e^{-K}=  \int_{X_u} \nabla_u \omega_0 \wedge \bar{\omega}_0 = A_u \int_{X_u} \omega_0 \wedge \overline{\omega}_0\,,$$
since $-K_u= (h_{0\bar{0}})^{-1} \partial_u h_{0\bar{0}}$, this is also the coordinate expression for the Chern connection one-form, and we have:
\begin{equation}\label{eq:omega1}\omega_1  = ((\nabla_u-D_u)\, \omega_0 )\, \de u \in H^{1,1}(X_u,\mathbb{C}) \otimes T^*\mathcal{B}\,.\end{equation}
We proceed by constructing $\omega_2 \in H^{0,2}(X_u) \otimes (T^* \mathcal{B})^2$ similarly:
\begin{equation}\label{eq:omega2} \omega_2\, = ((\nabla_u - D_u ) \, (\omega_1 ))  \de u \, ,
\end{equation}
with $D_u \omega_1 = h_{1\bar{1}}^{-1} \partial_u h_{1\bar{1}}\, \omega_1$.  We then obtain similarly:
\begin{equation}\label{VHSclosure}
 ((\nabla_u - D_u ) \, \omega_2 )  =0\,, 
 \end{equation}
with $$D_u\, \omega_2 =h_{2\bar{2}}^{-1} \partial_u h_{2\bar{2}}\,  \omega_2 \,.$$ 
We get the following:
\begin{prop}\label{hermetric}
The expression for the Hermitian metric $h=h_{a\bar{b}} dt^a\otimes d\bar{t}^{\bar{b}}\,, a,b=0,1,2$ in the frame $(\omega_0, \omega_1,\omega_2)$ is given by
\begin{equation}\label{eq:hinex}
(h_{a\bar{b}})= \left(\begin{array}{ccc} 
e^{-K} &0 &0 \\
0& e^{-K}G_{u\overline{u}} \de u \de \bar{u} & 0 \\
0 &0 &\e^{-K} G_{u \ubar}^2 \de u^2 \de \bar{u}^2   
\end{array} \right) \,.
\end{equation}
where $G_{u \ubar}= \del_u \del_{\ubar} K$. The 
 K\"ahler metric $G_{u\bar{u}} \de u \,\de \bar{u}$ further satisfies:
\begin{equation} \label{eq:Kahlerinex} e^{-K} \, G_{u\bar{u}} = c_{uu}\, e^{K}\, G^{u\bar{u}} \,\overline{c_{uu}} \,,
\end{equation}
where $G^{u\bar{u}}=G_{u \ubar}^{-1}$ are the components of the inverse metric. The expression for the complex pairing $\eta=\eta_{ab}dt^a dt^b\,, a,b=0,1,2$ in this frame is:
\begin{equation} 
(\eta_{ab})= c_{uu} \left(\begin{array}{ccc} 
0 &0 &1 \\
0& 1  & 0 \\
1 &0 & 0   
\end{array} \right) \de u^2 \,.
\end{equation}

\end{prop}

\begin{proof}
The metric is diagonal in the frame $(\omega_0, \omega_1, \omega_2)$ because (1) the basis respects the Hodge decomposition and (2) the polarization of the VHS obeys the Riemann--Hodge bilinear relations \ref{eq:RHi}, \ref{eq:RHii} given in \S\ref{ss:VHS}. We have defined $h(\omega_0, \omega_0)=e^{-K}$
in \eqref{eq:h00}, for the other entries we have
\begin{align*}
    h_{1\bar{1}}&= h(\omega_1,\omega_1)= -\int_{X_u} \omega_1 \wedge \overline{\omega_1} \\
 &= -\int_{X_u} (\nabla_u \omega_0 + K_{u} \omega_0) \wedge (\nabla_{\bar{u}} \overline{\omega_0} + K_{\bar{u}} \overline{\omega_0})   \de u\, \de \bar{u} \\
 &= \left(- \partial_u \partial_{\bar{u}} e^{-K} - K_{\bar{u}} \partial_u e^{-K} - K_{u} \partial_{\bar{u}} e^{-K} - K_u K_{\bar{u}} e^{-K}\right) \de u \, \de \bar{u} \\
 &= e^{-K}G_{u\bar{u}} \de u\, \de \bar{u}\,.
 \end{align*}

To obtain $h_{2\bar{2}}:= h(\omega_2,\omega_2)$ we first note that by construction we have $\omega_2 \in H^{0,2} \otimes (T^*\mathcal{B})^2$. 
Since $H^{0,2}$ is spanned by $\overline{\omega_0}$, we make the ansatz
\begin{equation}\label{eq:ansatzomega2}
\omega_2 = \overline{\omega_0}\, D_{uu}\,\de u^2\,,
\end{equation}
with $D_{uu}\de u^2 \in  (T^*\mathcal{B})^2$. 
On the one hand, the polarization on $H^2(X_u,\Z)$ together with the Riemann--Hodge bilinear relations give  
\begin{align*}
    \int_{X_u} \omega_0 \wedge \omega_2 &= \int_{X_u} \omega_0 \wedge (\nabla_u-D_u)^2 \omega_0 \de u^2 \\
    &= \int_{X_u} \omega_0 \wedge \nabla_u^2\, \omega_0 \,\de u^2 = c_{uu}\, \de u^2\,.
\end{align*}
On the other hand, using the ansatz \eqref{eq:ansatzomega2}, we obtain
\begin{align*}
    \int_{X_u} \omega_0 \wedge \omega_2 &= \int_{X_u} \omega_0 \wedge \overline{\omega_0}\,D_{uu} \de u^2 \\
    &= e^{-K}\,D_{uu} \de u^2 \,.
\end{align*}
Comparing the previous two equalities gives $D_{uu}= c_{uu} e^{K}$ and hence
$$ h_{2\bar{2}}= \int_{X_u} \omega_2 \wedge \overline{\omega_2} = e^{2K} c_{uu} \overline{c_{uu}} \left(\int_{X_u} \overline{\omega_0} \wedge \omega_0\right) \de u^2 \, \de \bar{u}^2 = e^{K}c_{uu}\overline{c_{uu}} \de u^2 \,\de \bar{u}^2\,. $$

To obtain the expression relating $G_{u \ubar}, K, c_{u u}$ in \eqref{eq:Kahlerinex} we note that $\omega_1 \in H^{1,1} \otimes T^*\mathcal{B}$; since $H^{1,1}$ is one-dimensional we must have:
\begin{equation} \label{eq:ansatz2} \omega_1 = \overline{\omega_1}\, L_{u}^{\bar{u}}\,\de u \otimes \frac{\partial}{\partial \bar{u}}\,,\end{equation}
with $L_{u}^{\bar{u}}\,\de u \otimes \frac{\partial}{\partial \bar{u}} \in  T^*\mathcal{B} \otimes \overline{T\mathcal{B}}$. To determine $L_u^{\ubar}$, we compute $\int_{X^1} \omega_1 \wedge \omega_1$ in two ways. First,
\begin{align*}
    -\int_{X_u} \omega_1 \wedge \omega_1 &= -\left(\int_{X_u} (\nabla_u-D_u)\omega_0 \wedge (\nabla_u-D_u) \omega_0\right) \de u^2 \\
    &= -\left(\int_{X_u} \nabla_u \,\omega_0 \wedge \nabla_u\,\omega_0\right) \,\de u^2 \overset{\eqref{eq:defc2}}{=} c_{uu} \,\de u^2\,,
\end{align*}
on the other hand, using the ansatz in \eqref{eq:ansatz2} we obtain:
\begin{align*}
    -\int_{X_u} \omega_1 \wedge \omega_1 &= - \left(\int_{X_u} \omega_1 \wedge \overline{\omega_1}\right) \,L_{u}^{\bar{u}} \de u\otimes \frac{\partial}{\partial \bar{u}}\\
    &=  e^{-K} G_{u\bar{u}} \,L_{u}^{\bar{u}} \de u^2\,, 
\end{align*}
and hence $L_{u}^{\bar{u}}= e^{K}G^{u\bar{u}} c_{uu}\,.$
We have:
$$ e^{-K}G_{u\bar{u}} \de u \de \bar{u}= - \int_{X_u} \omega_1 \wedge \overline{\omega_1}= - e^{2K}G^{u\bar{u}} G^{u\bar{u}} c_{uu}\overline{c_{uu}} \int_{X_u} \overline{\omega_1} \wedge \omega_1 =e^{K}G^{u\bar{u}}c_{uu}\overline{c_{uu}} \de u \, \de \bar{u}\,,$$
this also gives the form the the entry $h_{2\bar{2}}$ in \eqref{eq:hinex}.
 The entries of $\eta$ follow from \eqref{eq:defc},\eqref{eq:defc2}.
\end{proof}

\begin{rem}\label{rem:metric}
In \eqref{eq:hinex} we use the fact that $\nabla$ is an oper to identify $\mathcal{H}^{p,q}$ with $\mathcal{H}^{p-1,q+1}\otimes TB$ via the map $\nabla-D$ (whose $(1,0)$-part is the Higgs field, see Section \ref{s:VHSHiggs}). For example, $e^{-K}G_{u\bar{u}}\de u \de \bar{u}$ is a metric on $\mathcal{H}^{2,0}\otimes TB  \cong \mathcal{H}^{1,1}$. The entry $h_{2\bar{2}}=e^{-K}G_{u\bar{u}}^2 \de u^2 \, \de \bar{u}^2$ is a metric on $\mathcal{H}^{2,0}\otimes TB^{\otimes 2}\cong \mathcal{H}^{0,2}$. 
\end{rem}

\subsection{Holomorphic Gau\ss--Manin and non-abelian Hodge/\texorpdfstring{$tt^*$}{tt*} flat connection}\label{sec:opergm}

The variation of Hodge structure provides the data of a holomorphic Gau\ss-Manin connection which is reflected by the Picard--Fuchs equation describing the variation of Hodge structure in \eqref{cohPF},\eqref{eq:pfK3}. 
Assume that $\omega_0$ is a cyclic vector, i.e. 
$$\omega_{\mathrm{GM}} = \left( \omega_0\,, \nabla_u \omega_0 \,\de u \,, \nabla^2_u \omega_0\, \de u^2 \right)\,, $$
is a frame of $\mathcal{H}$  (cf. Definition \ref{dfn:cyclic})\footnote{In analogy to Remark \ref{rem:metric} we do not contract with a vector field here.}.
From \eqref{cohPF} we see that the Gau\ss--Manin connection $\nabla$ is given by
\begin{equation}\label{eq:GMK3}
\nabla_{\mathrm{GM}} = \de +\left( \begin{array}{ccc} 0&0&- b_0 \\
1&0&- b_1\\
0&1&- b_2
\end{array}\right) \de u \,. 
\end{equation}
in this frame.
Note that it is a holomorphic connection on $\mathcal{H}$.

The $\C^\times$-family of flat non-abelian Hodge (or $tt^*$)-connections can be explicitly given in terms of the previously defined entities. We denote the frame, which we have constructed in Proposition \ref{hermetric}, by 
\begin{equation*}
\omega_{\mathrm{NAH}}=(\omega_0\,,\omega_1\,,\omega_2)\,,
\end{equation*}
where $\omega_1=(\nabla_u - D_u)\omega_0$ and $\omega_2=(\nabla_u - D_u)\omega_1$.

\begin{prop}\label{NAHconnection}
The Higgs field $\varphi \in \mathrm{End}(H^{2}(X_u,\mathbb{C})) \otimes T^* \mathcal{B}$ in the holomorphic frame $\omega_{\mathrm{NAH}}$ is
\begin{equation} \label{eq:Higgsinex}
\varphi= \left(  \begin{array}{ccc}  0&0&0 \\ 1&0&0 \\ 0&1&0 \end{array} \right) \,
\end{equation}
so that the following relationship is satisfied:
\begin{equation}
(\nabla_u - D_u) \,\omega_{\mathrm{NAH}}\, \de u = \varphi(\omega_{\mathrm{NAH}}).
\end{equation}
The Hermitian metric $h=h_{a\bar{b}}\, dt^a\,d\bar{t}^{\bar{b}}\,, a,b=0,1,2$, defined in \eqref{eq:hinex}, in the frame $\omega^\hbar_{\mathrm{NAH}}$ is 
\begin{equation} \label{eq:hinexzeta}
(h_{a\bar{b}})= \left(\begin{array}{ccc} 
e^{-K} &0 &0 \\
0& \,e^{-K}G_{u\overline{u}} \de u \de \ubar  & 0 \\
0 &0 & \,e^{-K} G_{u\bar{u}} G_{u\bar{u}}   \de u^2 \de \ubar^2
\end{array} \right) \,.
\end{equation}
Thus the adjoint of the Higgs field with respect to $h$ 
is:
\begin{equation} \label{eq:HiggsTinex} 
\varphi^{\dagger}=  \left(  \begin{array}{ccc}  0&G_{u\bar{u}} \de u \de \ubar &0 \\ 0&0&G_{u\bar{u}}  \de u \de \ubar \\ 0&0&0 \end{array} \right)\,,
\end{equation}
so that the following relationship holds:
\begin{align}
&(\nabla_{\bar{u}} - D_{\bar{u}})\,  \omega_{\mathrm{NAH}}  \,\de \bar{u} = \varphi^{\dagger} (\omega_{\mathrm{NAH}})  \,. 
\end{align}
This gives a solution of Hitchin's equations (and the $tt^*$-equations)
\begin{equation}
\left[ D_u , D_{\bar{u}}\right] = -\left[ \varphi, \varphi^{\dagger}\right]\,, \quad D_{\bar{u}} \varphi=0\,,
\end{equation}
hence  we get the family of flat non-abelian Hodge ($tt^*$)-connections:
\begin{equation}\label{eq:nahconns}
\nabla_{\mathrm{NAH}}^{\zeta} =   \frac{1}{\zeta}\,\varphi + D+ \zeta\, \varphi^{\dagger}\,, \quad \zeta \in \C^{\times}\,,
\end{equation}
where 
$$ D= \de + h^{-1} \partial_u h \,\de u.$$
\end{prop}

For the general setup of this proposition, see Section \ref{s:VHSHiggs}. 

\begin{proof}
The basis elements $\omega_1$ and $\omega_2$ are defined in \eqref{eq:omega1} and \eqref{eq:omega2} in such a way that the difference of the Gau\ss--Manin and the Chern connection is
\begin{equation}
 (\nabla_u -D_u ) \left(\omega_0\,, \omega_1\,, \omega_2 \right) \de u= \left(   \omega_0\,,   \omega_1\,,  \omega_2  \right)\left(  \begin{array}{ccc} 0&0&0 \\ 1&0&0 \\ 0&1&0  \end{array}\right) \,.
\end{equation}
The matrix on the right hand side of this equation is the Higgs field in \eqref{eq:Higgsinex}. 
We then compute $\varphi^{\dagger} = h^{-1} \overline{\varphi}^T h$ for $h$ defined in \eqref{eq:hinexzeta}. The result matches with the expression in \eqref{eq:HiggsTinex}. 
The equation $D_{\ubar} \varphi = 0$ imposes no equations on the function $K$ and $G_{u \ubar}$ featured in $h$. Meanwhile, the equation $\left[D_u, D_\ubar\right] = - \left[\varphi, \varphi^\dagger \right]$ imposes that $G_{u \ubar}= \del_u \del_{\ubar} K$ and that $G_{u \ubar}$ satisfies
\begin{equation} \label{eq:logG}
    \del_u \del_{\ubar} \log G_{u \ubar} = G_{u \ubar}.
\end{equation}
The differential ring relation in \eqref{eq:diffring1} is  $\del_u \log G_{u \ubar} - \del_u K = - \frac{1}{3\hbar} b_2$. Taking $\del_\ubar$ of both sides in the first differential ring relation \eqref{eq:diffring1} and using that $\del_{\ubar} b_2=0$, we get \eqref{eq:logG}.
\end{proof}

We next construct a one--parameter family $\nabla_{\mathrm{GM}}^\hbar$ of flat connections, $\hbar\in \C^*$, from $\nabla_{\mathrm{GM}}$.
It will be explicitly related to $\nabla_{\mathrm{NAH}}^\zeta$ later on. 
For the construction, consider a solution $\omega_0^\hbar$ of the $\hbar$-deformed Picard--Fuchs equations
\begin{equation}\label{eq:PFK3hbar}
\hbar^3 \nabla^3_{\frac{\partial}{\partial u}} \omega_0^\hbar(u) =  -b_2\, \hbar^2 \nabla^2_{\frac{\partial}{\partial u}} \omega_0^\hbar(u)-b_1 \,\hbar \nabla_{ \frac{\partial}{\partial u}} \omega_0^\hbar(u)  -b_0 \,\omega_0^\hbar(u)
\end{equation}
for any $\hbar\in \C^*$.
Then we consider the frame
\begin{equation}\label{eq:omegaGMhbar}
\omega^{\hbar}_{\mathrm{GM}} := \left( \omega_0^\hbar\,, \hbar \nabla_u \omega_0^\hbar\, \de u \,, \hbar^2 (\nabla_u)^2 \omega_0^\hbar \, \de u^2 \right)\,.
\end{equation}
In this frame we define the connection
\begin{equation}\label{eq:GMopers}
\nabla^{\hbar}_{\mathrm{GM}} = \de +\frac{1}{\hbar}\left( \begin{array}{ccc} 0&0&- b_0 \\
1&0&- b_1\\
0&1&- b_2
\end{array}\right) \de u
\end{equation}
where $b_i = b_i(u)$ as before.

\begin{rem}
The study of the asymptotics in $\hbar$ of the periods of the $\hbar$-deformed $\omega_0^\hbar$ is subject of the higher rank WKB method, see e.g.~\cite{Hollands:2019wbr} and references therein.
\end{rem}

We note furthermore that the differential equation \eqref{eq:gyeq} obeyed by the coordinate expression of the Griffiths--Yukawa coupling becomes:
\begin{equation}\label{eq:gyeqhbar}
\partial_u c_{uu} = -\frac{2}{3\hbar} b_2 \,c_{uu}\,,
\end{equation}
after using the $\hbar$-deformed Picard-Fuchs equation \eqref{eq:PFK3hbar}.

\subsection{The differential ring relations}\label{sec:differentialring}
We next derive three differential ring relations between functions associated to the K\"{a}hler metric $G_{u\bar{u}}\, \de u\, \de \bar{u}$ and the coefficient functions of the Picard--Fuchs equation in \eqref{eq:pfK3}.
This will be very useful to explicitly relate the families $\nabla_{\mathrm{NAH}}^\zeta$ and $\nabla_{\mathrm{GM}}^\hbar$ of flat connections. 
At the end of this subsection we comment on how these differential ring relations are related to the ones found in \cite{Alim:2014vea}.  

\begin{prop}\label{diffring}
Let $\hbar\in\mathbb{C}\,,K_u := \partial_u K\,,\Gamma_{uu}^u:=G^{u\bar{u}} \partial_u G_{u\bar{u}} $, $\omega^\hbar_0 \in H^{2,0}(X_u,\mathbb{C})$ and $b_i(u) \,, i=0,1,2$ the coefficients of the $\hbar$-deformed Picard--Fuchs equation describing the variation of Hodge structure \eqref{eq:PFK3hbar}. The following relations hold:
\begin{enumerate}
\item
\begin{equation}\label{eq:diffring1}
\Gamma_{uu}^u - K_u = -\frac{1}{3\hbar} b_2\,, \\
\end{equation}
\item
\begin{equation}\label{eq:diffring2}
 \partial_u K_u = \frac{1}{2} K_u^2 -\frac{1}{3\hbar} b_2 K_u -\frac{1}{9\hbar^2} b_2^2 -\frac{1}{6\hbar} \partial_u b_2 +\frac{1}{2\hbar^2} b_1\,, \\
\end{equation}
\item
\begin{equation}\label{eq:diffring3}
\hbar^2 \,\partial_u^2 b_2 +2 b_2 \hbar\, \partial_u  b_2 -3 \hbar\,\partial_u b_1 +\frac{4}{9} b_2^3 - 2 b_1 b_2 + 6 b_0 =0\,.
\end{equation}
\end{enumerate}
\end{prop}

\begin{proof}
The first relation follows from the relation \eqref{eq:Kahlerinex}, proven in Proposition \ref{hermetric}:
\begin{equation} e^{-K} \, G_{u\bar{u}} = c_{uu}\, e^{K}\, G^{u\bar{u}} \,\overline{c_{uu}} \,,
\end{equation}
by taking the derivative with respect to $u$ and then using the differential equation satisfied by the Griffiths--Yukawa coupling in \eqref{eq:gyeqhbar}.
We obtain the second and third relations by considering \eqref{VHSclosure}:
\begin{equation}
\hbar^3\,(\nabla_u -D_u) \omega_2^\hbar =0 \,.
\end{equation}
This innocent looking equation turns out to be very rich once its ingredients are spelled out, using the first relation as well as the explicit coordinate expressions for the Chern connection components we obtain the following equation:

\begin{align}
&\hbar^3 \left( \nabla_u + K_u - 2 \Gamma_{uu}^u \right) \left(\nabla_u + K_u -\Gamma_{uu}^u\right) \left(\nabla_u + K_u\right) \omega_0^\hbar =0 \,, \\
&\left( \hbar \nabla_u -\hbar K_u +\frac{2}{3} b_2 \right) \left(\hbar\nabla_u + \frac{1}{3} b_2\right) \left(\hbar\nabla_u + \hbar K_u\right) \omega_0^\hbar =0 \,,
\end{align}
in this equation we substitute for $\hbar^3 \nabla_u^3 \omega_0$ using the $\hbar$-deformed Picard--Fuchs equation \eqref{eq:PFK3hbar} and obtain three independent equations in cohomology, namely the coefficients of $\hbar^2 \nabla_u^2 \omega_0^\hbar, \hbar \nabla_u \omega_0^\hbar$ and of $\omega_0^\hbar$. The coefficient of $\hbar^2 \nabla_u^2 \omega_0^\hbar$ is zero by the algebraic equations, which were shown to hold in the first part of the proposition. The vanishing of the coefficient of $\hbar \nabla_u \omega_0^\hbar$ gives the second relation of this proposition. Substituting this relation into the coefficient of $\omega_0^\hbar$ gives the third part of the proposition. 
\end{proof}

\begin{rem} \hfill
\begin{itemize}
\item Note that first relation can be derived from the second relation by taking $\del_{\ubar}$ of both sides, and then dividing by $G_{u \ubar}$.  Similarly, by taking $\del_{\ubar}$ of the first relation we get $\del_{u} \del_{\ubar} \log G_{u \ubar} = G_{u \ubar}$. This means that $G_{u \ubar} \de u \de \ubar$ is a metric of constant curvature $-2$\footnote{Recall that the curvature of a metric $G_{u \ubar} \de u \de \ubar$ on a Riemann surface is $-\frac{2}{G_{u \ubar}} \del_{u} \del_{\ubar} \log G_{u \ubar}$.} on the Riemann surface $\mathcal{B}=\mathbb{P}^1-\{0, 1, \infty\}$.

\item The function $K$ is a K\"ahler potential for the metric $G$ determined by the choice of $\omega_0^\hbar$ by \eqref{eq:h00}. Given a choice of $\omega_0$, the coefficients of the Picard--Fuchs equation are determined and the differential ring gives algebraic relations between $K,\partial_u K$ and $\partial_u^2 K$ encoded in \eqref{eq:diffring2}. Different choices of $\omega_0^\hbar$ result in different potentials $K$, namely rescaling $\omega_0^\hbar$ by $\omega_0 \mapsto f \omega_0$ shifts the potential $K \mapsto K +\log |f|^2$.  The functions  $b_0, b_1, b_2$ change accordingly.
We note that the third relation doesn't feature the metric $G_{u \ubar} \de u \de \ubar$ at all.

\item We note here that the first and second parts of this proposition were derived in a different manner in \cite{Alim:2014vea}, namely by using the explicit expressions for the curvature of the Chern connection in terms of the Higgs field. The outcome in \cite{Alim:2014vea} are differential ring relations which are determined up to rational functions on the moduli space which are determined case-by-case. This is due to the fact that for obtaining the relations in \cite{Alim:2014vea} a successive integration of the $\partial_{\bar{u}}$ was used, which yields relations up to holomorphic functions. In the current derivation, the holomorphic functions are determined by using the Picard--Fuchs equation. As a qualitative example highlighting the difference of the approach of \cite{Alim:2014vea} to our current approach we note that the resulting Hitchin equation of the current setup \eqref{eq:logG}:
\begin{equation*} 
    \del_u \del_{\ubar} \log G_{u \ubar} = G_{u \ubar}\,,
\end{equation*}
can be integrated to give:
\begin{equation*}
    \Gamma_{uu}^u - K_u = \textrm{hol}\,,
\end{equation*}
where $\textrm{hol}$ is an undetermined holomorphic function, which is in the kernel of $\partial_{\bar{u}}$. This is in contrast to the relations in cohomology which use the Picard--Fuchs equations and have no undetermined data. In this example $\textrm{hol}(u)=-\frac{1}{3\hbar}b_2$. Moreover the third relation is completely novel to the current work. It is a constraint on the coefficients of the Picard--Fuchs equation which we interpret as the constraint that this Picard--Fuchs equation describes a polarized variation of Hodge structures.  
\end{itemize}

\end{rem}

\subsection{Gauge transformation and conformal limit}\label{sec:gauge}
Using the differential ring relations, we next give an explicit gauge transformation between the non-abelian Hodge connection $\nabla_{\mathrm{NAH}}^\hbar$ and the oper $\nabla_{\mathrm{GM}}^\hbar$ for any $\hbar\in \C^\times$.
As an application we determine the $\lambda$-conformal limit (\cite{Gaiotto:2014bza}, \cite{Dumitrescu2016a}) for any $\lambda \in \C^\times$ associated to the family $\nabla_{\mathrm{NAH}}^\zeta$ of non-abelian Hodge flat connections. 
For the next proposition, let $\omega_{\mathrm{NAH}}^\hbar = (\omega_0^\hbar, \varphi_u(\omega_0^\hbar), \varphi_u^2(\omega_0^\hbar))$.

\begin{prop}\label{prop:gauge}
For any $\hbar\in \C^\times$ fix $\omega_0^\hbar$ solving the $\hbar$-deformed Picard--Fuchs equation in cohomology (cf. \eqref{cohPF}).
Then $\nabla^\hbar_{\mathrm{NAH}} = \frac{1}{\hbar} \varphi  + D + \hbar \varphi^{\dagger}$ is gauge equivalent to $\nabla^{\hbar}_{\mathrm{GM}}$.
Concretely, let $\omega_{GM}^\hbar$ be the frame \eqref{eq:omegaGMhbar} of $\mathcal{H}^\hbar$.  
Define the gauge
transformation $\mathcal{A}^{\hbar}$ satisfying
\begin{equation}
  \omega_{\mathrm{GM}}^{\hbar}=  \omega_{\mathrm{NAH}}^{\hbar}\mathcal{A}^{\hbar}\,,
\end{equation}
which is concretely given by
\begin{equation}\label{eq:gauge}
(\mathcal{A}^{\hbar})^{-1}=\begin{pmatrix} 1& \hbar \,K_u\, \de u & \hbar^2 B_{uu} \de u^2 \\
0&1 & \hbar K_u +\frac{1}{3} b_2 \\
0&0& 1
\end{pmatrix}\,,
\end{equation}
with $$B_{uu}= \frac{1}{2} K_u^2 -\frac{1}{9\hbar^2} b_2^2 -\frac{1}{6\hbar}\partial_u b_2 + \frac{1}{2\hbar^2} b_1 \overset{\eqref{eq:diffring2}}{=} \del_{u} K_u + \frac{1}{3\hbar} b_2 K_u \,.$$
Then the gauge transformation $\mathcal{A}^{\hbar}$ transforms the $\hbar$-non-abelian Hodge flat connection $\nabla_{\mathrm{NAH}}^\hbar$ into the oper $\nabla_{\mathrm{GM}}^\hbar$, i.e.
\begin{equation}\label{eq:gaugetrafo}
\nabla^{\hbar}_{\mathrm{GM}} =  (\mathcal{A}^{\hbar})^{-1} \circ \nabla^{\hbar}_{\mathrm{NAH}} \circ \mathcal{A}^{\hbar}.
\end{equation}
\end{prop}

\begin{proof}
Recall that in the frame $\omega_{GM}^\hbar$, the connection $\nabla_{GM}^\hbar$ is given by \begin{equation}
\nabla^{\hbar}_{\mathrm{GM}} = \de +\frac{1}{\hbar}\left( \begin{array}{ccc} 0&0&- b_0 \\
1&0&- b_1\\
0&1&- b_2
\end{array}\right) \de u \,. 
\end{equation}
Then the proof that $\nabla_{\mathrm{GM}}^{\hbar} = (\mathcal{A}^{\hbar})^{-1}\circ  \nabla_{\mathrm{NAH}}^{\hbar} \circ  \mathcal{A}^{\hbar}$ is a tedious but straightforward computation and follows from Propositions \ref{NAHconnection} and \ref{diffring}.  We give the details to illustrate how the differential ring relations in (\ref{eq:diffring1})-(\ref{eq:diffring3}) and the relation $G_{u\ubar}= \del_u \del_{\ubar} K$ are used.
We compute that 

\begin{equation}
    \left((\mathcal{A}^{\hbar})^{-1}\circ  \nabla_{\mathrm{NAH}}^{\hbar} \circ  \mathcal{A}^{\hbar} -  \nabla^{\hbar}_{\mathrm{GM}}\right)^{(0,1)} = \hbar \begin{pmatrix} 0 & *_{(1,2)} &*_{(1,3)} \\ 0 & 0 & *_{(2,3)}\\ 0 & 0 & 0 \end{pmatrix} \,,
\end{equation}
where
\begin{eqnarray*}
*_{(1,2)}&=& G_{u \ubar} - \del_u \del_{\ubar} K \,,\\
*_{(1,3)}&:=& -\del_{\ubar} B_{uu} + G_{u \ubar} K_u\,,\\
*_{(2,3)} &:=& G_{u \ubar} - \del_u \del_{\ubar} K\,.
\end{eqnarray*}
This vanishes since $G_{u \ubar} = \del_u \del_{\ubar} K$ by definition, and similarly $\del_{\ubar} B_{uu}:= \del_{\ubar}(\frac{1}{2}K_u^2)$. 
The $(1,0)$-part is given by
\begin{equation}
    \left((\mathcal{A}^{\hbar})^{-1}\circ  \nabla_{\mathrm{NAH}}^{\hbar} \circ  \mathcal{A}^{\hbar}  - \nabla^{\hbar}_{\mathrm{GM}} \right)^{(1,0)} \!\!= \!\begin{pmatrix} 0 &*_{(1,2)} &  *_{(1,3)} \\ 0 & *_{(2,2)}& *_{(2,3)} \\ 0 & 0 & *_{(3,3)} \end{pmatrix}\!\!,
\end{equation}
where
\begin{eqnarray*}
    *_{(2,2)}&:=&\frac{1}{3\hbar} b_2 + \Gamma^{u}_{uu} - K_u \\
    *_{(3,3)}&:=&  \frac{2}{3\hbar} b_2 + 2 \Gamma^{u}_{uu} - 2 K_u\\
   *_{(1,2)}&:=& \frac{1}{2\hbar} b_1 - \frac{1}{9\hbar} b_2^2 - \frac{1}{6} \del_u b_2 + \hbar \Gamma^{u}_{uu} K_u - \hbar\frac{1}{2} K_u^2 - \hbar\del_u K_u \\
   *_{(2,3)} &:=& \frac{1}{2\hbar} b_1- \frac{1}{6} \del_u b_2 - \frac{1}{3} b_2 K_u - \frac{\hbar}{2} K_u^2 + \frac{1}{3} \Gamma^u_{uu} b_2 + \hbar\Gamma^u_{uu} K_u - \hbar\del_u K_u \\
   *_{(1,3)} &=&\frac{1}{\hbar}b_0 - \frac{1}{6\hbar} b_1 b_2 + \frac{1}{27\hbar} b_2^3 - \frac{1}{2} \del_u b_1 + \frac{5}{18}b_2 \del_u b_2 + \frac{\hbar}{6} \del_{u}^2 b_2 + b_1 \Gamma^u_{uu} - \frac{2}{9}b_2^2 \Gamma^u_{uu} \\ 
   & &- \frac{\hbar}{3}\del_u b_2 \Gamma^u_{uu} - b_1 K_u + \frac{2}{9} b_2^2 K_u + \frac{\hbar}{3} \del_u b_2 K_u - \frac{\hbar}{3} b_2 \Gamma^u_{uu} K_u + \frac{\hbar}{6} b_2 K_u^2 + \frac{\hbar}{3} b_2 \del_u K_u.
\end{eqnarray*}
These five entries vanish because of the three differential ring relations.
Indeed, the $(2, 2)$-entry and $(3,3)$-entry vanish because $\Gamma^{u}_{uu} - K_u = - \frac{1}{3\hbar} b_2$. Similarly the $(1,2)$-entry vanishes because
\begin{equation*}
    *_{(1,2)} \overset{\eqref{eq:diffring2}}{=} \hbar(-K_u^2 + \Gamma_{uu}^u K_u) + \frac{1}{3} b_2 K_u \overset{\eqref{eq:diffring1}}{=} 0.
\end{equation*}
The $(2,3)$-entry vanishes because of \eqref{eq:diffring1} and \eqref{eq:diffring2}.
The $(1,3)$-vanishes because
\begin{equation*}
    *_{(2,3)} \overset{\eqref{eq:diffring1}, \eqref{eq:diffring2}}{=} \frac{b_0}{\hbar} - \frac{1}{3\hbar} b_1 b_2 + \frac{2}{27\hbar} b_2^3 - \frac{1}{2} \del_u b_1 + \frac{1}{3} b_2 \del_u b_2 + \frac{\hbar}{6} \del^2_{u} b_2  \overset{\eqref{eq:diffring3}}{=} 0.
\end{equation*}
\end{proof}

Having studied the relation between the non--abelian Hodge flat connection  $\nabla_{\mathrm{NAH}}^{\hbar}$ and the oper $\nabla_{\mathrm{GM}}^{\hbar}$, we will now study 
the relation between the one--parameter family of non--abelian Hodge flat connections $\nabla_{\mathrm{NAH}}^{\zeta}$ and the oper $\nabla_{\mathrm{GM}}^{\hbar}$.

Following \cite{Gaiotto:2014bza,Dumitrescu2016a}, we can extend the one--parameter family of $\nabla_{\mathrm{NAH}}^{\zeta}$ to a two--parameter family
\begin{equation} \label{eq:two-param-family}
  \nabla^{\zeta,R}_{\mathrm{NAH}} = \zeta^{-1} R \varphi  + D(h_R) + \zeta R \varphi^{\dagger_{h_R}} ,
\end{equation}
parameterized by $\zeta \in \C^\times$ and $R \in \R^+$.
The Hermitian metric $h_R$ solves the $R$-rescaled Hitchin's equations
\begin{equation}\label{eq:RrescaledHitchin}
F_{D(h_R)}
= -R^2\left[ \varphi, \varphi^{\dagger_{h_R}}\right]\,, \quad D_{\bar{u}} \varphi=0.\,
\end{equation}
The \emph{$\lambda$-conformal limit} of $\nabla_{\mathrm{NAH}}^{\zeta,R}$ is defined by taking $\zeta \to 0 $ and $R\to 0$ but fixing the ratio $\zeta/R$ to be equal to $\lambda\in \C^\times$, in particular $\lambda R = \zeta$. 
Gaiotto conjectured \cite{Gaiotto:2014bza} that $\lim_{R \to 0}\nabla_{\mathrm{\mathrm{NAH}}}^{\lambda R, R}$ is defined and is an oper. 
In this particular case, Proposition \ref{prop:gauge} will imply that the limit is gauge equivalent to 
\begin{equation}
\nabla^{\hbar}_{\mathrm{GM}} = \de +\frac{1}{\hbar}\left( \begin{array}{ccc} 0&0&- b_0 \\
1&0&- b_1\\
0&1&- b_2
\end{array}\right) \de u \,, 
\end{equation}
an oper (cf. Example \ref{ex:opers}) determined by the $\hbar$-deformed Picard--Fuchs equation
\eqref{eq:PFK3hbar}. 

\begin{cor}
Consider the two-parameter family of non--abelian Hodge flat connections $$\nabla^{\zeta,R}_{\mathrm{NAH}}=\frac{R}{\zeta} \varphi + D_R + R \zeta \varphi^{\dagger_{h_R}},$$
where $D_R$ is the Chern connection for the pair $(E, h_R)$.
Then its $\lambda$-conformal limit exists for any $\lambda \in \C^\times$ and is given by
\begin{equation*}
\lim_{R\rightarrow 0} \nabla^{\lambda R,R}_{\mathrm{NAH}}=\nabla^{\lambda}_{\mathrm{NAH}}= \frac{1}{\lambda} \varphi  + D_{h} + \lambda \varphi^{\dagger_{h}},  \qquad h=h_1\,.
\end{equation*}
Moreover,
$g_{\lambda, \hbar}^{-1} \circ \nabla^{\lambda}_{\mathrm{NAH}} \circ g_{\lambda, \hbar} = \nabla^{\hbar}_{\mathrm{NAH}}= \mathcal{A}^\hbar \circ \nabla_{\mathrm{GM}}^\hbar \circ (\mathcal{A}^\hbar)^{-1}$ where
\begin{equation}
 g_{\lambda, \hbar}= \begin{pmatrix} 1 & 0 & 0 \\ 0 & \hbar \lambda^{-1} & 0 \\0 & 0 & \hbar^2\lambda^{-2} \end{pmatrix},\,
\end{equation}
i.e. the $\lambda$-conformal limit is gauge equivalent to $\nabla_{GM}^\hbar$ for any $\hbar\in \C^\times$.
\end{cor}

\begin{proof}
This claim essentially follows from the fact that the Higgs bundle $(E, \xi\,\varphi)$ is isomorphic to $(E, \varphi)$ for any $\xi \in \C^\times$.
More precisely, let $h_R$ be the Hermitian metric which solves the $R$-rescaled Hitchin's equations in \eqref{eq:RrescaledHitchin}.
In the frame $\omega^\hbar_{\mathrm{NAH}}$ it is given by
\begin{equation}
    h_R = \begin{pmatrix} \e^{-K} & 0 & 0 \\
    0 & |R|^{-2}\e^{-K} G_{u\ubar} \de u \de \ubar & 0 \\ 0 & 0 & |R|^{-4}\e^{-K} G_{u \ubar} G_{u \ubar} \de u^2 \de \ubar ^2 \end{pmatrix}.
\end{equation}
Now define
\begin{equation}
    g_R = \begin{pmatrix} 1 & 0 & 0 \\ 0 & R^{-1} & 0 \\ 0 & 0 & R^{-2}\end{pmatrix}
\end{equation}
with respect to the direct sum decomposition of $\mathcal{E}$.
Then we check that
\begin{equation}
    R\varphi = g_R^{-1} \varphi g_R, \quad h_R = g_R^\dagger h_1 g_R.
\end{equation}
Hence it follows that $\lambda R^2 \varphi^{\dagger_{h_R}}= \lambda \varphi^{\dagger_{h_1}}$ and similarly $D(h_R) = D(h_1)$. 
Consequently, $\nabla_{\mathrm{NAH}}^{\lambda R, R}$ is independent of $R$ and the $\lambda$-conformal limit is equal to
\begin{equation}
\nabla^{\lambda}_{\mathrm{NAH}}= \frac{1}{\lambda} \varphi  + D_{h} + \lambda \varphi^{\dagger_{h}},
\end{equation}
for $h=h_1$.
One can immediately check that $g_{\lambda, \hbar}^{-1} \circ \nabla^{\lambda}_{\mathrm{NAH}} \circ g_{\lambda, \hbar} = \nabla^{\hbar}_{\mathrm{NAH}}$.
\end{proof}

\subsection{Quartic example}\footnote{Here and in the following concrete computations, we do not consider the $\hbar$- or $\zeta$-dependence of the solutions which is the starting point of their exact WKB analysis (see e.g. \cite{Hollands:2019wbr}). This will be discussed elsewhere.}
In the following we will provide the geometric data of the VHS $\mathcal{H}$ of weight $2$ attached to the mirror quartic family as a specific example of the previous discussion. We postpone the definition of the mirror quartic and the discussion of how to obtain the associated Picard--Fuchs equation to \ref{ex:mirrorquartic} in \S\ref{examples} because we will need some notions of mirror constructions to do this. The moduli space in this case is $\mathcal{B}= \mathbb{P}^{1}\setminus\left\{0,1,\infty\right\}$. The Picard--Fuchs equation is given in terms of a local coordinate $z\in \mathcal{B}$ centered around 0 as\footnote{We will use $z$ throughout this paper as a local coordinate whenever it is centered around a regular singular point or when it is obtained from the toric data of the mirror symmetry constructions of \S \ref{examples}. We will denote by $u$ either the local coordinate of generic point in the base manifold or, if the context is clear, some explicit coordinate system.}

\begin{equation}\label{PFquartic}
L_{PF}= \theta^3- z \prod_{i=1}^3\,(\theta +i/4) \,,\quad \theta=z\frac{d}{dz}\,,
\end{equation}
the discriminant of this operator is:
\begin{equation}
\Delta= 1-z\,.
\end{equation}
And the Griffiths--Yukawa coupling can be computed to be:
\begin{equation}
c_{zz} = \frac{\kappa}{z^2 \Delta}\,,
\end{equation}
with $\kappa$ an integration constant which we set to 1.

The solutions of the Picard--Fuchs (PF) equation, corresponding to the integrals of the holomorphic form $\omega_0$ over a basis of integral cycles, i.e. periods, are given by \cite{Alim:2014vea}:
\begin{eqnarray}
\pi^0 &=&{}_3 F_2\left( \frac{1}{4},\frac{1}{2},\frac{3}{4};1,1, z\right)\,,\\
\pi^1&=& i\frac{1}{2\pi^{3/2} \Gamma(1/4)\Gamma(3/4)}\,G^{2\,3}_{3\,3} \left(  \begin{array}{ccc|} \frac{1}{4} &\frac{1}{2} & \frac{3}{4}\\ 0&0&0 \end{array}  \,\,z \right)\,,\\
\pi^2&=& \frac{1}{2} (\pi^1)^2/\pi^0\,.
\end{eqnarray}
Here $\pi^0,\pi^1$ are given in terms of the hypergeometric functions $_3 F_2$ and Meijer G-functions respectively. We define:
\begin{equation}
\tau=\frac{\pi^1}{\pi^0}\, , \quad q=e^{2\pi i \tau}\,,
\end{equation}

The integrality and modular properties of the inverse mirror map $z(q)$ have been addressed in \cite{LianYau}. The periods $\pi^0,\pi^1$, mirror map and K\"ahler potential of this example can be expressed in terms of the differential ring of quasi-modular forms associated to the congruence subgroup $\Gamma_0(2)$ of $SL(2,\mathbbm{Z})$. A careful study of the Picard-Fuchs operator and of the monodromy of its solutions reveals that the monodromy group in this case is $\Gamma_0(2)_{+}$ \cite{LianYau,Hosono:2000eb}.
We will use the quasi-modular forms, reviewed in Appendix~\ref{modularappendix}:
\begin{eqnarray}
A(\tau) &=& (\theta_2(\tau) + \theta_3(\tau))^{1/2}\,,\\
B(\tau) &=& \theta_{4}^{2}(2\tau)\, \\
C(\tau) &=& \frac{1}{\sqrt{2}} \theta_{2}^{2}(\tau) \, \\
E(\tau) &=&\frac{1}{3}\left(2 E_2(2\tau) + E_2(\tau) \right) \,,
\end{eqnarray}
which obey the algebraic relation
$$ A^4=B^4+C^4\,.$$
We have moreover the differential ring relations for $\Gamma_0(2)$, given in Appendix \ref{modularappendix}:

\begin{align}\label{quarticdiffring}
\partial_\tau A &=\frac{1}{8} A(E + \frac{A^4-2B^4}{A^2})\,,\\ \nonumber
\partial_\tau B&= \frac{1}{8} B(E-A^2)\,,\\ \nonumber
\partial_\tau E&=\frac{1}{8}(E^2-A^4)\,.
\end{align}

We find an expression for the inverse mirror map
\begin{equation}
z(\tau)= \frac{4 B(\tau)^4 (A(\tau)^4 - B(\tau)^4)}{ A(\tau)^8}\,, 
\end{equation}
and moreover 
\begin{equation}
\pi^0(\tau) = A^2(\tau)\,.
\end{equation}
and we obtain for the K\"ahler potential and metric 
\begin{equation}
e^{-K}= 2 |\pi^0|^2 (\mathrm{Im} \tau)^2 = 2 A(\tau)^2 \, A(\bar{\tau})^2 (\mathrm{Im} \tau)^2\,, \quad G = \frac{1}{4 (\Im \tau)^2} \de \tau \de \bar{\tau}\,.
\end{equation}
We note that $G$ is the Poincar\'e metric on $\mathbb{H}/\Gamma_0(2)^+$ which is isomorphic to $\mathcal{B}$.

For the coefficients of the Picard--Fuchs equation for the holomorphic form $\omega_0$:
\begin{equation}
\nabla^3_{\frac{\partial}{\partial z}} \omega_0 =  -b_2(z) \nabla^2_{\frac{\partial}{\partial z}} \omega_0-b_1(z) \nabla_{ \frac{\partial}{\partial z}} \omega_0  -b_0(z) \omega_0   \,,
\end{equation}
we obtain from \eqref{PFquartic}:
\begin{equation}
b_0(z)=\frac{3}{32 (z-1) z^2} \,,\quad b_1(z)= \frac{51 z-16}{16 (z-1) z^2}\,,\quad b_2(z)=\frac{6-9 z}{2 z-2 z^2}\,.
\end{equation}
We can now verify the differential ring equations \ref{quarticdiffring} and map these using the modular expressions for $z, \pi^0$ and $e^{-K}$ to the differential ring relations of the quasi--modular forms of $\Gamma_0(2)$. 

Finally, we turn to the parabolic Higgs bundle $(\hat{E},\hat{\varphi})$ on $\hat{\mathcal{B}}=\mathbb{CP}^1$ with divisor $D=\{0,1,\infty\}$ defined by the VHS $\VH$ of weight $2$.
As reviewed in detail in \S\ref{s:VHSHiggs}, the filtered holomorphic bundle $(\VH_{\Oo},F^\bullet \VH_{\Oo})$ extends to a filtered holomorphic bundle $(\hat{\VH}_{\Oo}, \hat{F}^\bullet \hat{\VH}_{\Oo})$. 
Likewise the holomorphic Gau\ss--Manin connection $\nabla$ extends to a logarithmic connection $\hat{\nabla}$ on $\hat{\VH}_{\Oo}$ with logarithmic poles along $D$.

Define the holomorphic bundle $\hat{E}:=\hat{E}^{2,0}\oplus \hat{E}^{1,1}\oplus \hat{E}^{0,2}$ on $\hat{\mathcal{B}}$ for $\hat{E}^{p,q}=\hat{F}^{p}/\hat{F}^{p+1}$.
Together with Griffiths' transversality, $\hat{\nabla}$ induces the logarithmic Higgs field $\hat{\varphi}=\oplus_{p=0}^2 \hat{\varphi}^p$ with components \begin{equation*}
\hat{\varphi}^p:\hat{E}^{2-p,p}\to \hat{E}^{2-p-1,p+1}\otimes \Omega^1(\hat{\mathcal{B}},\log(D)).
\end{equation*}
As an application of Theorem \ref{thm:degEpq}, we make the bundle $\hat{E}$ explicit for this example.
\begin{ex}[Mirror quartic]
The exponents of the Picard--Fuchs equations at the points $d\in D=\{0,1,\infty\}$ (see \eqref{eq:indicial} below for a definition) are given by
\begin{center}
\begin{tabular}{l | l  l  l}
d & 0 & 1 & $\infty$ \\ \hline
$\mu_1^d$ & 0 & 0 & $\frac{1}{4}$ \\
$\mu_2^d$ & 0 & $\frac{1}{2}$ & $\frac{1}{2}$ \\
$\mu_3^d$ & 0 & 1 & $\frac{3}{4}$
\end{tabular}
\end{center}
Then the induced parabolic Higgs bundle $(\hat{E},\hat{\varphi})$ satisfies
\begin{equation*}
\hat{E}=\hat{E}^{2,0}\oplus \hat{E}^{1,1} \oplus \hat{E}^{0,2}\cong \Oo_{\mathbb{P}^1}\oplus \Oo_{\mathbb{P}^1}(-1) \oplus \Oo_{\mathbb{P}^1}(-1). 
\end{equation*}
This follows from Theorem \ref{thm:degEpq} and the classification of line bundles over $\mathbb{P}^1$.
Hence the logarithmic Higgs field $\hat{\varphi}$ has components 
\begin{gather*}
    \hat{\varphi}^0\in H^0(\mathbb{P}^1,\mathcal{O}_{\mathbb{P}^1})\cong \C, \\
    \hat{\varphi}^1\in H^0(\mathbb{P}^1, \mathcal{O}_{\mathbb{P}^1}(1))
\end{gather*}
since $\Omega^1(\mathbb{P}^1,\log(D))=\mathcal{O}_{\mathbb{P}^1}(1)$.
\end{ex}



\section{Variations of Hodge structures and their relation to opers}\label{s:VHStt}
The aim of this section is twofold. 
Firstly, we motivate and review variations of Hodge structures.
Secondly, we show that variations of Hodge structures over a Riemann surface $\cu$ with a so-called generic cyclic vector are equivalent to opers. 
 

\subsection{Variations of Hodge structures}\label{ss:VHS}

We begin with a single projective manifold $X\hookrightarrow \mathbb{CP}^N$ of $\dim_{\C}(X)=n$ with K\"ahler class $\omega \in H^2(X,\Z)$. 
For every $0\leq k \leq n$, the cohomology groups $H^k(X,\C)$ admit the Hodge decomposition 
\begin{equation}\label{eq:hodgedecomp}
H^k(X,\C)=\bigoplus_{p+q=k} H^{p,q}(X),\quad H^{p,q}(X)=\overline{H^{q,p}}(X).
\end{equation}

It is equivalent to the Hodge filtration $F^\bullet H^k(X,\C)$ defined by 
\begin{equation*}
F^pH^k(X,\C)=\bigoplus_{l\geq p} H^{l,k-l}(X),\quad F^p\cap \bar{F}^q=0 \mbox{ if }p+q=k+1,
\end{equation*}
via $H^{p,q}(X)=F^pH^k(X,\C)\cap \bar{F}^qH^k(X,\C)$. 
The group $H^k(X,\Z)$ carries an additional structure, namely the bilinear form $Q:H^k(X,\Z)\otimes_{\Z} H^k(X,\Z)\to \Z$, 
\begin{equation}\label{eq:intersectionform}
Q(\alpha,\beta)=(-1)^{\tfrac{1}{2}k(k-1)}~\int_X \alpha\wedge \beta \wedge \omega^{\wedge n-k}.
\end{equation}
It is symmetric if $k$ is even and skew-symmetric if $k$ is odd. 
To state the Riemann-Hodge bilinear relations (\cite{Huybrechts}) satisfied by $Q$, we introduce the Weil operator $W\in \mathrm{End}(H^k(X,\C))$ defined by
\begin{equation}\label{eq:WeilOp}
\left.W\right|_{H^{p,q}}=i^{p-q}. 
\end{equation}
The Riemann-Hodge bilinear relations are then given by
\begin{enumerate}[label=\Roman*)]
\item \label{eq:RHi}
$Q(H^{p,q},H^{r,s})=0$, $(r,s)\neq (q,p)$, or equivalently $Q(F^l, F^{k-l+1})=0$, 
\item \label{eq:RHii}
for every non-zero primitive cohomology class\footnote{Recall that a cohomology class $\alpha\in H^k(X,\Q)$ is primitive if $\alpha\wedge \omega^{\wedge n-k+1}=0$.}   $\alpha\in H^k_{\mathrm{\mathrm{prim}}}(X,\Q)$, 
\begin{equation}\label{eq:hdgmetric} 
h(\alpha,\alpha):=Q(W\alpha,\bar{\alpha})>0. 
\end{equation} 
\end{enumerate}
Abstracting these properties yields the following 

\begin{dfn}
An integral Hodge structure ($\Z$-Hodge structure) of weight $k$ is a pair $(H_{\Z}, F^\bullet H_{\C})$ consisting of a free abelian group $H_{\Z}$ of finite rank and a decreasing filtration $F^\bullet H_{\C}$ of $H_{\C}=H_{\Z}\otimes \C$ such that $F^pH_{\C}\cap \bar{F}^qH_{\C}=0$ if $p+q=k+1$. \\
A polarization on $(H_{\Z},F^\bullet H_{\C})$ is a bilinear map $Q:H_{\Z}\otimes H_{\Z}\to \Z$ such that \ref{eq:RHi} and \ref{eq:RHii} are satisfied. 
The Hermitian metric $h$ on defined by \eqref{eq:hdgmetric} is called Hodge metric. 
\end{dfn}
The notions of a rational or real Hodge structure ($\Q$-/$\R$-Hodge structure) and polarizations are defined analogously by replacing $\Z$ with $\Q$ or $\R$.

\begin{ex}
Using the Lefschetz decomposition of $H^k(X,\C)$, it is possible to construct a polarization $Q$ on all of $H^k(X,\C)$ and not just on the primitive part $H^k_{\mathrm{prim}}(X,\Z)$. 
However, the polarization is no longer defined over $\Z$ in general because the Lefschetz decomposition is only defined over $\Q$. 
Hence $H^k(X,\Q)$ is a polarizable rational Hodge structure of weight $k$. 
\end{ex}

\begin{ex}[Hodge structures of K3 surfaces]\label{ex:hodgeK3}
	Let $X$ be an algebraic K3 surface, i.e. a compact connected K\"ahler surface $X$ such that $\Omega_X^2\cong \Oo_X$, $H^{0,1}(X)=0=H^{1,0}(X)$ and $X$ admits an integral K\"ahler class $\omega\in H^2(X,\Z)$.
	Then $H^2(X,\Z)$ together with the intersection form \eqref{eq:intersectionform} is a polarized $\Z$-Hodge structure of weight $2$. 
	It has the property that 
	\begin{equation*}
	h^{0,0}=h^{2,2}=1, \quad h^{1,0}=h^{0,1}=0, \quad h^{2,0}=h^{0,2}=1, \quad h^{1,1}=20
	\end{equation*}
	and all other $h^{p,q}:=\dim_\C H^{p,q}(X)$ are zero. 
\end{ex}

The previous discussion works in the family case as well. 
More precisely, let $\pi: \mathcal{X}\to \mathcal{B}$ be a family of projective manifolds of dimension $n$ over the complex manifold $\mathcal{B}$. 
Then the polarized integral Hodge structures $(H^k_{\mathrm{prim}}(X_b,\Z), F^\bullet H^k_{\mathrm{prim}}(X_b,\C),Q_b)$ of weight $k$ vary nicely over $\mathcal{B}$ and determine a polarized integral variation of Hodge structures (\cite{GriffithsI}): 
\begin{dfn}
Let $\mathcal{B}$ be any complex manifold. 
An \emph{integral variation of Hodge structures ($\Z$-VHS) of weight $k$} is a tuple $\VH=(\VH_{\Z}, F^\bullet \VH_{\Oo})$ consisting of 
\begin{itemize}
\item 
a locally constant sheaf $\VH_{\Z}$ of free abelian groups of finite rank, 
\item 
a decreasing filtration $F^\bullet \VH_{\Oo}$ of the associated holomorphic bundle $\VH_{\Oo}=\VH_{\Z}\otimes \Oo_B$ in holomorphic subbundles.  
\end{itemize}
These are subject to the conditions 
\begin{enumerate}[label=\roman*)]
\item 
the fibers $\VH_{b}=(\VH_{\Z,b}, F^\bullet \VH_{\Oo,b})$ form an integral Hodge structure of weight $k$, and
\item
the filtration $F^\bullet \VH_{\Oo}$ satisfies Griffiths transversality with respect to the holomorphic Gau\ss--Manin connection $\nabla$: 
\begin{equation*}
\nabla F^p \VH_{\Oo}\subset F^{p-1} \VH_{\Oo}\otimes \Omega^1_B. 
\end{equation*}
\end{enumerate}
\medskip

A \emph{polarization} of a $\Z$-VHS $\VH$ is a morphism $Q:\VH_{\Z}\otimes \VH_{\Z}\to \underline{\Z}_{\mathcal{B}}$, for the constant sheaf $\underline{\Z}_\mathcal{B}$, such that its restriction $Q_b$ to the fiber over $b$ is a polarization of the integral Hodge structure $\VH_b$. 
A $\Z$-VHS together with a polarization $Q$ is called a \emph{polarized $\Z$-VHS}. 
It is called \emph{polarizable} if it admits a polarization. 
\end{dfn}
\begin{rem}
	By working over $R=\Q,\R$ instead of $\Z$, we obtain the notion of an $R$-VHS. 
	A polarization of an $R$-VHS is defined analogously. 
	In \S\ref{s:VHSHiggs} we further review complex VHS. 
	Clearly, every polarized $\Z$-VHS $(\VH,Q)$ on $\mathcal{B}$ determines a polarized $R$-VHS by tensoring with the constant sheaf $\underline{R}_{\mathcal{B}}$ for $R=\Q,\R$.
	
    All of the following result in this section hold true for $R$-VHS for $R=\Z,\Q,\R$.
    We concentrate on $\Z$-VHS because these arise from our geometric examples. 
\end{rem}

Given a polarized $\Z$-VHS $(\VH,Q)$ of weight $w$ with Hodge filtration $F^\bullet \VH_{\Oo}$, we define the Hodge bundles
\begin{equation}\label{eq:hdgbdls}
\VH^{p,q}= F^p\VH_{\Oo} \cap \bar{F}^q\VH_{\Oo}\cong F^p\VH_{\Oo}/F^{p+1}\VH_{\Oo}. 
\end{equation}
The last isomorphism is only a $C^\infty$-isomorphism. 
However, the bundles on the right-hand side of  (\ref{eq:hdgbdls}) are holomorphic bundles.
Hence each $\VH^{p,q}$ is naturally endowed with a holomorphic structure. 
The smooth splitting 
\begin{equation}\label{eq:hpqbdl} 
\VH_{\Oo}=\bigoplus_{p+q=w} \VH^{p,q}
\end{equation} 
therefore endows $\VH_{\Oo}$ with another holomorphic structure.
For better distinction, we denote the resulting holomorphic bundle as
\begin{equation}\label{eq:epq}
E=\bigoplus_{p+q=w} E^{p,q}. 
\end{equation}
As for a single polarized Hodge structure, the polarization $Q$ induces the Hodge metric $h(v,w)=Q(W(v),w)$ on $E$ and the decomposition (\ref{eq:epq}) is orthogonal with respect to $h$. 
 
Besides the Hodge metric, we further obtain the non-degenerate pairing
\begin{equation*}
\eta(v,w)=Q(W(v),w). 
\end{equation*}
It is related to the Hodge metric $h$ by $h(v,w)=\eta(v,\tau(w))$ for the complex conjugation $\tau(v)=\bar{v}$ with respect to $\VH_{\Z}$. 
The tuple $(\VH,h,\eta)$ is an example of a $tt^*$-geometry of Cecotti--Vafa \cite{Cecotti:1991vb} (see \cite{Hertlingtt} for a mathematical account).


In the following, we concentrate on any $\Z$-VHS over a not necessarily compact Riemann surface $\cu$, generalizing the setup of Section \ref{ss:vhsintro}.

\subsection{From \texorpdfstring{$\Z$}{Z}-VHS to opers and back} 
We next explain how VHS on a Riemann surface $C$ are related to ($GL(r,\C)$-)opers (\cite{beilinson2005opers}).

\begin{dfn}
A $(GL(r,\C))$-\emph{oper} over the Riemann surface $C$ is a pair $(F^\bullet \mathcal{V}, \nabla)$ consisting of 
\begin{itemize} 
\item 
a holomorphic bundle $\mathcal{V}$ of rank $r$ with a decreasing filtration 
\begin{equation*}
\mathcal{V}=F^0\mathcal{V}\supset \dots \supset F^{r-1}\mathcal{V}\supset F^{r}\mathcal{V}=0
\end{equation*}
such that $\mathrm{rk}(Gr_F^k\mathcal{V})=1$ for $Gr_F^k\mathcal{V}=F^k\mathcal{V}/F^{k+1}\mathcal{V}$ and all $0\leq k \leq r-1$, 
\item
a holomorphic connection $\nabla$ such that $F^\bullet\mathcal{V}$ satisfies Griffiths transversality
\begin{equation*}
\nabla F^k \mathcal{V} \subset F^{k-1} \mathcal{V}\otimes \Omega^1_C.
\end{equation*}
Moreover, the $\Oo_C$-linear morphisms $Gr^k_F\mathcal{V} \to Gr^{k-1}_F\mathcal{V}\otimes \Omega^1_C$ induced by $\nabla$ are isomorphisms for all $0 \leq k \leq r-1$.  
\end{itemize} 
We call such a filtration of $(\mathcal{V},\nabla)$ an \emph{oper filtration}. 
An oper $(F^\bullet \mathcal{V}, \nabla)$ is an $SL(r,\C)$-oper if $\nabla$ induces the trivial connection on $\det \mathcal{V}$. 
\end{dfn} 

\begin{ex}\label{ex:opers}
	Locally, for every oper $(\mathcal{V},\nabla)$ there is a frame of $\mathcal{V}$ such that the corresponding connection $1$-form of $\nabla$ has the form
	\begin{equation*}
	\begin{pmatrix} * & * & \cdots & \cdots & * \\
		+ & * & \cdots & \cdots & * \\
		0 & \ddots & \ddots & \ddots & \vdots \\
		\vdots & \ddots & \ddots & \ddots & \vdots \\
		0 & \cdots & 0 & + & * 
	\end{pmatrix}.
	\end{equation*}
	Here $+$ are nowhere vanishing entries and $*$ are arbitrary ones. 
	
	A global standard example is given as follows: 
	let $C$ be a compact Riemann surface of genus $\geq 2$ and $\mathcal{L}$ a spin bundle, i.e. $\mathcal{L}^2\cong K_C$.  
	Let 
	\begin{equation}\label{eq:extK}
	\begin{tikzcd}
	0 \ar[r] & \mathcal{L} \ar[r] & \mathcal{V} \ar[r] & \mathcal{L}^{*} \ar[r] & 0
	\end{tikzcd}
	\end{equation}
	be the non-trivial extension. 
	Let $g$ be a Riemannian metric in the conformal class of $C$. 
	Then the Levi-Civita connection of $g$ defines a holomorphic connection $\partial^{\mathcal{L}}$ on $\mathcal{L}$.
	With respect to the smooth splitting of \eqref{eq:extK}, define the holomorphic connection 
	\begin{equation*}
	\nabla=\begin{pmatrix}
	\partial^{\mathcal{L}} & 0 \\
	1 & \partial^{\mathcal{L}^*} 
	\end{pmatrix}.
	\end{equation*}
	Note that $1$ makes sense here because $K_C\otimes \mathrm{Hom}(\mathcal{L},\mathcal{L}^*)\cong \Oo_C$.
	Then $(\mathcal{V},\nabla)$ is an oper, in fact an $SL(2,\C)$-oper. 
\end{ex}

Opers are closely related to VHS.
The only missing datum is a compatible integral (or rational/real) structure, i.e. a locally constant sheaf $\mathcal{V}_{\Z}\subset \mathcal{V}$ of free abelian groups of rank $r$ such that $\mathcal{V}_{\Z}\otimes_{\Z}\Oo_C\cong \mathcal{V}$ and $\nabla$ coincides with the canonical connection on the left-hand side under this isomorphism. 
Moreover, we require that $F^\bullet$ is a Hodge filtration on $\mathcal{V}_{\Z}\otimes \Oo_C\cong \mathcal{V}$.
In this case, $(\mathcal{V}_{\Z}, F^\bullet \mathcal{V})$ is a $\Z$-VHS of weight $w=r-1$ and type $(1,\dots, 1)$, i.e. $rk(Gr^p_F \mathcal{V})=1$ for all $p\in \{0,\dots, r-1 \}$.
\begin{ex}[Families of elliptic curves]\label{ex:ell1}
	Let $\mathcal{X}\to C$ be a family of elliptic curves over the Riemann surface $C$ and let $(\VH_{\Z},F^\bullet\VH_{\Oo})$ be the induced polarizable $\Z$-VHS of weight $1$ over $C$. 
	It defines the period map $\mathcal{P}: C\to \mathbb{H}/\Gamma$ (see \cite[\S 4.5]{CarlsonEtAl2}) where $\Gamma\subset \mathrm{Aut}(\mathbb{H})$ is the monodromy group of the family $\mathcal{X}$.
	Here we have identified the period domain $\mathcal{D}$ for Hodge structure of weight $1$ and rank $2$, i.e. the space of all Hodge filtrations $F^1\subset H_{\C}\cong \C^2$, with $\mathbb{H}\subset \mathbb{CP}^1$. 
	
	The condition that $(F^\bullet \VH_{\Oo}, \nabla)$ is an oper is rephrased as a condition on $\mathcal{P}$ as follows.
	The tangent space $T_u \mathbb{H}$ to the period domain $\mathcal{D}=\mathbb{H}$ is canonically identified with $T_u \mathbb{H}\cong \mathrm{Hom}(F_u^1,F_u^0/F_u^1)$. 
	Then the derivative $d\mathcal{P}_u: T_u C \to T_u \mathbb{H}$ of the period map is identified with 
	\begin{equation}\label{eq:dP}
	T_u C \to \mathrm{Hom}(F_u^1,F_u^0/F_u^1),\quad v\mapsto (\alpha \mapsto \nabla_v \alpha \mbox{ mod } F_u^1),
	\end{equation}
	see  (\cite[Lemma 5.3.2.]{CarlsonEtAl2}).
	Hence $(F^\bullet \VH_{\Oo}, \nabla)$ is an oper if, and only if, $d\mathcal{P}_u$ is an isomorphism if, and only if,  $\mathcal{P}$ is a local isomorphism. 
	If the last condition is satisfied, then the family $\mathcal{X}\to C$ is called complete. 
\end{ex}

\begin{ex}[Families of K3 surfaces]\label{ex:K3}
	Let $(\VH_{\Z},F^\bullet \VH_{\Oo})$ be a $\Z$-VHS of weight $2$ which is determined by a family of algebraic K3 surfaces over a Riemann surface $\cu$. 
	In this case $\VH_{\Z}$ is always of rank $22$, compare Example \ref{ex:hodgeK3} and 
	\begin{equation*}
	\mathrm{rk}(Gr_F^2)=1, \quad \mathrm{rk}(Gr^1_F)=20, \quad \mathrm{rk}(Gr^0_F)=1. 
	\end{equation*}
	Therefore $(F^\bullet \VH_{\Oo}, \nabla)$ cannot be an oper for dimension reasons. 
	However, by working with complete families of $M$-polarized K3 surfaces for $\check{M}=\langle 2n \rangle$, we obtain integral variations of Hodge structures with a generic cyclic vector, cf. Section \ref{ss:hodgek3}. 
\end{ex}
\begin{rem}
The previous two examples show that complete families $\mathcal{X}\to C$ of elliptic curves and certain lattice polarized K3 surfaces determine an oper over $C$.
It seems plausible that an analogous statement is true for all complete families of Calabi--Yau $d$-folds even over higher-dimensional bases.
\end{rem}

To explain how to pass from $\Z$-VHS to opers, we need the notion of a (generic) cyclic vector: 
\begin{dfn}\label{dfn:cyclic}
Let $(\mathcal{V}, \nabla)$ be a holomorphic bundle of rank $r$ on the Riemann surface $C$ with a holomorphic connection $\nabla$. 
A \emph{cyclic vector} of $(\mathcal{V},\nabla)$ is a holomorphic section $\omega\in H^0(C,\mathcal{V})$ such that for each $u\in C$ and every holomorphic vector field $X$ with $X(u)\neq 0$, 
\begin{equation}
\nabla^k_{X} \omega, \quad k=0,\dots, r-1
\end{equation}
is a local frame of $\mathcal{X}$ around $u$. 
A \emph{generic cyclic vector} of $(\mathcal{V},\nabla)$ is a non-zero meromorphic section $\omega\in \mathcal{M}(C,\mathcal{V})$ such that $\omega_{|C'}$ is a cyclic vector where $C'=C-D_\omega$ is the complement of the divisor $D_\omega$ defined by $\omega$.
\end{dfn}

\begin{ex}
	In Example \ref{ex:ell1} we have seen that the filtered holomorphic bundle $(F^\bullet \VH_{\Oo},\nabla)$ with holomorphic connection determined by the family $\mathcal{X}\to C$ of elliptic curves is an oper if, and only if, the period map $\mathcal{P}:C\to \mathbb{H}/\Gamma$ is a local isomorphism. 
	This is in turn equivalent to the existence of a generic cyclic vector: 
	
	Let $\omega\in \mathcal{M}^0(C, F^1)$ be a non-zero meromorphic section and $C'=C-D_\omega\subset C$ the complement of its zeros and poles. 
	Fix $u\in C'$ and identify $F_u^0/F_u^1\cong H^{0,1}_u$.
	In the bases $\omega_u\in F^0_u=H^{1,0}_u$ and $\bar{\omega}_u\in H^{0,1}_u$ the homomorphism $d\mathcal{P}_u(v)\in \mathrm{Hom}(H^{1,0}_u,H^{0,1}_u)$ for $v\in T_u C$ is represented by 
	\begin{equation}
	\frac{Q(\omega_u, \nabla_v \omega_u)}{Q(\omega_u,\bar{\omega}_u)}\in \C.
	\end{equation}
	This is non-zero for $v\neq 0$ if, and only if,  $(\omega, \nabla_V \omega)$ is a local frame of  $\VH_{\Oo}$ around $u$ (for an holomorphic vector field $V$ extending $v$), i.e. if, and only if,  $\omega$ is a generic cyclic vector.
	
	A similar discussion holds true for a family $\mathcal{X}\to C$ of compact Calabi-Yau threefolds over a Riemann surface $C$ which is complete, i.e. the Kodaira-Spencer map $\kappa_u: T_u C\to H^1(X_u,T_{X_u})$ is an isomorphism for each $u\in C$.
	We refer \cite[(1.4)]{BryantGriffiths} for details. 
\end{ex}

Every generic cyclic vector $\omega\in \mathcal{M}(C,\mathcal{V})$ defines an oper filtration $F^\bullet_\omega$. 
Define $F^{w-l}_\omega$, $w=r-1$, as the smallest subbundle of $\mathcal{V}$ which contains 
\begin{equation*}
\nabla_X^k \omega, \quad k=0,\dots, l
\end{equation*}
for every local holomorphic vector field $X$. 
Since $\omega$ is a generic cyclic vector, $rk(F_\omega^{w-l})=l+1$ and $rk(Gr_{F_\omega}^k)=1$ for $k\in \{ 0,\dots, w\}$.

\begin{prop}\label{lem:cyclicoper}
	The filtration $F_\omega^\bullet$ of $(\mathcal{V},\nabla)$ is an oper filtration. 
	Conversely, if $(\VH_{\Oo},\nabla)$ carries an oper filtration $F^\bullet \VH_{\Oo}$, then $(\VH_{\Oo},\nabla)$ admits a generic cyclic vector $\omega\in \mathcal{M}(C,\mathcal{V})$ such that $F_\omega^\bullet=F^\bullet$. 
	
	In particular, if a $\Z$-VHS $(\VH_{\Z},F^\bullet \VH_{\Oo})$ of rank $r$ and weight $w=r-1$ admits a generic cyclic vector $\omega\in \mathcal{M}(C,F^{w})$,  then $(F^\bullet\VH_{\Oo},\nabla)$ is an oper with $F^\bullet=F^\bullet_\omega$. 
	Hence it is of type $(1,\dots, 1)$.
\end{prop}
\begin{proof}
	Let $C'=C-D$ be the complement of the pole and zero divisor $D$ of $\omega$.
	By the construction of $F_{\omega}^\bullet$ since $\omega$ is a cyclic vector, $(F_\omega^\bullet, \nabla)$ is an oper on $C'$. 
	
	Now let $d\in D$ and choose a local coordinate $z$ centered at $d$. 
	Then $\omega=z^k \omega'$ for a holomorphic section $\omega'$ with $\omega'(0)\neq 0$. 
	Let $s_i$ be a local flat frame of $\mathcal{V}$ around $d$. 
	Then the section $\omega'$ is given by 
	\begin{equation*}
	\omega'= \sum_{j=1}^r f_j s_j.
	\end{equation*}
	We denote by $\mu_i^d=v_d(f_i)$ the vanishing order of $f_i$ at $d$ and assume without loss of generality that $\mu_1^d\leq \mu_2^d \leq \dots \leq \mu_r^d$. 
	By Remark \ref{rem:exponents} below, we know that 
	\begin{equation}
	\mu_j^d=j-1, \quad 1\leq j \leq r.
	\end{equation} 
	This implies that the $k$-th derivatives $f^{(k)}_j$ satisfy 
	\begin{equation}
	f^{(k)}_j(0)\neq 0 \mbox{ for }k=j-1,\qquad   f^{(k)}_j(0)=0 \mbox{ for }k\geq j. 
	\end{equation}
	As a consequence,  
	\begin{equation}
	\nabla^k_{d/dz}\omega'=\sum_{j=1}^{k+1} f_j^{(k)}(0) s_j,\quad k=0,\dots,l
	\end{equation}
	is a basis of $F_\omega^{(r-1)-l}$ at $0$. 
	Therefore $\omega'$ is a cyclic vector around $d$ and hence $(F_\omega^\bullet \VH_{\Oo},\nabla)$ is an oper. 
	
	Conversely, if $(F^\bullet \VH_{\Oo},\nabla)$ is an oper, then there exists a non-zero meromorphic section $\omega\in \mathcal{M}(C,F^{r-1})$. 
	By the properties of an oper, $\omega$ is a generic cylic vector. 
	
	The last claim follows from the fact that the generic cyclic vector $\omega$ is a meromorphic section of $F^w$ and Griffiths' transversality. 
\end{proof}
Hence opers with a compatible integral structure are equivalent to $\Z$-VHS with a generic cyclic vector. 

\subsection{From \texorpdfstring{$\Z$}{Z}-VHS to Picard--Fuchs equations}\label{ss:FromVHStoPF}

We next recall the relation between polarizable $\Z$-VHS with a generic cyclic vector and Picard--Fuchs equations. 

Let $(\VH_{\Z},Q, F^\bullet)$ be a polarized $\Z$-VHS of weight $w=r-1$ on the punctured disk $\Delta^*\subset \Delta$. 
Then the monodromy $T$ around $0\in \Delta$ is quasi-unipotent by a result of Borel (\cite[Lemma 4.5]{Schmid})
We assume that there exists a cyclic vector $\omega \in H^0(\Delta^*, \Fc^{w})$.
This implies the existence of $a_j\in \Oo(\Delta^*)$ satisfying
\begin{equation}\label{eq:prePF}
\nabla_{d/dz}^r \omega + a_{r-1}(z) \nabla_{d/dz}^{r-1} \omega + \dots + a_0(z) \omega=0. 
\end{equation}

\begin{dfn}
The point $0\in \Delta^*$ is...
\begin{itemize}
	\item ... called a \emph{regular point} of $\nabla$ if $a_j$ extends to holomorphic functions at $0$;
	
	\item ... a \emph{regular singular point} of $\nabla$ if $b_j:=z^{j}a_j$ extends to a holomorphic function at $0$. 
\end{itemize}
\end{dfn}
The Gau\ss--Manin connection is known to be regular singular (\cite[Theorem 4.13]{Schmid}), i.e. the limiting point $0\in \Delta$ is either a regular point or a regular singular point. 

Assume that $0\in \Delta$ is a regular singular point so that 
\begin{equation}\label{eq:prePF2}
\nabla_{z d/dz}^r \omega + b_{r-1}(z) \nabla_{z d/dz}^{r-1} \omega + \dots + b_0(z) \omega=0
\end{equation}
for $b_j\in \Oo(\Delta)$.
If $\gamma\in H^0(\Delta^*, \VH_{\Z}^\vee)$ is a (multi-valued) section, then the (multi-valued) function $f:=Q( \gamma, \omega)$ satisfies the scalar differential equation
\begin{equation}\label{eq:PF}
L_{\mbox{\tiny{PF}}}f:=\theta^r f + b_{r-1}(z) \theta^{r-1} f+ \dots + b_0(z) f=0, \quad \theta= z \frac{d}{dz}. 
\end{equation}
It is called the Picard--Fuchs equation associated with $\omega$ and is an ordinary differential equation with a regular singularity at $0$. 
Its solutions, the \emph{periods}, form a local system which is denoted by $Sol(L_{\mbox{\tiny{PF}}})$. 
The next lemma is immediate. 
\begin{lem}\label{lem:LS}
Let $\omega\in H^0(\Delta^*,F^w)$ be a cyclic vector as before. 
Then the morphism
\begin{equation}\label{eq:sol}
\VH_{\Z} \mapsto Sol(L_{\mbox{\tiny{PF}}}), \quad \gamma\mapsto Q( \gamma, \omega) 
\end{equation}
is an isomorphism of local systems.
In particular, $\omega=\sum_{i=1}^r f_i s_i$ for a multi-valued frame $s_i$ of $\VH_{\Z}$ on $\Delta^*$ and $f_i$ corresponds to the $\gamma_j$ such that $Q(\gamma_j,s_i)=\delta_{ij}$ under the isomorphism \eqref{eq:sol}.
\end{lem}
\begin{rem}
If $\omega$ is not cyclic but only non-vanishing on $\Delta^*$, then (\ref{eq:sol}) is only a surjection. 
\end{rem}
In particular, if $T$ is the local quasi-unipotent monodromy of $\VH_{\Z}$ around a puncture $d\in D$, then the local monodromy of the periods is the dual $T^\vee$. 
If we represent $T$ as a matrix $A$ with respect to a basis, then $T^\vee$ corresponds to $(A^{-1})^t$ in the dual basis. 
Under the isomorphism $\VH_{\Z}\cong \VH_{\Z}^\vee$ induced by $Q$ we therefore identify $T$ with $T^\vee$. 
In particular, the eigenvalues $\lambda_i$ of $T$ do not only satisfy $\lambda_i\in U(1)\subset \C^\times$ but also $\lambda_1\cdots \lambda_r=1$. 

These eigenvalues are related to the Gau\ss--Manin connection: in the frame $\nabla^k_{z d/dz} \omega$, $k=0,\dots, r-1$, the connection $1$-form is given by 
\begin{equation}
B(z) \tfrac{dz}{z}:=\begin{pmatrix} 0 & 0 & 0 & -b_0(z) \\
1 & 0 & \ddots & -b_1(z) \\
\vdots & \ddots & \ddots & \vdots \\
0 & \cdots & 1 & -b_{r-1}(z) \end{pmatrix}  \tfrac{dz}{z},
\end{equation}
cf. Example \ref{ex:opers}.
It is known (\cite[1.17.2.]{Deligne}) that $\exp(2\pi i B(0))$ has the same eigenvalues $\lambda_j$ as $T$. 
The eigenvalues of $\exp(2\pi  i B(0))$ are in turn of the form $\lambda_j=\exp(2\pi i \mu_j)$ for the eigenvalues $\mu_j$ of $B(0)$.
These are the roots of the polynomial
\begin{equation}\label{eq:indicial}
p(X):=X^r+b_{r-1}(0) X^{r-1}+ \dots + b_0(0)\in \C[X]
\end{equation}
(and $B(0)$ is its companion matrix). 
Since $\lambda_j=\exp(2\pi i \mu_j)\in U(1)$, we must have $\mu_j\in \R$ and we choose the ordering $\mu_1\leq \dots \leq \mu_r$. 

On the other hand, $\mu_j$ are called \emph{exponents} of the Picard--Fuchs equation (\ref{eq:PF}).
They determine the structure of solutions to \eqref{eq:PF}.
Let $\mu_{i_1}< \dots < \mu_{i_s}$ be the pairwise distinct exponents where $\mu_{i_j}$ has multiplicity $m_{j}$. 
The Frobenius method (\cite{Frobenius}, \cite[\S 3]{CoddingtonLevinson}) shows that a basis of solutions to (\ref{eq:PF}) is given by multi-valued functions of the form
\begin{equation}\label{eq:basismultivalued}
\sum_{k=1}^{m_j} z^{\mu_{i_j}} (\log z)^{k-1} g^j_k(z), \quad j=1,\dots, s,
\end{equation}
for holomorphic functions $g_k^j$ on $\Delta$ with $g^j_k(0)\neq 0$. 

\begin{rem}\label{rem:exponents}
	If $0\in \Delta$ is a regular point of $\nabla$, then the monodromy $T$ is trivial and the $\Z$-VHS canonically extends to $0$. 
	In this case the exponents $\mu_j$ are defined as well and are given as follows. 
	If $k=v_0(\omega)$ is the order of $\omega$ at $0$, then $\mu_j=k+(j-1)$. 
\end{rem}

The previous discussion globalizes: 
let $C^\circ \subset C$ be the complement of a reduced divisor $D_s\subset C$  in a compact Riemann surface $C$. 
If $(\VH_{\Z},F^\bullet,Q)$ is a polarized $\Z$-VHS of weight $w=r-1$ on $C^\circ$ with a generic cyclic vector $\omega\in \mathcal{M}(C^\circ, F^w)$, then the exponents $\mu_1^u\leq \dots \leq \mu_r^u$ are defined for any $u\in C$ by the local discussion above. 
Note that $\mu_j^u$ are independent of a local chart around each $u\in C$ and only depend on $\omega$. 
Moreover, if $f\in \mathcal{M}(C)$ is a non-zero meromorphic function, then the exponents $\tilde{\mu}_j^u$ defined by $\tilde{\omega}=f\omega$ are given by $\tilde{\mu}_j^u=\mu_j^u+v_u(f)$ for the order $v_u(f)$ of $f$ at $u$. 

Therefore $\Z$-VHS together with a generic cylic vector determine local Picard-Fuchs equations and their exponents $\mu_j^u$.
The later will play a crucial role in the relation between $\Z$-VHS with a generic cyclic vector and (parabolic) Higgs bundles as we explain next. 

\section{From \texorpdfstring{$\Z$}{Z}-VHS to (parabolic) Higgs bundles}\label{s:VHSHiggs}
In the beginning of this section, we review the relationship between (complex) variations of Hodge structures over a compact Riemann surface $\cu$ and Higgs bundles $(E,\varphi)$ such that $(E,\varphi)$ is isomorphic to $(E,\zeta \varphi)$ for any $\zeta\in \C^{\times}$.
These are called systems of Hodge bundles, cf. \cite{SimpsonUbiquity} (also see \cite[Chapter 13]{CarlsonEtAl2}). 
Afterwards we explain how $\Z$-VHS determine parabolic Higgs bundles. 
Finally, we determine, in the presence of a generic cyclic vector, the degrees of the resulting parabolic Higgs bundles in terms of the exponents introduced in the last section. 

\subsection{From \texorpdfstring{$\Z$}{Z}-VHS to Higgs bundles with harmonic metric}
Let $\VH=(\VH_{\Z}, Q,  F^\bullet\VH_{\Oo})$ be a polarized $\Z$-VHS of weight $w$ over the Riemann surface $\cu$. 
The local system $\VH_{\Z}$ induces the flat \emph{smooth} Gau\ss--Manin connection $\nabla_{\C}$ on the smooth bundle $\VH_{sm}$ underlying $\VH_{\Oo}$. 
Griffiths' transversality and the smooth decomposition (\ref{eq:hpqbdl}) implies that 
\begin{equation}\label{eq:smGM}
 \nabla_{\C}= \mathsf{D}+ \varphi + \psi
\end{equation}
where\footnote{Here $\Omega^k_{sm}(\VH^{p,q})$ stands for smooth $k$-forms with values in the smooth bundle $\VH^{p,q}$ (dropping the subscript `sm' for $\VH^{p,q}$).}  $\mathsf{D}:\Omega^0_{sm}(\VH^{p,q})\to \Omega^1_{sm}(\VH^{p,q})$ is a connection preserving the $(p,q)$-types and
\begin{align*}
&\varphi:\Omega^0_{sm} (\VH^{p,q}) \to \Omega_{sm}^{1,0}(\VH^{p-1,q+1}), \\
&\psi: \Omega^0_{sm} (\VH^{p,q}) \to \Omega_{sm}^{0,1}(\VH^{p+1,q-1}).
\end{align*}
Note that the holomorphic Gau\ss--Manin connection $\nabla$ is just the $(1,0)$-part of $\nabla_{\C}$, 
\begin{equation*}
\nabla=\nabla_{\C}^{1,0}=\mathsf{D}^{1,0}+\varphi.
\end{equation*}
In particular, we recover the filtered holomorphic bundle $F^\bullet \VH_{\Oo}$ with holomorphic connection $\nabla$.
\begin{lem}\label{l:orth}
The connection $\mathsf{D}$ is the Chern connection for the Hodge metric $h$ and the direct sum holomorphic bundle $E=\bigoplus_{p+q=w} E^{p,q}$ (with respect to the holomorphic structure $\mathsf{D}^{0,1}$).
Moreover,
\begin{equation*}
h(\varphi(s_0),s_1)=h(s_0,\psi(s_1)) 
\end{equation*}
for all (local) sections $s_0,s_1$ of $E$ so that $\psi=\varphi^{\dagger_h}$. 
\end{lem}
\begin{proof}
First of all, the holomorphic structure on $E$ coincides with $\mathsf{D}^{0,1}$.
We prove that $\mathsf{D}$ is compatible with $h$. 
Let $s_0, s_1$ be local sections of $E^{p,q}$. 
Then we compute (the sign is determined by (\ref{eq:WeilOp}))
\begin{align*}
dh(s_0,s_1)&=\pm i ~ dQ(s_0,\bar{s}_1)\\
&=\pm i \left(Q(\nabla_{\C}s_0,\bar{s}_1)+Q(s_0,\nabla_{\C}\bar{s}_1)\right)    & \mbox{(flatness of $Q$)}\\
&=\pm i \left( Q(\mathsf{D}s_0,\bar{s}_1)+Q(s_0,\mathsf{D}\bar{s}_1)\right)  & \mbox{(by \eqref{eq:smGM})} \\ 
&=h(\mathsf{D}s_0,\bar{s}_1)+h(s_0,\mathsf{D}\bar{s}_1), 
\end{align*}
i.e. $D$ is $h$-unitary. Thus, $D$ is the Chern connection.

Now let $s_0$ and $s_1$ be (local) sections of $E^{p,q}$ and $E^{p',q'}$ respectively with $p\neq p'$ (hence $q\neq q'$). 
A similar computation as before together with $h(s_0,s_1)=0$ shows 
\begin{equation*}
0=h(\varphi (s_0),s_1)+h(s_0,\psi( s_1))
\end{equation*}
which concludes the proof. 
\end{proof}

The smooth Gau\ss--Manin connection and the previous lemma motivates a weaker notion than an $R$-VHS for $R=\Z,\Q,\R$: 
\begin{dfn}
	A \emph{complex variation of Hodge structure} ($\C$-VHS) of weight $w$ is a pair $(V,\nabla_{\C})$ consisting of a smooth bundle $V$ over $C$ with a decomposition 
	\begin{equation}\label{eq:Vpq}
	V=\oplus_{p+q=w} V^{p,q}
	\end{equation}
	and a flat connection $\nabla_{\C}$ admitting a decomposition as in \eqref{eq:smGM}. 
	A \emph{polarization} is a Hermitian metric $h$ on $V$ such that \eqref{eq:Vpq} is orthogonal with respect to $h$ and $(-1)^p h(s,s)>0$ for any non-zero (local) section $s$ of $V^{p,q}$.
\end{dfn}
In particular, every (polarized) $R$-VHS, $R=\Z,\Q,\R$ induces a (polarized) $\C$-VHS. 
The converse is false in general. 
There might not be an underlying locally constant sheaf of $R$-modules inducing a Hodge filtration. 

\bigskip

Every $\C$-VHS gives rise to a Higgs bundle:
\begin{lem}[\cite{SimpsonUbiquity}]\label{l:D}
Let $(V,\nabla_\C,h)$ be a polarized $\C$-VHS of weight $w$ with decomposition $V=\oplus_{p+q=w} V^{p,q}$ and $\nabla_{\C}=\mathsf{D}+\varphi+\varphi^{\dagger_h}$.
Then the pair 
\begin{equation}
(\mathcal{V}:=(V,\mathsf{D}^{0,1}), \varphi)
\end{equation}
is a Higgs bundle on $\cu$.
It is the \emph{system of Hodge bundles} associated to $\C$-VHS.
In particular, $(\mathcal{V},\varphi,h)$ is a harmonic Higgs bundle on $\cu$. 
It is $\C^\times$-invariant (up to isomorphism), i.e. $(\mathcal{V},\varphi)$ is isomorphic to $(\mathcal{V},\lambda \varphi)$ for each $\lambda\in \C^\times$. 
Conversely, every Higgs bundle with this property is of this form. 
\end{lem}
\begin{rem}
	If $C$ is compact, then every $\C^\times$-fixed Higgs bundle is induced by a complex VHS under the non-abelian Hodge correspondence. 
	We emphasize that such Higgs bundles are often referred to as VHS in the literature. 
	However, in our context it is crucial to distinguish between VHS and systems of Hodge bundles (as originally done by Simpson (\cite{SimpsonUbiquity})).
\end{rem}
\begin{proof}
Type considerations together with  the decomposition $\nabla_{\mathbb{C}} = D + \varphi + \varphi^{\dagger_h}$ in (\ref{eq:smGM}) imply that the flatness condition $\nabla_{\C}^2=0$ decomposes into the following equations
\begin{equation}\label{eq:ttlike}
\begin{gathered} 
(\mathsf{D}^{0,1})^2=0 =(\mathsf{D}^{1,0})^2, 
\\
\mathsf{D}\varphi=0=  \mathsf{D}\varphi^{\dagger_h},  
\\
F^{\mathsf{D}}+[\varphi,\varphi^{\dagger_h}]=0. 
\end{gathered}
\end{equation}
For example, $\mathsf{D}\varphi$ is the only summand in $\nabla_{\C}^2=0$ that maps 
\begin{equation*}
\VH^{p,q}\to \Omega_{sm}^1(\VH^{p-1,q+1})). 
\end{equation*}
and hence has to be zero. 
Its $(0,1)$-part gives $\mathsf{D}^{0,1}\varphi=0$ as claimed. 

The statement about the $\C^\times$-invariance follows by a standard argument using a decomposition of $\mathcal{V}$ into the generalized eigenspaces of $\varphi$. 
\end{proof}

\begin{rem}\label{rem:tthit}
The equations \eqref{eq:ttlike} are clearly equivalent to the $tt^*$-equations \eqref{tt-eqs} by setting $D=\mathsf{D}$ and $\mathsf{C}=\varphi+\varphi^{\dagger_h}$.
Note that the equations $[\mathsf{C}',\mathsf{C}']=[\varphi,\varphi]=0$, and analogously for $\mathsf{C}''=\varphi^{\dagger_h}$, are here trivially satisfied for type and dimension reasons. 
\end{rem}
The Higgs field $\varphi$ can be expressed in terms of the holomorphic Gau\ss--Manin connection: 
under the isomorphism (\ref{eq:hdgbdls}), its $(w-q,q)$-components (all other components are zero) correspond to 
\begin{gather}\label{eq:KS}
\begin{split}
\varphi^{q}\colon &\Fc^{w-q}/\Fc^{w-q+1} \to \Fc^{w-q+1}/\Fc^{w-q}\otimes \Omega^1_C ,\\
&s \, \, \mathrm{mod}\, \Fc^{w-q+1}\mapsto \nabla s \, \, \mathrm{mod}\, \Fc^{w-q}. 
\end{split}
\end{gather}

\subsection{Parabolic Higgs bundles from Deligne's canonical extension}\label{ss:deligne}
Let $D\subset \cu$ be a reduced divisor on the compact Riemann surface $\cu$ and $\cuo:=\cu-D$ be its complement. 
Further let $\VH=(\VH_{\Z}, Q, \Fc^\bullet)$ be a polarized integral VHS on $\cuo$. 
\begin{ex}
Such examples naturally arise from geometry. 
For example, let $\pi:X\to C$ be a compact elliptic surface. 
Then the middle cohomology groups $H^1(X_u,\Z)$ of the smooth fibers $X_u$, $u\in \cu$, define a polarized $\Z$-VHS of weight $1$ over the smooth locus $\cuo\subset \cu$ of $\pi$ and $D$ is the divisor of singular fibers. 
\end{ex}
We give one example of an extension of the VHS $\VH$ on $\cuo$ to $\cu$ as a filtered holomorphic bundle with logarithmic connection.
This is Deligne's canonical extension $(\hat{\VH},\hat{\nabla})$ (\cite[\S II, Proposition 5.4]{Deligne}) which eventually determines a parabolic Higgs bundle.  

Since an extension across $u\in D$ is a local question, it suffices to consider the case $\cuo=\Delta^*\subset \cu=\Delta$. 
Let
\begin{equation*}
e:\mathbb{H}\to \Delta^*,\quad e(\tau)=\exp(2\pi i \tau)
\end{equation*}
be the universal covering of $\Delta^*$ and $T\in \mathrm{Aut}(\VH_{\Z,z_0})$ be the monodromy of $\VH_{\Z}$ for a fixed $z_0\in \Delta^*$. 
Recall that $T$ is necessarily quasi-unipotent\footnote{All what follows works for $\C$-VHS if we assume that the monodromy around the punctures is quasi-unipotent. Since we are interested in geometric examples, we phrase everything in terms of $\Z$-VHS.}. 
Hence any multi-valued section $s$ of $\VH_{\Z}$ on $\Delta^*$ satisfies
\begin{equation}\label{eq:mon}
e^* s(\tau+1)=T(e^*s)(\tau). 
\end{equation}
If $T=T_s T_u$ is the Jordan decomposition of $T$ into its semisimple part $T_s$ and unipotent part $T_u$, then we define
\begin{gather*}
N=N_u+N_s, \\
N_u=\frac{1}{2\pi i} \log T_u, \quad N_s=\frac{1}{2\pi i} \log T_s. 
\end{gather*}
Here $\log T_s$ is the logarithm of $T_s$ which is determined by requiring that its eigenvalues $\nu$ satisfy
\begin{equation*}
-1< \nu \leq 0.
\end{equation*}
In particular, $\exp(2\pi i \nu)\in S^1\subset \C^\times$ are the eigenvalues of $T_s$. 
Since $\exp(2\pi i N)=T$, the section 
\begin{equation}\label{eq:shat}
\hat{s}(\tau):=\exp(-2\pi i N \tau) e^*s(\tau), \quad s\in \VH_{\Z}(\Delta^*)
\end{equation}
is invariant under monodromy and descends to $\Delta^*$. 
Sections of the form $\hat{s}$ define Deligne's canonical extension $\hat{\VH}$ on $\Delta$. 
Deriving equation (\ref{eq:shat}) implies that the holomorphic Gau\ss--Manin connection $\nabla$ extends to the logarithmic connection $\hat{\nabla}$ with residue
\begin{equation*}
\mathrm{res}_0(\hat{\nabla})=-N. 
\end{equation*}
In particular, the eigenvalues of the residue $-N$ lie in $[0,1)$ which uniquely determines $\hat{\VH}$. 
Finally, the holomorphic subbundles $\Fc^p\subset \VH_{\Oo}$ extend to holomorphic subbundles $\hat{\Fc}^p\subset \hat{\VH}_{\Oo}$  such that 
\begin{equation}
    \hat{\nabla}:\hat{F}^p \to \hat{F}^{p-1}\otimes \Omega^1_C(\log D). 
\end{equation}

This due to \cite[Theorem 4.13]{Schmid} for unipotent monodromy and was generalized by \cite[\S 2.5 (iii)]{Kollar} to quasi-unipotent monodromy.

The extension $(\hat{\VH},\hat{\nabla})$ of $(\VH,\nabla)$ determines a parabolic Higgs bundle on $\cu$ as follows. 
The underlying holomorphic bundle is
\begin{equation}\label{eq:Ehatpq}
\hat{E}=\ \bigoplus_{p+q=w} \hat{E}^{p,q},\quad \hat{E}^{p,q}:= \hat{\Fc}^p/\hat{\Fc}^{p+1},
\end{equation}
on $\cu$.
The logarithmic connection $\hat{\nabla}$ induces the maps 
\begin{equation}\label{eq:phihat}
\hat{\varphi}^{q}:\hat{E}^{w-q,q} \to \hat{E}^{w-q,q}\otimes \Omega^1_C(\log D)
\end{equation}
analogously constructed as in (\ref{eq:KS}). 
Hence they define the meromorphic Higgs field
\begin{equation*}
\hat{\varphi}:=\oplus_{q}\, \hat{\varphi}^{q}:\hat{E}\to \hat{E}\otimes \Omega^1_C(\log D). 
\end{equation*}
To construct the parabolic structure on $\hat{E}_d$ at $d\in D$, we only look at $D=\{0\} \subset \cu=\Delta$ to simplify notation. 
Let \footnote{This convention will become clear in when we consider the exponents of Picard--Fuchs equations, see in particular (\ref{eq:nuj}) and (\ref{eq:kj}).} 
\begin{equation*}
0 \leq -\nu_1 \leq -\nu_2 \leq \dots \leq -\nu_r < 1
\end{equation*}
be the eigenvalues of $N_s$.
Moreover, we let $1 \leq i_1< \dots < i_s \leq r$ be the indices of pairwise distinct eigenvalues with multiplicity $m_j$, $j=1,\dots, s$. 
For any $\nu=\nu_j$ the subspace
\begin{equation*}
\VH_{z_0}^\nu=\{ v\in \VH_{z_0}~|~ (\exp{2\pi i \nu}-T)^N v=0 \quad \mbox{for some }N\in \Z\}
\end{equation*}
is the generalized eigenspaces of $T$ for the eigenvalue $\exp(2\pi i \nu)$. 
The span of $\hat{s}(0)\in \hat{\VH}_0$ for $s\in \VH_{z_0}^\nu$ defines the subspace $\hat{\VH}_{0}^\nu$ for $\nu\in (-1,0]$ and we set
\begin{equation}\label{eq:hnu}
    \hat{\VH}_{0}(\nu) =\bigoplus_{\alpha \geq \nu} \hat{\VH}_{0}^\alpha.
\end{equation}

These subspaces define 
\begin{equation*}
\hat{E}_{0}(\nu)=\bigoplus_{p} \frac{\hat{\VH}_{0}(\nu)\cap \hat{F}_0^p}{\hat{\VH}_{0}(\nu) \cap \hat{F}^{p+1}_0}\subset \hat{E}_0
\end{equation*}
and hence determine the parabolic structure (see \eqref{eq:parabolicstr})
\begin{align}
 \hat{E}_0 = \hat{E}_{0}(-\nu_{i_1}) & \supsetneq \hat{E}_{0}(-\nu_{i_2})  \supsetneq \dots \supsetneq  \hat{E}_{0}(-\nu_{i_s})  \supsetneq 0
\end{align}
with parabolic weights $0 \leq -\nu_{i_1} < -\nu_{i_2} < \dots < -\nu_{i_s} < 1$.

\begin{prop}
Let $\VH$ be a VHS of weight $w$ on $\cu^\circ$ and $\hat{\VH}$ Deligne's canonical extension to $\cu$. 
Then $(\hat{E},\hat{\varphi})$ with the parabolic structure determined by monodromy at each $u\in D$ is a parabolic Higgs bundle. 

\end{prop}
\begin{proof}
Let $T$ be the local monodromy around a puncture $u\in D$ and $s_i$ be a local multi-valued frame of $\VH$.
In the local frame $\hat{s}_i$ of $\hat{\VH}$, we have $Res_u(\hat{\nabla})=-N$. 
Since $[T,N]=0$, it follows that $\hat{\varphi}(\hat{E}_{u}(-\nu_{i_j}))\subset \hat{E}_{u}(-\nu_{i_j})$. 
\end{proof}


\subsection{Parabolic degrees from Picard--Fuchs equations}

As before, let $\cu$ be a compact Riemann surface and $\cuo=C-D \subset C$ a Zariski-dense subset which carries a polarizable $\Z$-VHS $\VH$ of weight $w=r-1$. 
We assume that $\VH$ admits a generic cyclic vector $\omega\in H^0(C^\circ, F^w)$,
In this section we express the degree of the extended Hodge bundle $\hat{\Fc}^{w}\subset \hat{\VH}$ of Deligne's canonical extension in terms of the previously defined exponents $\mu_1^u \leq  \dots \leq \mu_r^u$ for $u\in \cu$.

To state the results, it is convenient introduce the divisor $D_\omega\subset C^\circ$ defined by $\omega$. 
By Remark \ref{rem:exponents} it is given by 
\begin{equation}
D_\omega=\{ u\in C^\circ ~|~\mu_1^u\in \Z-\{ 0\} \}.
\end{equation}
Note that $C'=C^\circ-D_\omega$ is the largest open subset of $C^\circ$ on which $\omega$ is a cyclic vector. 
We further define $\hat{D}:=D_\omega+D\subset D$ the divisor in $C$ on which the extension $\hat{\omega}$ of $\omega$ to $C$, cf. the next proof, possibly vanishes or has poles. 

\begin{prop}\label{prop:deg}
Let $\VH$ be a polarizable $\Z$-VHS of weight $w=r-1$ on $\cuo$ and $\omega\in H^0(\cuo, F^{w})$ be a generic cyclic section.
Further let $\mu_1^u \leq \dots \leq   \mu_{r}^u$, $u\in \cu$, be the exponents of $\omega$. 
Then the degree of the extended Hodge line bundle $\hat{F}^{w}$ of Deligne's canonical extension $\hat{\VH}$ to $\cu$ is given by 
\begin{equation}\label{eq:deg}
\deg (\hat{F}^{w})=\sum_{u\in \cu} \lfloor \mu_1^u \rfloor=\sum_{u\in \hat{D}} \lfloor \mu_1^u \rfloor.
\end{equation}
\end{prop}
Even though the exponents $\mu_j^u$ clearly depend on $\omega$, the sum on the right-hand side of (\ref{eq:deg}) is independent of $\omega$. 
In particular, rescaling by a non-zero meromorphic function does not change the result. 
\begin{proof}
As before we order the exponents $\mu_1^u\leq \dots \leq  \mu_r^u$ for each $u\in C$. 
Let $\cu'\subset \cuo$ be the dense subset such that $\omega$ is cyclic. 
In particular, $\omega(u)\neq 0$ and $\mu_1^u= 0$ for all $u\in \cu'$, cf. Remark \ref{rem:exponents}. 
If $u\in \cu-\cu'$,  we choose $\nu^u_j\in (-1,0]$, as in the construction of $\hat{\VH}$,  such that\footnote{Note that \eqref{eq:nuj} does not imply $\nu_i^u\leq \nu_j^u$ for $i\leq j$ in general. However, this is not important at the moment because we do not consider the parabolic structures at $u$ here.}
\begin{equation}\label{eq:nuj}
\exp(2\pi i \nu_j^u)=\lambda_j^u= \exp(-2\pi i \mu_j^u).
\end{equation}
Here $\lambda_j^u$ are the eigenvalues of the monodromy $T(u)$ of $\VH_{\Z}$ around $u$. 
Moreover, we denote by $N(u)$ the logarithm of $T(u)$ determined by $\nu_j^u$. 

From now on, we drop $u$ from the notation. 
Let $(s_1,\dots, s_r)$ be a multi-valued frame of $\VH_{\Z}$ 
Then we express
\begin{equation*}
\omega= (s_1, \dots, s_r) \begin{pmatrix}
f_1 \\ 
\vdots \\ 
f_r
\end{pmatrix}= 
(\hat{s}_1, \dots, \hat{s}_r) \exp(2\pi i N \tau) \begin{pmatrix}
f_1 \\ 
\vdots \\
f_r
\end{pmatrix}
=: 
(\hat{s}_1, \dots, \hat{s}_r)
\begin{pmatrix}
g_1 \\
\vdots 
\\ 
g_r
\end{pmatrix}. 
\end{equation*}
Note that $f_1,\dots, f_r$ form a basis of solutions to (\ref{eq:PF}) by Lemma \ref{lem:LS} and that $g_1, \dots, g_r$ are single-valued holomorphic functions. 
It follow from (\ref{eq:basismultivalued}) that they are of the form
\begin{equation*}
g_j(z)=z^{\mu_j+\nu_j} g_j'(z)
\end{equation*}
with holomorphic $g_j'$ such that $g_j'(0)\neq 0$. 
In order for $g_j$ to be holomorphic, we must have 
\begin{equation}\label{eq:kj}
k_j:=\mu_j+\nu_j\in \Z.
\end{equation}
By the choice of $\nu_j$, this implies $k_j=\lfloor \mu_j \rfloor$. 
Thus we can write 
\begin{equation*}
\begin{pmatrix}
g_1 \\ 
\vdots \\
g_r
\end{pmatrix}= 
z^{\lfloor \mu_1 \rfloor} \begin{pmatrix}
g_1'' \\
\vdots \\
g_r''
\end{pmatrix}
\end{equation*}
with $g_1'(0) \neq 0$ so that $\omega= z^{\lfloor \mu_1 \rfloor} \omega'$ for a nowhere vanishing section $\omega'\in H^0(\Delta, \hat{F}^{w})$. 
Therefore we conclude the claim: 
\begin{equation*}
\deg(\hat{F}^{w})=\sum_{u \in C} \lfloor \mu_1^u \rfloor= \sum_{u\in \hat{D}} \lfloor \mu_1^u \rfloor.
\end{equation*}
\end{proof}

Proposition \ref{prop:deg} enables us to compute the degrees of the bundles $\hat{E}^{p,q}$ constructed from Deligne's canonical extension, cf. (\ref{eq:Ehatpq}).
To give the formula, we decompse $\hat{D}=\hat{D}_a+\hat{D}_s$.
Here the divisor
\begin{equation*}
\hat{D}_a=\{u\in \hat{D}~|~\mu_j^u\in \Z \mbox{ for all }j=1,\dots,r  \}=D_a+D_\omega
\end{equation*}
is the union of the divisor of apparent singularities $D_a\subset D$ of the corresponding Picard--Fuchs equations and the divisor $D_\omega\subset C^\circ$ defined by $\omega$. 
The divisor $\hat{D}_s\subset D$ is the divisor of regular singularities of $\hat{\nabla}$. 

\begin{thm}\label{thm:degEpq}
Consider the line bundle $\hat{E}^k:=\hat{E}^{w-k,k}$, $l=0,\dots, w$, where $w=r-1$ is the weight of the $\Z$-VHS $(\VH_{\Z},F^\bullet)$ with generic cyclic vector $\omega$. 
Then its degree is given by\footnote{We denote by $|D|$ the number of irreducible components of a divisor $D\subset C$.} 
\begin{equation}\label{eq:deghatEpq}
\deg(\hat{E}^{k})=\sum_{u\in \hat{D}} \lfloor \mu_{k+1}^u \rfloor - k \left(|\hat{D}_a|+|D|+ (2g_C-2) \right)
\end{equation}
for the genus $g_C$ of $C$.
\end{thm}
\begin{rem}
Clearly, the formula \eqref{eq:deghatEpq} coincides with \eqref{eq:deg} for $k=0$. 
Moreover, observe that $\deg(\hat{E}^k)$ for $\hat{E}^k=\hat{E}^{w-k,k}$ is completely determined by $\omega$ and the pair $(C,D)$. 
\end{rem}

\begin{proof}
The components 
\[
\hat{\varphi}^k:\hat{E}^{k}\to \hat{E}^{k+1}\otimes \Omega^1_C(\log D),
\]
$k=0, \dots, w-1$, of the Higgs field $\hat{\varphi}$ give the formula
\begin{equation}\label{eq:hatEnk}
\begin{aligned}
\deg(\hat{E}^{k+1})&=\deg\left(\mathrm{Div}(\hat{\varphi}^{k})\right)+\deg(\hat{E}^{k})-\deg(\Omega^1_C(\log D))
\\
&= \deg\left(\mathrm{Div}(\hat{\varphi}^{k})\right)+\deg(\hat{E}^{k})- \left(|D|+(2g_C-2)\right)
\end{aligned}
\end{equation}
for the divisor $\mathrm{Div}(\hat{\varphi}^{k})$ defined by $\hat{\varphi}^{k}$.
The order $v_u(\hat{\varphi}^k)$ of $\hat{\varphi}^k$ at $u$ is given by 
\begin{equation}\label{eq:vhatphi}
v_u(\hat{\varphi}^k)=\begin{cases}
\lfloor \mu_{k+2}^u \rfloor - \lfloor \mu_{k+1}^u \rfloor, & u\in D_s, \\
\lfloor \mu_{k+2}^u \rfloor - \lfloor \mu_{k+1}^u \rfloor - 1, & u\in \hat{D}_a, 
\end{cases}
\end{equation}
see \cite[Lemma 6.3]{EskinEtAl} and \cite[Theorem 2.7]{DoranEtAlVHS}. 
Note that $\hat{\varphi}^k$ is an isomorphism for the regular points $u\in D_\omega\subset \hat{D}_a$ because $\mu_{k+1}^u=\mu_{k}^u+1$ in these cases, see Remark \ref{rem:exponents}.
Summing over $u\in C$ gives
\begin{equation}\label{eq:divphi}
\deg(\mathrm{Div}(\hat{\varphi}^{k}))=\left(\sum_{u\in \hat{D}}\lfloor \mu_{k+2}^u  \rfloor -\lfloor \mu_{k+1}^u \rfloor \right) - |\hat{D}_a|. 
\end{equation}

Now we prove (\ref{eq:deghatEpq}) inductively: 
For $k=0$, combining (\ref{eq:divphi}) with Proposition \ref{prop:deg} and (\ref{eq:hatEnk}) yields 
\begin{equation*}
\deg(\hat{E}^{1})= \sum_{u\in D} \lfloor \mu_2^u \rfloor - \left(|\hat{D}_a | + |D|+(2g_C-2)\right). 
\end{equation*}

The induction step $k-1 \to k$ gives:
\begin{align*}
\deg(\hat{E}^{k})& = \deg(\mathrm{Div}(\hat{\varphi}^{k})+\deg(\hat{E}^{k-1})-\deg(\Omega^1_C(\log D)) \\
& =
\sum_{u\in \hat{D}} \lfloor \mu_{k+1}^u \rfloor  - k\left(|\hat{D}_a| + |D|+(2g_C-2)\right).
\end{align*}
\end{proof}
\begin{cor}\label{cor:conjpdeg}
Assume $\mu_r^u-\mu_1^u\in [0,1)$ for all $u\in D_s$. 
Then the parabolic degrees of $\mathrm{pdeg}(\hat{E}^{k})$, $\hat{E}^k=\hat{E}^{w-k,k}$, with respect to the induced parabolic structure, are
\begin{equation*}
\mathrm{pdeg}(\hat{E}^{k})= \deg(\hat{E}^{k})-\sum_{u\in D_s}
\nu_{k+1}^u.
\end{equation*}
for $k=0,\dots,w$. 
\end{cor}
Since $\mu_{k+1}^u+\nu_{k+1}^u=\lfloor \mu_{k+1}^u\rfloor$ for $u\in D_s$ the parabolic degrees are given by (\ref{eq:deghatEpq}) with $\lfloor \mu_{k+1}^u \rfloor$ replaced by $\mu_{k+1}^u$. 
Note that $\lfloor \mu_{k+1}^u \rfloor=\mu_{k+1}^u$ for $u\in \hat{D}_a=\hat{D}-D_s$ anyway. 
\begin{rem}\label{rem:comp}
A formula for $\mathrm{pdeg}(\hat{E}^{w-k,k})$ first appeared in Theorem 6.1 of \cite{EskinEtAl} in the special case of 14 VHS of Calabi-Yau type over $C=\mathbb{CP}^1$, see Example \ref{ex:14VHS} for more details. 
Moreover, \cite{DoranEtAlVHS} give degree formulas in the example of families of elliptic curves ($n=2$) and K3 surfaces ($n=3$) over $C=\mathbb{CP}^1$. 
\end{rem}
\begin{proof}
The statement is equivalent to the claim that the induced parabolic weights of $\hat{E}^k$ at $u$ are $-\nu_{k+1}^u$ for every $u\in D_s$. 
Elsewhere the parabolic structure of $\hat{E}$ and hence of $\hat{E}^k$ is trivial.

Fix $u\in D_s$.  
We denote by $i_j$, $j=1,\dots, s$, the indices such that 
$0\leq -\nu_{i_1} < \dots < -\nu_{i_s} < 1$ and each $\nu_{i_j}$ has multiplicity $m_j$.
Observe that this implies  
\begin{equation}\label{eq:dimEk}
\mathrm{rk}(\hat{E}(-\nu_{i_j}))=\sum_{l=1}^j m_l\quad \mbox{ and } \sum_{l=1}^s=r.
\end{equation}
We claim that 
\begin{equation}\label{eq:parabolicdegs}
\hat{E}^k\subset \hat{E}(-\nu_{i_j}),\quad \hat{E}^k\cap \hat{E}(-\nu_{i_{j+1}})=0
\end{equation}
if $\nu_{k+1}=\nu_{i_j}$.  
We first show that 
\begin{equation}\label{eq:claimEk}
\hat{E}^k\subset \hat{E}(-\nu_{i_1}), \quad \hat{E}^k\cap \hat{E}(-\nu_{i_2})=0
\end{equation}
for $k=0,\dots, m_1-1$. 
Assume the contrary, i.e. $\hat{E}^k\subset \hat{E}(-\nu_{i_j})$ for some $k< m_1-1$ and $j\geq 2$. 
By assumption on the exponents at $u$ and \eqref{eq:vhatphi}, $\hat{\varphi}^k_u$ maps $\hat{E}_u^k$ isomorphically to $\hat{E}^{k+1}_u$. 
Since $\hat{\varphi}$ preserves the parabolic filtration, $\hat{E}_u^{k+1}\subset \hat{E}(-\nu_{i_j})$. 
Inductively, we see that  $\hat{E}_u^l\subset \hat{E}(-\nu_{i_j})$ for all $l \geq k$. 
But then  
\begin{equation*}
\mathrm{rk}(\hat{E}(-\nu_{i_j}))> \sum_{l=1}^j m_j
\end{equation*}
contradicting \eqref{eq:dimEk} so that \eqref{eq:claimEk}) follows. 
Repeating this argument successively we arrive at \eqref{eq:parabolicdegs}.
Finally, \eqref{eq:parabolicdegs} implies that $\hat{E}^k$ has parabolic weight $-\nu_{i_j}=-\nu_{k+1}$ as claimed. 
\end{proof}

\section{Examples and oper gauge transformation}\label{examples}

We will discuss examples of Calabi--Yau manifolds over complex one-dimensional moduli spaces $\mathcal{B}$ which arise in mirror symmetry. 
These are obtained as mirror families of Calabi--Yau hypersurfaces $Y$ in toric varieties with $\dim_{\C} H^{1,1}(Y)=1$.
The data of the mirror families is given by dual polyhedra $\check{\Delta},\Delta$ using Batyrev's construction \cite{Batyrev}, equivalently using Hori--Vafa's construction \cite{Hori:2000kt}. 
We briefly review the necessary ingredients of these constructions and refer to Refs.~\cite{CoxKatz,Clader:2014kfa} and references therein for further background.

We will start with an analogous discussion to \S\ref{explicit} of the rank 2 Higgs bundles obtained from VHS of general elliptic curves and then specialize to the Legendre and cubic families. 
In these cases, Theorem \ref{thm:degEpq} together with the classification of line bundles on $\mathbb{CP}^1$ allow us to completely determine the parabolic Higgs bundles which we illustrate in concrete examples. 
Moreover, we will see in Example \ref{ex:14VHS} that our formula reproduces the parabolic degrees computed in \cite{EskinEtAl}. 

\subsection{Mirror construction and Picard--Fuchs equations}\label{ss:ms}
To describe the mirror pair of d-dimensional CY hypersurfaces we use Batyrev's construction \cite{Batyrev}. We consider $\check{\Delta}$ a reflexive polyhedron in $\mathbbm{R}^{d+2}$ defined as a convex hull of $d+3$ integral vertices $\nu_i\in \mathbbm{Z}^{d+2}\subset \mathbbm{R}^{d+2}\,, i=0,\dots,d+2$ lying in a hyperplane of distance one from the origin. $\check{W}=P_{\Sigma(\check{\Delta})}$ is the toric variety with fan $\Sigma(\check{\Delta})$ defined by the set of cones over the faces of $\check{\Delta}$. $\Delta$ is the dual polyhedron and $W$ is the toric variety obtained from $\Sigma(\Delta)$. The mirror pair of CY d-folds given as hypersurfaces in $(\check{W},W)$ is denoted by $(\check{\CY},\CY)$.

The hypersurface $\CY$ is determined as the vanishing locus of the equation:
\begin{equation}
P(\CY)=\sum_{i=0}^{d+2}a_i y_i =\sum_{\nu_i\in \Sigma} a_i X^{\nu_i}\,,
\end{equation}
where $a_i$ are complex parameters and $y_i$ certain homogeneous coordinates on $W$ \cite{Hori:2000kt}. $X_k\,,k=1,\dots,d+1$ are inhomogeneous coordinates on an open torus $(\mathbbm{C}^*)^{d+1} \subset W$ and $X^{\nu_i}:=\prod_k X_k^{(\nu_i)_k}$ \cite{Batyrev}, where $(\nu_k)$ denotes the $k-$th entry of $\nu_i$.

The integral points $\nu_i$ and the homogeneous coordinates $y_i$ fulfill one relation specified in terms of a $d+3$ dimensional vector $l$:
\begin{equation}
\sum_{i=0}^{d+2} l_i \nu_i =0 \,, \quad \prod_{i=0}^{d+2} y_i^{l_i}=1\,.
\end{equation}
The integral vector $l$ specifies the charge of the matter fields of the gauged linear sigma model associated with $\CY$ \cite{Witten:1993yc}. 

The period integrals of the holomorphic $d-$form on $\CY$ are given by:
\begin{equation}
\pi (a_i) =\frac{1}{(2\pi i)^{d+1}} \int_{|X_k|=1} \frac{1}{P(\CY)} \prod_{k=1}^{d+1} \frac{dX_k}{X_k}\,.
\end{equation}

These periods are annihilated by a system of differential equations of GKZ hypergeometric type\footnote{See \cite{GKZ} for background and definitions.}:
\begin{eqnarray}
L = \prod_{l_i>0} \left( \frac{\partial}{\partial a_i}\right)^{l_i} -\prod_{l_i < 0 } \left( \frac{\partial}{\partial a_i}\right)^{-l_i}\,, \\
\mathcal{Z}_k= \sum_{i=0}^{d+2} \nu_{i,k} \theta_i\,, k=1,\dots,d+1\,,\quad \mathcal{Z}_0=\sum_{i=0}^{d+1} \theta_i +1\,,
\end{eqnarray}
where $\theta_i:=a_i \frac{\partial}{\partial a_i}$. The differential equation $L \pi(a_i)=0$ is satisfied by definition. The equations $\mathcal{Z}_k \pi(a_i)=0$ express the invariance of the period integral under the torus action and imply that the period integrals depend only on special combinations of the parameters $a_i$. These are given by:
\begin{equation}
z:= (-1)^{l_0} \prod_i a_i^{l_i}\,,
\end{equation}
and define local coordinates on the moduli space of complex structures $\mathcal{B}$ of $\CY$.

\subsection{Rank two}\label{ss:rank2}
For the rank two case we consider families of elliptic curves
$$\pi: \mathcal{X}\rightarrow \mathcal{B}$$
with fibers $\pi^{-1}(u)=X_u\,,u\in \mathcal{B}$. Following the notation of \S\ref{explicit},  we have a line bundle $\mathcal{L}\rightarrow \mathcal{B}$ whose fibers are the cohomology groups $H^{1,0}(X_u)$. 
In particular, the associated VHS $\mathcal{H}$ of weight $1$ decomposes smoothly into $\mathcal{H}=\mathcal{L}\oplus \bar{\mathcal{L}}$.
We choose a local trivialization of $\mathcal{L}$ given by a choice of holomorphic one-form $\omega_0$. 
By working on a local coordinate chart of $\mathcal{B}$ with coordinate $u$, we arrive at the second order equation
\begin{equation}\label{ranktwoPF}
    \hbar^2 \nabla_u^{2} \omega_0 = -b_1 \hbar\nabla_u \,\omega_0 -b_0\, \omega_0\,, \quad \nabla_u:=\nabla_{\partial/\partial u}
\end{equation}
where we have considered the one--parameter deformation by $\hbar$ as in the introductory example. 
In particular, it yields the Picard--Fuchs equation for the two periods of $\omega_0$.
The Griffiths--Yukawa coupling $c\in \Gamma(\mathcal{L}^{-2}\otimes T^*\mathcal{B})$ is defined by
\begin{equation}
    c(\omega,\omega'):=\eta(\omega, \nabla \omega'), \quad \omega,\omega'\in \mathcal{L}, 
\end{equation}
cf. Definition \ref{def:gy}.
Its coordinate expression with respect to $u$ and $\omega_0$ is 
$$ 
c_u := \eta(\omega_0,\nabla_u \omega_0)= i \int_{X_u} \omega_0 \wedge \nabla_u \omega_0 \,.$$
From \eqref{ranktwoPF} we derive the following equation satisfied by $c_u$:\footnote{See also \cite{Alim:2014vea}.}
\begin{equation}
 \partial_u c_u = -\frac{b_1}{\hbar} \,c_u\,.
\end{equation}
We proceed to construct a basis for the fibers of the Hodge bundle that respects the Hodge decomposition.  We construct $\omega_1 \in H^{0,1}(X_u) \otimes T^* \mathcal{B}$ as: 
\begin{equation}
    \omega_1 = (\nabla_u -D_u) \omega_0\, \de u \,, 
\end{equation}
where $D_u$ denotes the Chern connection. 
Moreover, we define
\begin{equation}
    e^{-K}:=h(\omega_0,\omega_0)=h_{0\bar{0}}=i\int_{X_u} \omega_0\wedge\, \overline{\omega_0}\,.
\end{equation}
Then the Hodge metric $h=h_{a\bar{b}} \de t^a \de \bar{t}^{\bar{b}}\,, a,b,=0,1$ in the local holomorphic frame $\omega_{\mathrm{NAH}}=(\omega_0, \omega_1)$ is given by 
\begin{equation}
h_{a\bar{b}}= \left(\begin{array}{cc} 
e^{-K} &0  \\
0& e^{-K}G_{u\overline{u}}  \de u \de \ubar  
\end{array} \right)
\end{equation}
with $G_{u\bar{u}}:= \partial_u \partial_{\bar{u}} K$ is a K\"ahler metric on $\mathbb{P}^1-\{0, 1, \infty\}$, we have moreover the relation
$$ e^{-K} \, G_{u\bar{u}} = |c_{u}|^2\, e^{K} \,.$$

As in Proposition \ref{diffring}, we next derive the following differential ring relations:
let $\Gamma_{uu}^u:=G^{u\bar{u}} \partial_u G_{u\bar{u}}$ be the Levi--Civita connection of the previous K\"{a}hler metric and let $K_u := \partial_u K$. Then we have 
\begin{gather}\label{ranktwodiffring}
\Gamma_{uu}^u - 2 K_u = -\frac{b_1}{\hbar}\, , \\
\hbar^2 \partial_u K_u =\hbar^2 K_u^2 -\hbar b_1 K_u + b_0\,, \label{ranktwodiffring2}
\end{gather}
Here the $b_j$ are the coefficient functions of the Picard--Fuchs equation \eqref{ranktwoPF}.
The coordinate expressions of the Higgs field $\varphi$ and its adjoint $\varphi^{\dagger}$, as well as the flat non-abelian Hodge/$tt^*$ connections $\nabla_{\mathrm{NAH}}^\hbar = \lim_{R \to 0} \nabla_{\mathrm{NAH}}^{R \hbar, R}$ are obtained in complete analogy to Proposition \ref{NAHconnection}.
The limit $\nabla_{\mathrm{NAH}}^{\hbar} = \lim_{R \to 0} \nabla_{\mathrm{NAH}}^{\hbar R, R}$ in the basis $\omega_{\mathrm{NAH}}^{\hbar}=(\omega_0, \hbar \omega_1)$ is given by

\begin{equation}
    \nabla_{\mathrm{NAH}}^{\hbar} = \de+ \frac{1}{\hbar} \begin{pmatrix} 0 & 0 \\ 1 & 0 \end{pmatrix} +  \begin{pmatrix} - K_u & 0 \\ 0 & \Gamma_{uu}^u - K_u  \end{pmatrix} \de u + \hbar \begin{pmatrix} 0  & G_{u \ubar} \de u \de \ubar \\ 0 & 0 \end{pmatrix}.
\end{equation}

The gauge transformation $\mathcal{A}^{\hbar}$ between the frame $\omega^{\hbar}_{\mathrm{NAH}}=(\omega_0,\hbar\, \omega_1)$ and $\omega^{\hbar}_{\mathrm{GM}}=(\omega_0,\hbar\, \nabla_u \omega_0 \,\de u)$
$$ \omega^{\hbar}_{\mathrm{GM}}=  \omega_{\mathrm{NAH}}^{\hbar}\, \mathcal{A}^{\hbar},$$

is completely determined by $K_u$: 
\begin{equation}
(\mathcal{A}^{\hbar})^{-1}= \left( \begin{array}{cc}  1& \hbar \, K_u \, \de u \\
0&1
 \end{array}\right)\,,
\end{equation}
as can be easily seen by writing out:
$$\hbar \omega_1 = \hbar\, \nabla_u \omega_0\, \de u + \hbar\, K_u \,\omega_0\, \de u$$
In particular, we arrive again at 
\begin{eqnarray*}
((\mathcal{A}^{\hbar})^{-1} \circ \nabla^{\hbar}_{\mathrm{NAH}} \circ \mathcal{A}^{\hbar})^{0,1} &=&  (\mathcal{A}^{\hbar})^{-1}\, \partial_{\bar{u}} \mathcal{A}^{\hbar} \de \bar{u} + (\mathcal{A}^{h})^{-1} \varphi^{\dagger} \mathcal{A}^{\hbar} 
\\ &=&\left( \begin{array}{cc}  0& \hbar (-G_{u\bar{u}}+G_{u\bar{u}})\de u  \\
0&0
 \end{array}\right) \de \bar{u}\,,
\end{eqnarray*}
and
\begin{align*}
&((\mathcal{A}^{\hbar})^{-1} \circ \nabla^{\hbar}_{\mathrm{NAH}} \circ \mathcal{A}^{\hbar}))^{-1})^{1,0} \\&=  (\mathcal{A}^{\hbar})^{-1}\, \partial_{u} \mathcal{A}^{\hbar} \de u + (\mathcal{A}^{h})^{-1} \varphi \mathcal{A}^{\hbar}  +  (\mathcal{A}^{\hbar})^{-1} \, \begin{pmatrix} - K_u & 0 \\ 0 & \Gamma_{uu}^u - K_u  \end{pmatrix} \mathcal{A}^{\hbar} \de u 
\\ &=\left(
\begin{array}{cc}
 0 & \hbar \left(\Gamma_{uu}^u\, K_u
 -\partial_u K_u-K_u^2\right)
   \\
 \frac{1}{\hbar} & \Gamma_{uu}^u-2 K_u \\
\end{array}
\right)\de u \overset{\eqref{ranktwodiffring},\eqref{ranktwodiffring2}}{=}  \frac{1}{\hbar} \left( \begin{array}{cc} 0&-b_0 \\
1&-b_1
\end{array}\right) \de u\,. 
\end{align*}

And hence
\begin{equation}
\nabla^{\hbar}_{\mathrm{GM}} = (\mathcal{A}^{\hbar})^{-1} \circ \nabla^{\hbar}_{\mathrm{NAH}} \circ \mathcal{A}^{\hbar}\,,
\end{equation}
for the family of opers
\begin{equation}
\nabla^{\hbar}_{\mathrm{GM}} = \de +\frac{1}{\hbar} \left( \begin{array}{cc} 0&-b_0 \\
1&-b_1
\end{array}\right) \de u \,.
\end{equation}

\begin{ex}[Legendre family]
We continue with the case of families of elliptic curves in Example \ref{ex:ell1}, but now in the parabolic case.
Concretely, let $\pi:X\to \cu=\mathbb{P}^1$ be the elliptic surface determined by the Legendre family 
\begin{equation*}
y^2=x(x-1)(x-u),\quad u\in \C.
\end{equation*}
The fibers are non-singular over $\cuo=\mathbb{P}^1-\{0,1,\infty\}$ and $\omega_0=dx/y$ is a cyclic vector on all of $\cuo$ so that $|\cu_{a,pd}|=0$. We denote the periods of $\omega_0$ over a basis of integral cycles $A,B\in H_1(X_u,\mathbb{Z})$ by $\pi^0,\pi^1$. These are annihilated by the following Picard--Fuchs operator:\footnote{See for example \cite{CarlsonEtAl2} for the derivation of this operator and for the discussion of the modular properties.}
\begin{equation}\label{LegendrePF}
L= \theta_u^2- u \left( \theta_u +\frac{1}{2}\right)^2 \,,
\end{equation}
whose exponents are 
\begin{center}
\begin{tabular}{l | l  l  l}
d & 0 & 1 & $\infty$ \\ \hline
$\mu_1^d$ & 0 & 0 & $\frac{1}{2}$ \\
$\mu_2^d$ & 0 & 0 & $\frac{1}{2}$
\end{tabular}
\end{center}
Hence the formula for the degrees of $\widehat{E}^{\bullet, \bullet}$ in (\ref{eq:deghatEpq}) implies
\begin{equation*}
\hat{E}=\hat{E}^{1,0}\oplus \hat{E}^{0,1}\cong \Oo_{\mathbb{P}^1}\oplus \Oo_{\mathbb{P}^1}(-1). 
\end{equation*}
The parabolic structure is trivial at $d=0,1$ whereas at $d=\infty$ it is given by 
\begin{center}
\begin{tabular}{c c c c c c c}
$ \hat{E}_\infty$ & $\supsetneq $ & $\hat{E}_\infty(\tfrac{1}{2})$ & $=$ & $\hat{E}_\infty(\tfrac{1}{2})$ & $\supsetneq$ & $0$ 
\end{tabular}
\end{center}
The former two cases follow because the local monodromy around $d\in \{0,1\}$ is maximally unipotent. 
For $d=\infty$, the local monodromy has the eigenvalue $i$ with multiplicity $2$. 
In particular, we directly compute  
\begin{gather*}
\mathrm{pdeg}(\hat{E})= -1+ 2\cdot \frac{1}{2}=0, \\
\mathrm{pdeg}(\hat{E}^{1,0})=\frac{1}{2}=-\mathrm{pdeg}(\hat{E}^{0,1}). 
\end{gather*}
\end{ex}

The solutions of $L \,\pi^i=0,, i=0,1$ are given by:
\begin{equation}
\pi^0= {}_2F_1(1/2/1/2,1,u)\,,\quad \pi^1=\frac{i}{2}{}_2F_1(1/2/1/2,1,1-u)\,,
\end{equation}
where we chose the normalization such that the modular coordinate is:
$$
\tau =\frac{\pi^1}{\pi^0}\,.$$
Their monodromy group is $\Gamma_0(4)$.
Generators of the corresponding quasi-modular forms (see Appendix \ref{modularappendix}) are given by
\begin{eqnarray}
A&=& \theta_3^2(2\tau)\,, \\
B&=& \theta_4^2(2\tau)\,,\\
E&=& \frac{1}{3} \left(4 E_2(4\tau) + E_2(\tau)\right) -\frac{2}{3} E_2(2\tau)\,.
\end{eqnarray}
These satisfy the following differential ring relations:
\begin{align}\label{legendrediffring}
\partial_\tau A &=\frac{1}{4}A (E + A^2-2B^2)\,,\\ \nonumber
\partial_\tau B&= \frac{1}{4} B(E-A^2)\,,\\ \nonumber
\partial_\tau E&=\frac{1}{4}(E^2-A^4)\,.
\end{align}

We obtain expressions for the geometric data in terms of quasi-modular forms:
\begin{equation}\label{legendremodular}
    \pi^0(\tau)= A(\tau)\,, \quad
\mathrm{and}\quad e^{-K}= 2 |A(\tau)|^2 \Im \tau\,, \quad z(\tau)= 1- \frac{B(\tau)^2}{A(\tau)^2}\,.
\end{equation}
From the Picard--Fuchs operator we read off:
\begin{equation} \label{legendredata}
    b_1(u)= \frac{1-2u}{u(1-u)}\,, \quad b_0(u)=- \frac{1}{4 u(1-u)}\,.
\end{equation} 

In this case, the coordinate expression for the Griffiths--Yukawa coupling is given by
\begin{equation}
c_u = \frac{\kappa}{u (1-u)}\,,
\end{equation}
where $\kappa$ is an integration constant which we will set to 1.

One can now easily verify with the Picard--Fuchs coefficients $b_0, b_1$ in \eqref{legendredata} as well as the modular expressions in \eqref{legendremodular}  the differential ring relations \eqref{ranktwodiffring}. 
These translate into the differential ring relations for the quasi--modular forms for $\Gamma_0(4)$ \eqref{legendrediffring}.

Our analysis shows that the data of the oper corresponds to the data of the first order formulation of the Picard--Fuchs equation describing the VHS. The latter does however not correspond to the Higgs field of the non-abelian Hodge flat connection as was assumed in \cite[Table I]{Dumitrescu2014a}. Rather, the oper data is a combination of the Higgs field and holomorphic remnants of a gauge transformation.

\begin{ex}[Cubic curve] 
We consider the cubic curve $\check{\CY}$ given by a section of the anti-canonical bundle over the projective plane $\mathbbm{P}^2$ and its mirror $\CY$. The vertices of $\Delta$ are given by:
\begin{equation}
\nu_0=(1,0,0)\,, \quad \nu_1=(1,1,0)\,, \quad \nu_2=(1,0,1)\,, \quad \nu_3= (1,-1,-1)\, ,
\end{equation}
satisfying the relation $\sum_{i=0}^{3} l_i \nu_i=0$ where:
\begin{equation}
l=\begin{array}{c|ccc}
(-3&1&1&1)
\end{array}\,.
\end{equation}
$\CY$ is defined as a suitable compactification of
\begin{equation}
\{(X_1,X_2)\in (\C^\times)^2~|~ P(\CY)= a_0 + a_1 X_1 + a_2 X_2 +\frac{a_3}{X_1 \, X_2} =0 \}\,.
\end{equation}
We define a local coordinate $z=-\frac{27 a_1 a_2 a_3}{a_0^3}$ on the moduli space  $\mathcal{B}$ of $\CY$ and obtain the Picard--Fuchs operator:
\begin{equation}\label{PFcurve}
L= \theta^2-  z (\theta +1/3) ( \theta+2/3)\,,\quad \theta=z\frac{d}{dz}\,,
\end{equation}
with discriminant $\Delta=1-z$. We determine the exponents to be:
\begin{center}
\begin{tabular}{l | l  l  l}
d & 0 & 1 & $\infty$ \\ \hline
$\mu_1^d$ & 0 & 0 & $\frac{1}{3}$ \\
$\mu_2^d$ & 0 & 0 & $\frac{2}{3}$
\end{tabular}
\end{center}
Again Theorem \ref{thm:degEpq} implies 
\begin{equation*}
\hat{E}\cong \Oo_{\mathbb{P}^1}\oplus \Oo_{\mathbb{P}^1}(-1). 
\end{equation*}
But the parabolic structure differs from the one in Example \ref{ex:ell1}.
It is again trivial at $d=0,1$ but at $d=\infty$ it is given by 
\begin{center}
\begin{tabular}{c c c c c c l}
$\hat{E}_\infty$ & $=$ & $\hat{E}_\infty(\tfrac{1}{3})$ & $\supsetneq$ & $\hat{E}_\infty(\tfrac{2}{3})$ & $\supsetneq$ & $0$.
\end{tabular}
\end{center}
It follows that  
$
\mathrm{pdeg}(\hat{E}^{1,0})=\frac{1}{3}=-\mathrm{pdeg}(\hat{E}^{0,1}).
$

The solutions of $L \,\pi^i=0$ for $i=0,1$ are given by
\begin{equation}
\pi^0= {}_2F_1(1/3/2/3,1,z)\,,\quad \pi^1=\frac{i}{2\pi\sqrt{3}} {}_2F_1(1/3/2/3,1,1-z)\,,
\end{equation}
with monodromy group $\Gamma_0(3)$.
We chose the normalization such that the modular coordinate is:
\begin{equation}
\tau =\frac{\pi^1}{\pi^0}=\frac{1}{2\pi i}\log \left(\frac{z}{27}\right)+\frac{5 z}{9}+\frac{37 z^2}{162}+\dots
\end{equation}
As generators of the quasi-modular forms of $\Gamma_0(3)$ we choose
\begin{eqnarray}
E&=& \frac{1}{4} (E_2(\tau)+ 3 E_2(3\tau))\,,\\
A&=& \frac{(27 \eta(3\tau)^{12} +\eta(\tau)^{12})^{1/3}}{\eta(\tau) \eta(3\tau)}\,, \\
B&=& \frac{\eta(\tau)^3}{\eta(3\tau)}\,,\\
C&=& \frac{3(\eta(\tau))^3}{\eta(\tau)}\,,
\end{eqnarray}
which satisfy the algebraic relation:
\begin{equation}
A^3=B^3+C^3\,,
\end{equation}
as well as the differential ring relations:
\begin{align}\label{cubicdiffring}
\partial_\tau A &=\frac{1}{6} (E\, A + A^3-2B^3)\,,\\ \nonumber
\partial_\tau B&= \frac{1}{6} B(E-A^2)\,,\\ \nonumber
\partial_\tau E&=\frac{1}{6}(E^2-A^4)\,.
\end{align}

As before, the quasi-modular forms completely determine the periods, K\"ahler metric and mirror map
\begin{equation}\label{cubicmodular}
    \pi^0(\tau)= A(\tau)\,, \quad
\mathrm{and}\quad e^{-K}= 2 |A(\tau)|^2 \Im \tau\,, \quad z(\tau)= 1- \frac{B(\tau)^3}{A(\tau)^3}\,.
\end{equation}
From the Picard--Fuchs operator we read off:
\begin{equation} \label{cubicdata}
    b_1(z)= \frac{1-2z}{z(1-z)}\,, \quad b_0(z)=- \frac{2}{9 z(1-z)}\,.
\end{equation} 

The coordinate expression Griffiths--Yukawa coupling in this case is given by
\begin{equation}
c_z = \frac{\kappa}{z (1-z)}\,,
\end{equation}
where $\kappa$ is an integration constant which we will set to 1. 

One can now again  verify with the Picard--Fuchs data \eqref{cubicdata} as well as the modular expressions \eqref{cubicmodular} the differential ring relations \eqref{ranktwodiffring}, which in turn translate into the differential ring relations for the quasi--modular forms for $\Gamma_0(3)$ \eqref{cubicdiffring}.
\end{ex}

\subsection{Higher rank}

For higher rank, we consider (complete) families of Calabi--Yau $d$-folds ($d\geq 2$)
$$\pi: \mathcal{X}\rightarrow \mathcal{B}$$
with fibers $\pi^{-1}(u)=X_u\,,u\in \mathcal{B}$ and again $\dim_{\C} \mathcal{B}=1$. 
In \S\ref{explicit} we have explicitly discussed the rank 3 case attached to the VHS $\mathcal{H}$ of mirror lattice polarized K3 surfaces. We will therefore only shortly provide the definition of the mirror without repeating the analysis of \S\ref{explicit}.

\begin{ex}[The mirror quartic]\label{ex:mirrorquartic}
We consider the $K3$ surface $\check{\CY}$ given by a quartic hypersurface in $\mathbbm{P}^3$, and its mirror $\CY$. 
The vertices of $\Delta$ are given by:
\begin{eqnarray*}
&&\nu_0=(1,0,0,0)\,, \quad \nu_1=(1,1,0,0)\,, \quad \nu_2=(1,0,1,0)\,, \\
&&\nu_3= (1,0,0,1)\, , \quad \nu_4=(1,-1,-1,-1)
\end{eqnarray*}
satisfying the relation $\sum_{i=0}^{4} l_i \nu_i=0$ where:
\begin{equation}
l=\begin{array}{c|cccc}
(-4&1&1&1&1)
\end{array}\,.
\end{equation}
$\CY$ is defined by:
\begin{equation}
\CY= \{ P(\CY)= a_0 + a_1 X_1 + a_2 X_2 + a_3 X_3+ a_4 (X_1 X_2 X_3)^{-1} =0 \subset (\mathbbm{C}^*)^3 \}\,.
\end{equation}
We define a local coordinate $z=\frac{4^4\,a_1 a_2 a_3 a_4}{a_0^4}$ on the moduli space $\mathcal{M}$ of the mirror quartic $\CY$ and obtain the Picard-Fuchs operator:
\begin{equation}
L_{PF}= \theta^3-  z \prod_{i=1}^3\,(\theta +i/4) \,,\quad \theta=z\frac{d}{dz}\,.
\end{equation}

\end{ex}

In the following we explain how how the gauge transformation, and hence differential ring relations, carry over to higher rank. 

Let $\mathcal{H}$ be the VHS of weight $d$ attached to $\pi$ with the line subbundle $\mathcal{L}:=F^d\mathcal{H}$.
We choose a local frame $\omega_0$ of $\mathcal{L}$, which we may assume to be a cyclic vector, given by a fiberwise choice of holomorphic $d$-form $\omega_0$, we moreover assume that it depends further on $\hbar\in \C^{\times}$. 
By working on a local coordinate chart of $\mathcal{B}$ with coordinate $u$, we arrive at the $(d+1)^{th}$ order equation
\begin{equation}\label{rankdPF}
    \hbar^{d+1}\nabla_u^{d+1} \,\omega_0 =-\sum_{i=0}^d b_i\, \hbar^{i} \nabla^i_u \,\omega_0 \,, \quad \nabla_u:=\nabla_{\partial/\partial u}\,.
\end{equation}
It yields the Picard--Fuchs equation for the periods of $\omega_0$ (around singular points).
The Griffiths--Yukawa coupling $c\in \Gamma(\mathcal{L}^{-2}\otimes \mathrm{Sym}^d (T^*\mathcal{B}))$ is defined by
\begin{equation}
    c(\omega):=\eta(\omega, \nabla^d \omega), \quad \omega\in \mathcal{L}, 
\end{equation}
cf. Definition \ref{def:gy}.
Its coordinate expression with respect to $u$ and $\omega_0$ is 
$$ 
c_{u\dots u} := \eta(\omega_0,\nabla_u^d \omega_0).$$
The frame
\begin{equation*}
\omega^{\hbar}_{\mathrm{NAH}}=(\omega_0,\hbar\omega_1,\dots, \hbar^d \omega_d)
\end{equation*}
of the smooth bundle $\mathcal{H}$ is constructed as in \S\ref{explicit} or \S\ref{ss:rank2}.
In particular, $\omega_i$ is a local frame of $\mathcal{H}^{d-i,i}$.
Since $\omega_0$ is a cyclic vector,  
$$\omega^{\hbar}_{\mathrm{GM}} = \left( \omega_0\,,  \hbar \nabla_u \omega_0\, \de u \,, \dots  \hbar^{d}\nabla^{d}_u \omega_0 \, \de u^d \right)\, $$
is a frame as well. 
The gauge transformation $\mathcal{A}^{\hbar}$, which transforms $\omega^{\hbar}_{\mathrm{GM}}$ into $\omega^{\hbar}_{\mathrm{NAH}}$,
$$ \omega_{\mathrm{GM}}^{\hbar}= \omega_{\mathrm{NAH}} \, \mathcal{A}^{\hbar}\,,$$ can be explicitly determined by writing out $\hbar^k\,(\nabla_u-D_u)^k \omega_0\,, k=1,\dots,d$ and expressing the result in terms of $\omega^{\hbar}_{\mathrm{GM}}$.
As before it satisfies 
\begin{equation}
    \nabla_{\mathrm{GM}}^\hbar=(\mathcal{A}^{\hbar})^{-1}\circ \nabla_{\mathrm{NAH}}^\hbar \circ \mathcal{A}^{\hbar}.
\end{equation}
Here $\nabla_{\mathrm{NAH}}^\hbar$ is the family of flat $tt^*$-connections and
\begin{equation}
\nabla^{\hbar}_{\mathrm{GM}} = \de +\frac{1}{\hbar} \begin{pmatrix} 0 & 0 & 0 & -b_0 \\
1 & 0 & \ddots & -b_1 \\
\vdots & \ddots & \ddots & \vdots \\
0 & \cdots & 1 & -b_{d} \end{pmatrix} \, du
\end{equation}
of the family opers associated to the Gau\ss--Manin connection as before. 
\\

We next determine the parabolic Higgs bundles induced by families of Calabi--Yau threefolds.  
\begin{ex}[Mirror quintic]
We consider the quintic threefold $\check{\CY}$ given by a quintic hypersurface in $\mathbbm{P}^4$, and its mirror $\CY$.
This can be described by the toric charge vector:
\begin{equation}
l=\begin{array}{c|ccccc}
(-5&1&1&1&1&1).
\end{array}
\end{equation}
We define a local coordinate $z=-3125\frac{a_1\,a_2\,a_3\,a_4\,a_5}{a_0^5}$ on the moduli space $\mathcal{B}$ of the mirror $\CY$ and obtain the Picard--Fuchs operator:
\begin{equation}\label{PFquintic}
L=\theta^4- z \prod_{i=1}^4 (\theta+i/5)\,, \quad \theta=z \frac{d}{dz}\, .
\end{equation}

The exponents satisfy
\begin{center}
\begin{tabular}{l | l  l  l}
d & 0 & 1 & $\infty$ \\ \hline
$\mu_1^d$ & 0 & 0 & $\frac{1}{5}$ \\
$\mu_2^d$ & 0 & 1 & $\frac{2}{5}$ \\
$\mu_3^d$ & 0 & 1 & $\frac{3}{5}$ \\
$\mu_4^d$ & 0 & 2 & $\frac{4}{5}$ \; .
\end{tabular}
\end{center}
Hence $\hat{E}$ is given by 
\begin{equation*}
\hat{E}=\hat{E}^{3,0}\oplus \hat{E}^{2,1}\oplus \hat{E}^{1,2}\oplus \hat{E}^{0,3}\cong\Oo_{\mathbb{P}^1}\oplus \Oo_{\mathbb{P}^1} \oplus \Oo_{\mathbb{P}^1}(-1)\oplus \Oo_{\mathbb{P}^1}(-1). 
\end{equation*}
The parabolic structure is only non-trivial at $\infty$ with 
\begin{center}
\begin{tabular}{c c c c c c c c l}
$\hat{E}_\infty$ & $=$ & $\hat{E}_\infty(\tfrac{1}{5})$ & $\supsetneq$ & $\hat{E}_\infty(\tfrac{2}{5})$ & $\supsetneq$ & $\hat{E}_\infty(\tfrac{3}{5})$ & $\supsetneq$ $\hat{E}_\infty(\tfrac{4}{5})$ & $\supsetneq$ 0
\end{tabular}
\end{center}
Moreover, we compute
\begin{gather*}
\mathrm{pdeg}(\hat{E}^{3,0})=\frac{1}{5}=-\mathrm{pdeg}(\hat{E}^{0,3}), \\
\mathrm{pdeg}(\hat{E}^{2,1})=\frac{2}{5}=-\mathrm{pdeg}(\hat{E}^{1,2}). 
\end{gather*}
\end{ex}

\begin{ex}[14 VHS of Calabi-Yau type]\label{ex:14VHS}
The previous example generalizes to all of the 14 variations of Hodge structures of Calabi-Yau type with three regular singular points with a maximally unipotent point at $0$ and a conifold point at $1$ (see \cite{DoranMorgan}). 
The exponents in these examples satisfy
\begin{center}
\begin{tabular}{l | l  l  c}
d & 0 & 1 & $\infty$ \\ \hline
$\mu_1^d$ & 0 & 0 & $\mu_1$ \\
$\mu_2^d$ & 0 & 1 & $\mu_2$ \\
$\mu_3^d$ & 0 & 1 & $1-\mu_2$ \\
$\mu_4^d$ & 0 & 2 & $1-\mu_1$
\end{tabular}
\end{center}
for $0 < \mu_1 \leq \mu_2 \leq \frac{1}{2}$ and we immediately obtain 
\begin{equation*}
\hat{E}\cong \mathcal{O}_{\mathbb{P}^1} \oplus \mathcal{O}_{\mathbb{P}^1}\oplus \mathcal{O}_{\mathbb{P}^1}(-1) \oplus \mathcal{O}_{\mathbb{P}^1}(-1). 
\end{equation*}
The parabolic structure is determined analogously as before. 
Corollary \ref{cor:conjpdeg} gives
\begin{gather*}
\mathrm{pdeg}(\hat{E}^{3,0})=\mu_1=-\mathrm{pdeg}(\hat{E}^{0,3}), \\
\mathrm{pdeg}(\hat{E}^{2,1})=\mu_2=-\mathrm{pdeg}(\hat{E}^{1,2}), 
\end{gather*}
which in particular reproduces Theorem 6.3 in \cite{EskinEtAl}. 
\end{ex}

\section{Conclusion}

In this work we have linked (parabolic) Higgs bundles and opers to mirror symmetry in non-trivial ways. Along the way, we have shown that opers with a compatible integral structure on any Riemann surface are equivalent to VHS with a generic cyclic vector.
In these cases, we computed the parabolic degrees of the induced parabolic Higgs bundles in terms of the exponents of the corresponding Picard--Fuchs equations which we have worked out explicitly in examples from mirror symmetry. 

In these examples, the $tt^*$ or non-abelian Hodge flat connections are gauge equivalent to the opers determined by the Gau\ss--Manin connections. This is because the corresponding Higgs bundles are fixed points of the $\C^\times$ action $\varphi \to \zeta \varphi$. In all of these cases, the conformal limit $\nabla_{\mathrm{NAH}}^\hbar:=\lim_{R \to 0} \nabla_{\mathrm{NAH}}^{R \hbar, R}$ of \cite{Gaiotto:2009hg,Dumitrescu2016a} is trivial because $\nabla_{\mathrm{NAH}}^{R \hbar, R}$ is independent of $R$. The gauge transformation relating $\nabla_{\mathrm{NAH}}^\hbar$ and the Gau\ss-Manin connection $\nabla_{\mathrm{GM}}^\hbar$ 
gives a new derivation of the differential rings on $\mathcal{B}$ which generalize the Ramanujan relation between quasi-modular forms. 

We expect our analysis of the gauge transformation relating the opers to the non-abelian Hodge flat connection to be useful in the study of exact WKB methods for higher order differential operators as was done e.~g.~in \cite{Hollands:2019wbr}. Moreover, there have been many exciting links between Higgs bundles, opers, exact WKB and the topological recursion, see for example \cite{Dumitrescu2015,Dumitrescu2017}. We expect that the further investigation of the connections to non-abelian Hodge theory, $tt^*$ geometry and mirror symmetry will lead to further exciting insights.  

Finally, let us comment on the role of the base curve $\mathcal{B}$ and potential generalizations. In our work, $\mathcal{B}$ is both the base of the parabolic bundles as well as the base of the $tt^*$-geometry. In the $tt^*$ geometries which we have considered in this work, $\mathcal{B}$ is the complex one-dimensional moduli space of a Calabi--Yau $d$-fold. The techniques which we have employed for obtaining the frames for the non-abelian Hodge ($tt^*$) flat connections, as well as the gauge transformation to the oper can be easily generalized to higher dimensional $\mathcal{B}$, with higher rank Higgs bundles. The differential rings, which we re-derived using the gauge transformation are also known for the Hodge bundles of CY $d-$folds, with $d=1,2,3$, see \cite{Alim:2014vea}. We note that the holomorphic Gau\ss-Manin connnection obtained in these cases would represent a generalization of the notion of oper used in our work.

We would like to further comment on the relation to the works of Gaiotto, Moore and Neitzke (GMN) \cite{Gaiotto:2009hg}, where Hitchin systems and $tt^*$-like equations feature prominently. In the setup of GMN, the base curve $\mathcal{C}$ of the Higgs bundle is a curve associated to a physical $\mathcal{N}=2$ gauge theory in four dimensions, called the UV curve. A covering of the base curve $\mathcal{C}$, gives the IR curve $\Sigma$. The UV curve $\mathcal{C}$ of a given physical theory has a moduli space $\mathcal{B}$, which corresponds to the Coulomb branch of the physical theory. 

To establish the connection to moduli spaces of Calabi-Yau geometries as we have studied, the GMN setting can be understood as being obtained from a field theory limit of 10 dimensional string theories living on the four dimensional space times a non-compact Calabi-Yau manifold. Mirror symmetry in this context refers to the fact that identical four dimensional theories can be obtained from two different string theories considered on mirror families of
non-compact CY manifolds. 
The Seiberg--Witten curves can be understood as a degeneration locus of the non-compact CY manifold in question on the B-side of mirror symmetry. The relevant data on the curve is naturally obtained from the variation of Hodge structure and special geometry of the underlying threefolds, this was first obtained in \cite{Klemm:1996bj}.

In this geometric setting, the UV curves are, however, \emph{not} the complex structure moduli spaces $\mathcal{B}$ of the underlying CY threefolds which were used to obtain them. 
The curves can nevertheless be understood as moduli spaces of objects living on the geometry. 
This interpretation was put forward in the work of Aganagic and Vafa \cite{Aganagic:2000gs}, where the mirror curves are identified with the moduli space of branes ending on points of the curve, giving components of the open string moduli space. Physically the curves are moduli spaces of defects of the theory. The Hitchin systems of GMN thus correspond to $tt^*$-equations on moduli spaces of objects of the underlying geometry as opposed to the $tt^*$-equations attached to the geometry itself.


\begin{appendix}
\section{Quasi modular forms and differential rings}\label{modularappendix}

In this appendix we summarize some basic concepts about modular forms and quasi modular forms, following the exposition of \cite{ASYZ}, and we refer to \cite{Diamond,Zagier} and the references therein for more details on the basic theory.

\subsection{Modular groups and modular curves}
The generator for the group $SL(2,\mathbb{Z})$ are given by:
\begin{equation}
  \label{eq:monodromies}
T =
  \begin{pmatrix}
    1 & 1\\ 0 & 1
  \end{pmatrix}\,,\quad
  S=
  \begin{pmatrix}
    0& -1 \\ 1 & 0\\
  \end{pmatrix}\,,\quad S^{2}=-I\,,\quad (ST)^{3}=-I\,.
\end{equation}
We consider the genus zero congruence subgroups called Hecke subgroups of $\Gamma(1)=PSL(2,\mathbb{Z})=SL(2,\mathbb{Z})
/\{\pm I\}$ given by:
\begin{equation}
\Gamma_{0}(N)=\left\{
\left.
\begin{pmatrix}
a & b  \\
c & d
\end{pmatrix}
\right\vert\, c\equiv 0\,~ \mathrm{mod} \,~ N\right\}< \Gamma(1)
\end{equation}
with $N=2,3,4$. A further subgroup that is considered is the unique normal subgroup in $\Gamma(1)$ of index 2 which is often
denoted $\Gamma_0(1)^*$. We write $N=1^*$ when
listing it together with the groups $\Gamma_0(N)$.

The group $SL(2,\mathbb{Z})$ acts on the upper half plane $\mathcal{H}
= \{ \tau \in \mathbb{C} |\, \mathrm{Im} \tau > 0 \}$ by fractional linear
transformations:
$$\tau \mapsto \gamma\tau=\frac{a\tau+b}{c\tau+d}\quad \mbox{for} \quad
\gamma=\begin{pmatrix} a&b\\c&d\end{pmatrix} \in
SL(2,\mathbb{Z})\,.$$ The quotient space $Y_{0}(N)=
\Gamma_{0}(N)\backslash \mathcal{H}$ is a non-compact orbifold with
certain punctures corresponding to the cusps and orbifold points
corresponding to the elliptic points of the group $\Gamma_{0}(N)$.
By filling the punctures, one then gets a compact orbifold
$X_{0}(N)=\overline{Y_{0}(N)}=\Gamma_{0}(N)\backslash
\mathcal{H}^{*}$ where $\mathcal{H}^* = \mathcal{H} \cup \{i\infty\}
\cup \mathbb{Q}$. The orbifold $X_0(N)$ can be equipped with the
structure of a Riemann surface. The signature for the group
$\Gamma_{0}(N)$ and the two orbifolds $Y_{0}(N),X_{0}(N)$ could be
represented by $\{p,\mu;\nu_{2},\nu_{3},\nu_{\infty}\}$, where $p$
is the genus of $X_{0}(N)$, $\mu$ is the index of $\Gamma_{0}(N)$ in
$\Gamma(1)$, and $\nu_{i}$ are the numbers of
$\Gamma_{0}(N)$-equivalent elliptic fixed points or parabolic fixed
points of order $i$. The signatures for the groups $\Gamma_{0}(N)$,
$N=1^*,2,3,4$ are listed in the following table (see
e.g.~\cite{Rankin}):
\begin{equation}\label{signature}
  \begin{array}[h]{|c|c|c|c|c|c|}
    \hline
    N & \nu_2 & \nu_3 & \nu_\infty & \mu & p\\
    \hline
    1^* & 0 & 1 & 2 & 2 & 0 \\
    2 & 1 & 0 & 2 & 3 & 0 \\
    3 & 0 & 1 & 2 & 4 & 0\\
    4 & 0 & 0 & 3 & 6 & 0\\
   \hline
  \end{array}
\end{equation}

The space $X_{0}(N)$ is called a modular curve and is the moduli space
of pairs $(E,C)$, where $E$ is an elliptic curve and $C$ is a cyclic
subgroup of order $N$ of the torsion subgroup
$E_{N} \cong \mathbb{Z}_{N}^{2}$. It classifies each cyclic $N$-isogeny $\phi: E\rightarrow E/C$ up to isomorphism, see for example Refs.~\cite{Diamond,Husemoller} for more details.

In the following, we will denote by $\Gamma$ a general subgroup of finite index in $\Gamma(1)$.

\subsection{Quasi modular forms}

\subsubsection{Modular functions}
\renewcommand{\labelenumi}{(\roman{enumi})}
A (meromorphic) modular function with respect to the a subgroup $\Gamma$ of finite index in $\Gamma(1)$ is a meromorphic function $f : X_{\Gamma}\rightarrow \mathbb{P}^{1}$.
Consider the restriction of $f$ to $Y_{\Gamma}=\Gamma\backslash \mathcal{H}$.
Since the restriction is meromorphic, we know $f$ can be lifted to
a function $f$ on $\mathcal{H}$. Then one gets a function $f: \mathcal{H}\rightarrow \mathbb{P}^{1}$
such that
\begin{enumerate}
\item $\,f(\gamma \tau)=f(\tau), \quad\forall \gamma\in \Gamma\,$.
\item $\,f$ is meromorphic on $\mathcal{H}$.
\item $\,f$ is ``meromorphic at the cusps" in the sense that the function
\begin{equation}
f|_{\gamma}: \tau\mapsto f(\gamma\tau)
\end{equation}
is meromorphic at $\tau=i\infty$ for any $\gamma\in \Gamma(1)$.
\end{enumerate}
The third condition requires more explanation. For any cusp class
$[\sigma] \in \mathcal{H}^*/\Gamma$\footnote{We use the notation
$[\tau]$ to denote the equivalence class of $\tau\in
\mathcal{H}^{*}$ under the group action of $\Gamma$ on
$\mathcal{H}^{*}$.} with respect to the modular group $\Gamma$, one
chooses a representative $\sigma\in \mathbb{Q} \cup \{i\infty\}$.
Then it is easy to see that one can find an element $\gamma\in
\Gamma(1)$ so that $\gamma: i\infty\mapsto \sigma$. Then this
condition means that the function defined by $\tau\mapsto f\circ
\gamma\,(\tau)$ is meromorphic near $\tau=i\infty$ and that the
function $f$ is declared to be ``meromorphic at the cusp $\sigma$"
if this condition is satisfied.

Therefore, equivalently, a (meromorphic) modular function
with respect to the modular group is a meromorphic function $f : \mathcal{H}\rightarrow \mathbb{P}^{1}$ satisfying the above properties on modularity, meromorphicity, and growth condition at the cusps.

\subsubsection{Modular forms}
Similarly, we can define a (meromorphic) modular form of weight $k$ with respect to the group $\Gamma$ to be a (meromorphic) function $f : \mathcal{H}\rightarrow \mathbb{P}^{1}$ satisfying the following conditions:
\begin{enumerate}
\item $\,f(\gamma \tau)=j_{\gamma}(\tau)^{k}f(\tau), \quad \forall \gamma\in \Gamma\,$, where $j$ is called the automorphy factor defined by

$$j: \Gamma \times \mathcal{H}\rightarrow
\mathbb{C},\quad \left(\gamma=\left(
\begin{array}{cc}
a & b  \\
c & d
\end{array}
\right),\tau\right)\mapsto j_{\gamma}(\tau):=(c\tau+d)\,.$$
\item $\,f$ is meromorphic on $\mathcal{H}$.
\item $\,f$ is ``meromorphic at the cusps" in the sense that the function
\begin{equation}
\label{slash}
f|_{\gamma}: \tau\mapsto j_{\gamma}(\tau)^{-k} f(\gamma\tau)
\end{equation}
is meromorphic at $\tau=i\infty$ for any $\gamma\in \Gamma(1)$.
\end{enumerate}

\subsubsection{Quasi modular forms}\label{dfnofquasiform}
A (meromorphic) quasi modular form of weight $k$ with respect to the group $\Gamma$ is a (meromorphic) function $f : \mathcal{H}\rightarrow \mathbb{P}^{1}$ satisfying the following conditions:
\begin{enumerate}
\item $\,$ There exist meromorphic functions $f_{i}, i=0,1,2,3,\dots, k-1$ such that
\begin{equation}
f( \gamma\tau)=j_{\gamma}(\tau)^{k}f(\tau)+\sum_{i=0}^{k-1} c^{k-i}\,j_{\gamma}(\tau)^i  f_{i}(\tau)\,,\quad \forall  \gamma=\left(
\begin{array}{cc}
a & b  \\
c & d
\end{array}
\right)\in \Gamma\,.
\end{equation}
\item $\,f$ is meromorphic on $\mathcal{H}$.
\item $\,f$ is ``meromorphic at the cusps" in the sense that the function
\begin{equation}
f|_{\gamma}: \tau\mapsto j_{\gamma}(\tau)^{-k} f(\gamma\tau)
\end{equation}
is meromorphic at $\tau=i\infty$ for any $\gamma\in \Gamma(1)$.
\end{enumerate}

We proceed by introducing the modular forms which are used in our paper, starting with the Jacobi theta functions with characteristics $(a,b)$ defined by:
\begin{equation}
\vartheta \left[\!\!\! \begin{array}{c}a\\ b\end{array}\!\!\!\right](z,\tau)=\sum_{n\in \mathbbm{Z}}  q^{{\frac{1}{2}}(n+a)^2} e^{2\pi i (n+a)(z+b)} \,.
\end{equation}

for special $(a,b)$ these are denoted by:

\begin{align}
\theta_1(z,\tau)&=\vartheta \left[\!\!\! \begin{array}{c}1/2\\ 1/2\end{array}\!\!\!\right](u,\tau)=\sum_{n\in \mathbbm{Z}+{\frac{1}{2}}}  (-1)^{n}q^{{\frac{1}{2}}n^2}e^{2\pi i n z}\,,\\
\theta_2(z,\tau)&=\vartheta \left[\!\!\! \begin{array}{c}1/2\\ 0\end{array}\!\!\!\right](u,\tau)=\sum_{n\in \mathbbm{Z}+{\frac{1}{2}}}  q^{{\frac{1}{2}}n^2}e^{2\pi i n z}\,,\\
\theta_3(z,\tau)&=\vartheta \left[\!\!\! \begin{array}{c}\,\,\,0\,\,\,\\ \,\,\,0\,\,\,\end{array}\!\!\!\right](u,\tau)=\sum_{n\in \mathbbm{Z}}  q^{{\frac{1}{2}}n^2}e^{2\pi i n z}\,,\\
\theta_4(z,\tau)&=\vartheta \left[\!\!\! \begin{array}{c}0\\ 1/2\end{array}\!\!\!\right](u,\tau)=\sum_{n\in \mathbbm{Z}} (-1)^n q^{{\frac{1}{2}}n^2}e^{2\pi in z}\,.
\end{align}

We further define the following $\theta$--constants:
\begin{equation}
\theta_{2}(\tau)=\theta_2(0,\tau),\quad \theta_{3}(\tau)=\theta_3(0,\tau),\quad \theta_{4}(\tau)=\theta_2(0,\tau)\,.
\end{equation}
The $\eta$--function is defined by
\begin{equation}
\eta(\tau)=q^{\frac{1}{24}}\prod_{n=1}^\infty(1-q^n)\,.
\end{equation}
It transforms according to
\begin{equation}\label{etatrafo}
\eta(\tau+1)=e^{\frac{i\pi}{12}}\eta(\tau),\qquad \eta\left(-\frac{1}{\tau}\right)=\sqrt{\frac{\tau}{i}}\,\eta(\tau)\,.
\end{equation}
The Eisenstein series are defined by
\begin{equation}\label{eisensteinseries}
E_k(\tau)=1-\frac{2k}{B_k}\sum_{n=1}^\infty\frac{n^{k-1}q^n}{1-q^n},
\end{equation}
where $B_k$ denotes the $k$-th Bernoulli number. $E_k$ is a modular form of weight $k$ for $k>2$ and even. The discriminant form and the $j$
invariant are given by
\begin{align}
   \Delta(\tau) &= \frac{1}{1728}\left({E_4}(\tau)^3-{E_6}(\tau)^2\right) = \eta(\tau)^{24},\\
    j(\tau)& = 1728\frac{E_{4}(\tau)^{3}}{ E_{4}(\tau)^3-{E_6}(\tau)^2}\,.
\end{align}

\subsection{Differential ring}\label{app:differential_ring}
The modular forms obey the following differential equations:
\begin{align}
\partial_{\tau}\log \eta(\tau)&=\frac{1} {24}E_{2}(\tau)\,,\\
\partial_{\tau}\log \sqrt{\mathrm{Im}~\tau}|\eta(\tau)|^2&={\frac{1}{24}}\widehat{E_{2}}(\tau,\bar{\tau})\, .
\end{align}
where we denote by $\partial_{\tau} :={\frac{1}{2\pi i}}{\frac{\partial}{\partial \tau}}$, $\widehat{E}_2$ is the non-homolorphic modular completion of the quasi-modular form $E_2$. $E_2,E_4 $ and $E_6$ satisfy the following differential ring:

\begin{equation}
\begin{aligned}
                  \partial_\tau E_{2}&={\frac{1}{12}}(E_{2}^2-E_{4})\,,\\
                  \partial_\tau E_{4}&={\frac{1}{3}}(E_{2}E_{4}-E_{6})\,,\\
\partial_\tau E_{6}&={\frac{1}{2}}(E_{2}E_{6}-E_{4}^{2})\,.
                          \end{aligned}
                          \end{equation}

For the subgroups $\Gamma_0(N)$ we introduce three modular forms $A,B,C$ of weight 1, which are given by:
\begin{equation}
    \hspace{2.5em} \begin{array}{c|ccc}\renewcommand{\arraystretch}{0.5}
        N&A&B&C\\[.2ex]
1^{*}&E_{4}(\tau)^{\frac{1}{4}}&(\frac{E_{4}(\tau)^{\frac{3}{2}}+E_{6}(\tau)}{2})^{\frac{1}{6}}&(\frac{E_{4}(\tau)^{\frac{3}{2}}-E_{6}(\tau)}{2})^{\frac{1}{6}}\\[.5ex]
2&\frac{(2^{6}\eta(2\tau)^{24}+\eta(\tau)^{24} )^{\frac{1}{4}}}{
\eta(\tau)^2\eta(2\tau)^2}&\frac{\eta(\tau)^{4}}{\eta(2\tau)^{2}}&2^{\frac{3}{
2}}\frac{\eta(2\tau)^4}{\eta(\tau)^2}\\[1ex]
3&\frac{(3^{3}\eta(3\tau)^{12}+\eta(\tau)^{12} )^{\frac{1}{3}}}{\eta(\tau)\eta(3\tau)}&\frac{\eta(\tau)^{3}}{\eta(3\tau)}&3\frac{\eta(3\tau)^3}{\eta(\tau)}\\[1ex]
4&\frac{(2^{4}\eta(4\tau)^{8}+\eta(\tau)^{8} )^{\frac{1}{2}}}{\eta(2\tau)^2}=
\frac{\eta(2\tau)^{10}}{\eta(\tau)^{4}\eta(4\tau)^{4}}&\frac{\eta(\tau)^{4}}{ \eta(2\tau)^2}&2^2\frac{\eta(4\tau)^4}{ \eta(2\tau)^2}
\end{array}
\end{equation}

These satisfy by definition
\begin{equation}
A^{r}=B^{r}+C^{r}\,.
\end{equation}

with the following values of $r$:
\begin{equation*}
 \begin{array}{c|ccccc}
N&1^{*}&2&3&4&\\
r&6&4&3&2
\end{array}
\end{equation*}
We introduce the analog of the Eisenstein series $E_{2}$ as a quasi-modular form as follows:
\begin{eqnarray}
E=\partial_\tau \log
B^{r}C^{r}\,.
\end{eqnarray}
The differential ring structure becomes:
\begin{equation}
\begin{aligned}
\partial_\tau A&=\frac{1}{2r}A(E+\frac{C^{r}-B^{r}}{ A^{r-2}})\,,\\
\partial_\tau B&=\frac{1}{2r}B(E-A^{2})\,,\\
\partial_\tau C&=\frac{1}{2r}C(E+A^{2})\,,\\
\partial_\tau E&=\frac{1}{2r}(E^{2}-A^{4})\,.
  \end{aligned}
\end{equation}


\section{Parabolic Higgs bundles}
Let $D\subset C$ be an effective divisor in a compact Riemann surface. 
A parabolic Higgs bundle on $(C,D)$ is a pair $(\hat{E},\hat{\varphi})$ consisting of 
\begin{itemize}
    \item 
    a holomorphic vector bundle $\hat{E}$ on $C$. 
    Each fiber $E_d$ over $d\in D$ is endowed with a filtration
    \begin{align}\label{eq:parabolicstr}
 \hat{E}_d = \hat{E}_{d}(\alpha_1) & \supsetneq \hat{E}_{d}(\alpha_2)  \supsetneq \dots \supsetneq  \hat{E}_{d}(\alpha_s)  \supsetneq 0
\end{align}
of subspaces. 
The real numbers $0 \leq \alpha_1 < \alpha_2 < \dots < \alpha_s < 1$ are the corresponding parabolic weights (which, together with $s$, depend on $d$ as well). 
\item A holomorphic section $\hat{\varphi}\in H^0(C,\Omega^1_C(\log D)\otimes \mathrm{End}(\hat{E}))$ such that $\hat{\varphi}_d(\hat{E}_d(\alpha_j))\subset \hat{E}_d(\alpha_j)$ for each $d\in D$ and every $j\in \{ 1, \dots, s=s(d)\}$. 
\end{itemize}
If the Higgs field $\hat{\varphi}$ satisfies $\hat{\varphi}_d(\hat{E}_d (\alpha_j))\subset \hat{E}_d(\alpha_{j+1})$ for each $d$ and $j$, then a parabolic Higgs bundle is called a strongly parabolic Higgs bundle (see \cite{LogaresMartens}). 
Accordingly, what we defined as a parabolic Higgs bundle is sometimes defined as a weakly parabolic Higgs bundle.

\end{appendix}

\newcommand{\etalchar}[1]{$^{#1}$}

\end{document}